\documentclass[10pt,letterpaper]{amsart}
\usepackage[utf8]{inputenc}\usepackage[left=2.8cm,right=2.8cm,top=3.5cm,bottom=3cm]{geometry}
\usepackage{amssymb,latexsym,amsmath,amsthm,amscd}
\usepackage{graphicx}
\usepackage{leftidx}
\usepackage{dsfont}
\usepackage{stmaryrd}
\usepackage{bm}
\usepackage{comment}
\usepackage[pagebackref=false,colorlinks]{hyperref}
\hypersetup{pdffitwindow=true,linkcolor=blue,citecolor=olive,urlcolor=cyan}
\usepackage{hyperref}
\usepackage[all]{xy}
\usepackage{mathrsfs}
\usepackage{color}
\usepackage[dvipsnames]{xcolor}
\definecolor{Blue}{rgb}{0.3,0.3,0.9}

\usepackage[T2A,OT1]{fontenc}
\DeclareSymbolFont{cyrillic}{T2A}{cmr}{m}{n}
\DeclareMathSymbol{\Sha}{\mathalpha}{cyrillic}{216}

\theoremstyle{plain}
\newtheorem{thm}{Theorem}[subsection] 

\usepackage{mathrsfs}

\theoremstyle{definition}

\newtheorem{defn}[thm]{Definition}

\newtheorem{rmk}[thm]{Remark}
\newtheorem*{rmk*}{Remark}

\newtheorem{assumption}{Assumption}[section]

\theoremstyle{definition}

\theoremstyle{plain}
\newtheorem{prop}[thm]{Proposition}
\theoremstyle{plain}
\newtheorem{lem}[thm]{Lemma}
\theoremstyle{plain}
\newtheorem{cor}[thm]{Corollary}
\newtheorem*{thm-intro}{Theorem}
\newtheorem*{cor-intro}{Corollary}

\theoremstyle{remark}

\newtheorem{rem}[thm]{Remark}

\numberwithin{equation}{section}

\newcommand{\cc}{\mathbf{c}}

\newcommand{\ch}{\xi}

\newcommand{\rH}{{\mathrm{H}}}

\newcommand{\pp}{\mathfrak{p}}
\newcommand{\ppbar}{\bar{\mathfrak{p}}}
\newcommand{\fkm}{\mathfrak{m}}

\newcommand{\fkf}{\mathfrak{f}}

\newcommand{\fkl}{\mathfrak{l}}

\newcommand{\fkn}{\mathfrak{n}}

\newcommand{\fa}{\mathfrak{a}}

\newcommand{\relstr}{{{\rm rel},{\rm str}}}

\newcommand{\ord}{{{\rm ord},{\rm ord}}}

\DeclareMathOperator{\GL}{GL}

\DeclareMathOperator{\pa}{\mathfrak{p}}

\DeclareMathOperator{\et}{\acute{e}t}

\DeclareMathOperator{\Ll}{\mathscr{L}}
\DeclareMathOperator{\Ss}{\mathscr{S}}
\DeclareMathOperator{\p}{\mathfrak{P}}

\DeclareMathOperator{\La}{\Lambda}
\DeclareMathOperator{\la}{\lambda}
\DeclareMathOperator{\Ga}{\Gamma}

\DeclareMathAlphabet\mathbfcal{OMS}{cmsy}{b}{n}
\DeclareTextSymbolDefault{\uhorn}{T5}
\DeclareTextSymbolDefault{\ocircumflex}{T5}
\DeclareTextSymbolDefault{\acircumflex}{T5}

\DeclareMathOperator{\ur}{ur} 

\newcommand{\bQ}{\mathbf{Q}}

\newcommand{\bC}{\mathbf{C}}

\def\cL{{\mathcal L}}

\def\cM{\mathcal M}
\def\cO{\mathcal O}

\def\cP{{\mathcal P}}

\def\cQ{\mathcal Q}

\def\Lcal{\mathcal L}

\newcommand{\VQdag}{{\mathbf{V}_{\underline{Q}}^\dagger}}
\newcommand{\Vsdag}{{\mathbf{V}_{Q_1}^\dagger}}
\newcommand{\Vdag}{{\mathbf{V}^\dagger}}
\newcommand{\Adag}{{\mathbf{A}^\dagger}}

\newcommand{\unb}{{\boldsymbol{f}}}

\newcommand{\Q}{\mathbf{Q}}
\newcommand{\Z}{\mathbf{Z}}

\def\makeop#1{\expandafter\def\csname#1\endcsname
	{\mathop{\rm #1}\nolimits}\ignorespaces}
\makeop{Hom}   \makeop{End}   \makeop{Aut}   
\makeop{Pic} \makeop{Gal}       \makeop{Div} \makeop{Lie}
\makeop{PGL}   \makeop{Corr} \makeop{PSL} \makeop{sgn} \makeop{Spf}
\makeop{Tr} \makeop{Nr} \makeop{Fr} \makeop{disc}
\makeop{Proj} \makeop{supp} \makeop{ker}   \makeop{Im} \makeop{dom}
\makeop{coker} \makeop{Stab} \makeop{SO} \makeop{SL} \makeop{SL}
\makeop{Cl}    \makeop{cond} \makeop{Br} \makeop{inv} \makeop{rank}
\makeop{id}    \makeop{Fil} \makeop{Frac}  \makeop{GL} \makeop{SU}
\makeop{Trd}   \makeop{Sp} \makeop{Tr}    \makeop{Trd} \makeop{Res}
\makeop{ind} \makeop{depth} \makeop{Tr} \makeop{st} \makeop{Ad}
\makeop{Int} \makeop{tr}    \makeop{Sym} \makeop{can} \makeop{SO}
\makeop{torsion} \makeop{GSp} \makeop{Tor}\makeop{Ker} \makeop{rec}
\makeop{Ind} \makeop{Coker}
\makeop{vol} \makeop{Ext} \makeop{gr} \makeop{ad}
\makeop{Gr}\makeop{corank} \makeop{Ann}
\makeop{Hol} 
\makeop{Fitt} \makeop{Mp} \makeop{CAP}

\newcommand{\dBr}[1]{\llbracket{#1}\rrbracket}
\newcommand{\bN}{\mathbf{N}}
\newcommand{\bT}{\mathbb{T}}

\newcommand{\cR}{\mathbb{I}}

\newcommand{\bfff}{{\boldsymbol{f}}}
\newcommand{\bff}{{\boldsymbol{f}}}
\newcommand{\bfg}{{\boldsymbol{g}}}
\newcommand{\bfh}{{\boldsymbol{h}}}

\begin{document}
	
\title{Euler systems for self-dual Hecke characters over biquadratic fields with applications}

\author[K.\,T.\,Do]{Kim Tuan Do}
	
\date{\today} 

\address[]{Department of Mathematics, University of California Santa Barbara, CA 93106, USA}
\email{ktdo@ucsb.edu}



\begin{abstract}
We construct a new Euler system attached to a self-dual Hecke character $\chi$ over an imaginary biquadratic field $K_0$. The Euler system is anticyclotomic in the sense of Jetchev--Nekov\'a{\v r}--Skinner. As arithmetic applications, we deduce instances of the Bloch--Kato conjecture in analytic ranks $0$, $1$, and $2$ (the last one for certain CM abelian surfaces). In addition, we establish the full Iwasawa--Greenberg main conjecture for $\chi$ over $K_0$. In this biquadratic setting, our work generalizes Rubin's work on the Iwasawa main conjectures for imaginary quadratic fields using Euler systems of elliptic units, while complementing known results of Hida and Hsieh for general CM fields.

\end{abstract}


\maketitle
\tableofcontents

\addtocontents{toc}{\setcounter{tocdepth}{2}}
\section{Introduction}

\addtocontents{toc}{\setcounter{tocdepth}{-10}}

{
\renewcommand{\thethm}{\Alph{thm}}

Let $K_0/\Q$ be an imaginary biquadratic field and $p=\mathcal{P}_1\mathcal{P}_2\mathcal{P}_3\mathcal{P}_4$  be a prime that splits completely in $K_0$. Then $K_0$ will contain two distinct imaginary quadratic subfields $K_1$, $K_2$ together with one real quadratic subfield $K_3$. Let $\chi$ be a conjugate self-dual Hecke character of $K_0$ (i.e. $\chi\chi^{\cc}=\bN$, where $\chi^\cc(\sigma)=\chi(\cc\sigma\cc^{-1})$, this is called self-dual by Bertolini--Darmon--Prasanna \cite[Def. $3.4$]{bdp3}) with infinity type $(a,-1-a,b,-1-b)$ such that $\chi^{-1}$ can be factored
 \begin{equation}\label{eqn:chi_decomp}
\chi^{-1}=\widetilde{\psi_1{\psi}_1'^{\cc}}\,\widetilde{\psi_2}\,\bN^{(k_1'+k_1+k_2-4)/2}.
    \end{equation}
Here, for $i\in\{1,2\}$, $\psi_i$ is a Hecke character of $K_i$ of infinity type $(1-k_i,0)$ and modulus  $\fkf_i$; $\tilde{\psi}_i$ is the Hecke character of $K_0$, obtained by composing $\mathbb{A}_{K_0}^{\times}\xrightarrow{\mathbb{N}_{K_0/K_i}}\mathbb{A}_{K_i}^{\times}\xrightarrow{\psi_i}\mathbb{C}$. Let $\psi_1'$ be another Hecke character of $K_1$ of infinity type $(1-k_1',0)$  with $k_1'\geq 1$, and modulus  $\fkf_1'$, which also comes with $\widetilde{\psi_1'^{\cc}}$, the restriction to $G_{K_0}$. Note that if this happens, we must have the following infinity type of $\chi$
\[\left(\frac{-k_1'+k_1+k_2}{2},\frac{k_1'-k_1-k_2}{2}+1,\frac{k_1'-k_1+k_2}{2},1-\frac{k_1'-k_1+k_2}{2}\right).\] 
Here, we enlarging $\Q_p$ to $F$ so that the characters $\psi_1,\psi_1',\psi_2$ (or their $p$-adic avatars) take values in $\mathcal{O}$, the ring of integers of $F$.
We then focus on the $1$-dimensional $G_{K_0}={\rm Gal}(\overline{\Q}/K_0)$-representation $V_\chi:=F(\chi)$, containing a lattice $\cO_\chi$.

Throughout the remainder of this section, 
we assume the following hypotheses:
\begin{itemize}
    \item $p>k_1'$.
    \item $p\nmid 6h_{K_0}$, where $h_{K_0}$ is the class number of $K_0$.
    \item $p\nmid N_{K_1/\Q}(\fkf_1'\fkf_1)\cdot{\rm disc}(K_1/\Q) N_{K_2/\Q}(\fkf_2)\cdot{\rm disc}(K_2/\Q)$.
    \item $\chi_{\psi_1'}\epsilon_{K_1}=1$ and $\chi_{\psi_1}\epsilon_{K_1}\chi_{\psi_2}\epsilon_{K_2}=1$. This forces the central character of $\chi$ to be non-trivial. We call this condition \emph{self-dual} following Burungale--He--Tian--Ye \cite[eqn. $1.1$]{BHTY}, see also \cite[Rmk. 3.7]{bdp3}. 
\end{itemize}

For example, $\chi$ is a Hecke character of infinity type $(1,0,0,1)$ conductor $(42+42i)$ defined over $K_0=\Q(\zeta_{12})$ and decomposes as in (\ref{eqn:chi_decomp}) (see Section \ref{Example} for more explicit examples of $\chi$ obtained from LMFDB), where:
\begin{itemize}
\item $\psi_1'$ is a Hecke character of infinity type $(-1,0)$ and conductor $(2\sqrt{-3})$ defined over $K_1=\Q(\sqrt{-3})$ and gives rise to the CM modular form $f\in S_2(\Ga_0(36))$ (this is the modular form 36.2.a.a in LMFDB).
   \item $\psi_1$ is a Hecke character of infinity type $(-2,0)$ and conductor $(6-3\sqrt{-3})$ defined  over $K_1=\Q(\sqrt{-3})$ and gives rise to the CM modular form $g\in S_3(\Ga_1(189),\chi_g)$ (this is the modular form 189.3.d.b in LMFDB). 
   \item $\psi_2$ is a Hecke character of finite order and conductor $(14+14i)$ defined  over $K_2=\Q(i)$ and gives rise to the CM modular form $h\in S_1(\Ga_1(1568),\chi_h)$ (this is the modular form 1568.1.c.a in LMFDB). 
\end{itemize}

For every integral ideal $\mu_3$ of $\cO_{K_3}$, let $K_0[\mu_3]$ be the maximal $p$-subextension of the ring class field of $K_0$ of conductor $\mu_3$. Denote by $\mathcal{N}'$ the subset of the set of squarefree products of primes $\mu_3\subset\cO_{K_3}$, where $m=N_{K_3/\Q}(\mu_3)$ is squarefree, prime to $p$, and split in $K_0$ (see Definition \ref{def:thin-primes}). The following result is the main theorem of the paper, constructing anticyclotomic Euler systems for $V_\chi$.

\begin{thm}[Theorem\,\ref{maintheorem2-chi}]
\label{thmA} 
There exists a collection of Iwasawa cohomology classes
\[
^{\cc}\mathbf{z}_{\mu_3,\chi}\in H^1_{\rm Iw}\bigl(K_0[\mu_3 p^\infty],\cO_{\chi}\bigr),
\] 
indexed by the ideals $\mu_3\in\mathcal{N}'$ with $m=N_{K_3/\Q}(\mu_3)$, 
such that for every prime $\la_3\in \mathcal{N}$ of norm $\ell$, with $(\ell,mp)=1$ we have the norm relation
\begin{equation*}
    \mathrm{Norm}_{K_0[\mu_3]}^{K_0[\mu_3\la_3]}(^{\cc}\mathbf{z}_{\mu_3\la_3,\chi})=P_{\cL_4}(\mathrm{Frob}_{\cL_4})(^{\cc}\mathbf{z}_{\mu_3,\chi}),
\end{equation*}
where $P_{\cL_4}(X)=\det(1-X\cdot \mathrm{Frob}_{\cL_4}\,|\,\cO_{\chi^{-1}}(1))$, and ${\rm Frob}_{\cL_4}$ is the geometric Frobenius.
\end{thm}

\begin{rmk*}
In \cite{JNS}, Jetchev--Nekov{\'a}{\v{r}}--Skinner have developed a theory of `split' anticyclotomic Euler systems attached to conjugate self-dual representations over CM fields, where classes are defined over ring class extensions of CM fields (indexed by ideals of their totally real subfields). Our construction fits within their framework (see some details for the imaginary quadratic case in Alonso--Castella--Rivero\cite[\S{8}]{ACR},\cite[\S{4.3}]{Do-PhD}).
\end{rmk*}

\begin{rmk*}
  In theorems \ref{maintheorem2} and \ref{maintheorem2}, we actually construct two Euler systems, one for $\chi$ and one for $\psi$, where \[\psi^{-1}=\widetilde{\psi_1{\psi}_1'}\,\widetilde{\psi_2}\,\bN^{(k_1'+k_1+k_2-4)/2}.\]
    The purpose of the left superscript notation is to distinguish those two, according to the convention of \cite[\S $3.3$]{CD} (here Castella and the author also constructed two Euler systems, one works for the definite setting and the other works for the indefinite setting). All of our applications hereafter are for $\chi$, since that is enough to reach some CM abelian surfaces.
\end{rmk*}

Since the construction of our Euler systems comes from diagonal cycles, it forces our classes to be in some specific Selmer subgroups of $H^1(K,V_{\chi})$ (for appropriate $K$) with Greenberg local conditions, see Section \ref{subsec:local-condition}. Feeding Theorem~\ref{thmA} to the general Euler system machinery of \cite{JNS}, we can deduce the following cases of the Bloch--Kato conjecture in analytic rank $0$ (for the rank $1$ case, see Remark \ref{rmk:BK_rank_1}).

\begin{thm}[Theorem \ref{thm:BK-def}]
\label{thmB}
Assume further that:
\begin{enumerate}
    \item $k_1\ge k_1'+k_2$,
    \item $\mathrm{gcd}(N_{\psi_1'},N_{\psi_1},N_{\psi_2})=1$,
\end{enumerate}
then we have the following implication
\[
L(\chi,0)\neq 0\quad\Longrightarrow\quad{\rm Sel}_{\rm BK}(K_0,V_{\chi})=0.
\] 
In other words, the Bloch--Kato conjecture holds in analytic rank zero for $V_{\chi}$.
\end{thm}

By applying the formalism of \cite{JNS}, our Euler system yields a divisibility result for the Iwasawa main conjecture over the anticyclotomic plane for the imaginary biquadratic field $K_0$. This generalizes previous results on the anticyclotomic main conjecture for imaginary quadratic fields, which were established by Agboola--Howard\cite{AHsplit} for CM elliptic curves and by Arnold \cite{Arnold} for higher weights CM forms. Furthermore, we deduce the full Iwasawa main conjecture for certain conjugate self-dual characters. In this biquadratic setting, our result complements the seminal work of Hida \cite{hida-coates} and Hsieh \cite{Hsieh-JAMS} on anticyclotomic characters over general CM fields, while directly generalizing Rubin's landmark proof \cite{rubin-IMC} of the Iwasawa main conjectures for imaginary quadratic fields via elliptic units.

\begin{thm}[Theorem \ref{thm:IMC-def}]
    \label{thmC}
    Assume further that the sign of $\chi$ is $+1$ together with condition \ref{Assumption:Hsieh}. Then inside $\Lambda_{K_0}\otimes_{\Z_p}\Q_p$, we have the following equality
\[
{\rm char}_{\Lambda_{K_0}}(X_{0,\varnothing,\varnothing,0}(K_{0,\infty},\cO_{\chi^{-1}}(1)))= (\mathcal{L}_{\cP_1,\cP_4,\chi}^{{\mathrm{K}_0}}).
\]
In other words, the Iwasawa main conjecture holds for $\chi$.
\end{thm}

Finally, we establish a rare instance in which the Bloch--Kato Selmer group has dimension $2$, generalizing the results of Castella \cite{Castella-rank2} from CM elliptic curves to CM abelian surfaces.

\begin{thm}[Theorem \ref{thm:rank-2}]
    \label{thmD}
   For $k_2=1$, $k_1'$ even, and $k_1=k_1'+1$, we assume that
\begin{enumerate}
    \item $\chi^{-1}$, having infinity type $(1,0,0,1)$ is attached to a CM abelian surface,
    \item $L(\chi,0)=0$ with sign of the functional equation being $+1$ (hence vanishing to an even order) 
\end{enumerate}
then we have:
      \[\dim_{F} (\mathrm{Sel}_{\mathrm{BK}}(K_0,V_\chi))=2.\]

\end{thm}

\subsection{Relation to previous works}\label{subsec:previous}


Relying on the Euler system of Beilinson--Flach classes  constructed by Lei--Loeffler--Zerbes \cite{LLZ,LLZ-K} and Kings--Loeffler--Zerbes \cite{KLZ,KLZ-AJM}, Lamplugh  \cite{JackLamplugh} constructed Euler systems for ${\rm Ind}_{K_0}^{K_1}\mathcal{O}(\vartheta\rho)$ over $K_1$ (where $\rho$ is an auxiliary character) and used that to bound the associated Selmer group over $K_0$ via Rubin's machinery \cite{Rubin-ES}. This method has structural limitations. Firstly, $\vartheta$ is not self-dual. Secondly, even though by Shapiro's lemma one has the isomorphism $H^1(K_1,{\rm Ind}_{K_0}^{K_1}\mathcal{O}(\vartheta\rho))\cong H^1(K_0,\mathcal{O}(\vartheta\rho))$, having an Euler system over the LHS meaning varying over class fields of $K_1$. This is arithmetically distinct from having an Euler system over the RHS, meaning varying over class fields of $K_0$, a degree $2$ extension of $K_1$. 

When $\chi$ is a finite order anticyclotomic Hecke character of $K$ (therefore not self-dual), where $K/F$ is a general CM extension, Lee \cite{YSLee} constructed an anticyclotomic Euler system for $\psi$ using congruences and $p$-adic local Langlands, following the method of Urban \cite{Urban-new-cyclotomic-ES}.

The anticyclotomic Euler system over $K_0$ that we construct in this paper is closely related to the anticyclotomic diagonal Euler system over $K_0$ \cite{Do-biquadratic}, and comparable with the Euler system \cite{Do-PhD,CD} over $K$ (an imaginary quadratic field). The construction of the cohomology classes, similar to \cite{Do-PhD,CD,Do-biquadratic}, is based on a generalization of the diagonal cycles pioneered by Gross--Kudla \cite{gross-kudla} and Gross--Schoen \cite{gross-schoen}, and improved recently by Darmon--Rotger and Bertolini--Seveso--Venerucci (see \cite{BDRSV} and \cite{BSV}). The proof of the norm relation relies on crucial maps coming from \cite{LLZ-K}.

Historically speaking, given $3$ newforms $f,g,h$ associated with the diagonal cycles, our program began by requiring the latter two forms to have CM over the same imaginary quadratic field $K$, yielding Euler systems for $V_{f,\chi}$ over $K$ in \cite{Do-PhD,CD}. Subsequently, the author considered these two newforms having CM over different imaginary quadratic fields, which resulted in Euler systems over the imaginary biquadratic field $K_0$ in \cite{Do-biquadratic}. Meanwhile, Castella \cite{Castella-rank2} chose all three forms to have CM over the same imaginary quadratic field $K$. While this did not lead to a new Euler system due to the presence of elliptic units, it sufficed to construct generalized Kato class together with results towards the rank-two Bloch--Kato conjecture for CM elliptic curves. 

The present paper is a natural development of these three approaches, considering the case where the three newforms have CM over two distinct fields. This framework allows us to obtain not only new Euler systems over $K_0$ (made possible by the technique in Loeffler--Zerbes \cite{LZ-Sym-square}), but also results for higher-rank cases of the Bloch--Kato conjecture for CM abelian surfaces.

\subsection{Acknowledgements} This work grew out of joint discussions with Francesc Castella and the author would like to thank him for many fruitful discussions. The author also thanks Ashay Burungale, Rusiru Gambheera, Haruzo Hida, Ming-Lun Hsieh, and Sam Mundy for helpful communications and interest in this work. 
\addtocontents{toc}{\setcounter{tocdepth}{2}}

}

\section{Preliminaries}
\subsection{Galois representations attached to newforms} We establish in this section the Galois representations associated to modular forms. Let $f \in S_k(\Gamma_1(N_f), \chi_f)$ be a normalized newform with $q$-expansion $\sum a_n q^n$. Fix a prime $p \nmid N_f$. Fix embeddings $i_\infty: \overline{\Q} \hookrightarrow \bC$ and $i_p: \overline{\Q} \hookrightarrow \overline{\Q}_p$. Let $L \subset \overline{\Q}$ be the coefficient field of $f$, and $\mathfrak{P}$ the prime of $L$ above $p$ induced by $i_p$. Setting $S = \{\ell \mid pN_f\} \cup \{\infty\}$, the work of Eichler--Shimura and Deligne yield a continuous, absolutely irreducible representation
$\rho_{f,\mathfrak{P}}:G_{\Q,S}\rightarrow \GL_2(L_{\mathfrak{P}})$, 
such that for all primes $\ell\notin S$:
\begin{enumerate}
    \item $\mathrm{tr}(\rho_{f,\mathfrak{P}}(\mathrm{Frob}_\ell))=i_p(a_\ell)$,
    \item $\det(\rho_{f,\mathfrak{P}}(\mathrm{Frob}_\ell))=i_p(\chi_f(\ell)\ell^{k-1})$,
\end{enumerate} 
where $\mathrm{Frob}_\ell$ is the geometric Frobenius.

The geometric realization $V_f$ of $\rho_{f,\mathfrak{P}}$ is given by the specific subspace of the \'etale cohomology group:
\[
H^1_{\et}(Y_1(N_f)_{\overline{\Q}},\Ss_{k-2})\otimes L_{\mathfrak{P}}.
\]
Dually, the contragredient representation $V_f^{\vee}={\rm Hom}(V_f,L_{\mathfrak{P}})$ can be identified with the maximal quotient of \[
H^1_{\et}(Y_1(N_f)_{\overline{\Q}},\Ll_{k-2}(1))\otimes L_{\mathfrak{P}}
\]
on which the dual Hecke operator $T_\ell'$ acts as multiplication by $a_\ell$ for all $\ell\nmid N_fp$ and the diamond operators $\langle d\rangle = \langle d\rangle^{*}$ acts as multiplication by $\chi_f(d)$ for all $d\in (\Z/N_f\Z)^\times$. Let $\cO$ be the ring of integers of $L_\mathfrak{P}$. There exists a $G_{\Q}$-stable $\cO$-lattice $T_f^\vee\subset V_f^\vee$ defined as the image of 
$H^1_{\et}(Y_1(N_f)_{\overline{\Q}},\Ll_{k-2}(1))\otimes\cO$ under the natural map into $V_f^\vee$.

If $f$ is ordinary at $p$ (meaning that $i_p(a_p)\in\cO^\times$), the restriction of the contragredient representation $V_f^{\vee}$ to $G_{\Q_p}$ induces the short exact sequence:
\begin{equation*}
0\rightarrow V_f^{\vee,+}\rightarrow V_f^{\vee}\rightarrow V_f^{\vee,-}\rightarrow 0.
\end{equation*}
 The $G_{\Q_p}$-action on the quotient $V_{f}^{\vee,-}$ is unramified, and the arithmetic Frobenius ${\rm Frob}_p^{-1}$ acts via multiplication by $\alpha_p$, which is the unit root of the Hecke polynomial $x^2-a_px+\chi_f(p)p^{k-1}$.

\subsection{Hecke characters and theta series }\label{theta}

This section is devoted to reviewing our conventions on Hecke characters over $K$, an imaginary quadratic field. Let $p=\pp\ppbar$ be a prime that splits in $K$ with $\pp$ induced by $i_p:\overline{\Q}\hookrightarrow\overline{\Q}_p$. We say that a Hecke character $\psi:\mathbb{A}_K^{\times}/K^{\times}\rightarrow \mathbb{C}^{\times}$ has infinity type $(m,n)\in\Z^2$, if $\psi_{\infty}(x_{\infty})=x_{\infty}^m\bar{x}_{\infty}^n$. Given $g\in G_K$, take $x\in\mathbb{A}_K^{\times}$ such that $\text{rec}_K(x)=g\vert_{K^{\rm ab}}$ where $\text{rec}_K:\mathbb{A}_K^{\times}\rightarrow G_K^{\text{ab}}$ is the geometrically normalized Artin reciprocity map. Define the $p$-adic avatar of $\psi$ as $  \psi_{\p}(g)=i_p\circ i_{\infty}^{-1}(\psi(x)x_{\infty}^{-m}\bar{x}_{\infty}^{-n})x_{\pa}^m x_{\bar{\pa}}^n$.
By abuse of notation, we will also denote this associated $p$-adic avatar by $\psi$ when the context is clear.

Attached to $\psi$, a Hecke character of $K$ of infinity type $(1-k,0)$ with modulus $\fkf$ that takes values in a finite extension $L/K$, is the theta series 
\begin{equation*}
\theta_{\psi}=\sum_{(\mathfrak{a},\fkf)=1}\psi(\mathfrak{a})q^{N_{K/\Q}(\mathfrak{a})} \in S_k(\Gamma_1(N_\psi),\chi_\psi\epsilon_K)
\end{equation*}
where $N_\psi=N_{K/\Q}(\fkf){\rm disc}(K/\Q)$, $\chi_\psi$ is the unique Dirichlet character modulo $N_{K/\Q}(\fkf)$ such that 
$\psi((n))=n\chi_\psi(n)$ for all $n\in\Z$ with $(n,N_{K/\Q}(\fkf))=1$, and $\epsilon_K$ is the quadratic character attached to $K$. 
The cuspform $\theta_\psi$ is new of level $N_\psi=N_{K/\Q}(\mathfrak{f})\cdot \text{disc}(K/\Q)$ if $\fkf$ is the conductor of $\psi$, see Miyake \cite{Mi}. 
The Galois representation associated with $\theta_\psi$ is
\[
\quad V_{\theta_{\psi}}^{\vee}\cong {\rm Ind}^{\Q}_KL_{\mathfrak{P}}(\psi^{-1}),
\]
where $\mathfrak{P}$ is the prime of $L$ above $p$ with respect to $i_p$.

\subsection{Hecke algebras and commutative diagrams of norm maps}\label{LLZrecall} We outline the construction of certain patched CM Hecke modules following \cite{LLZ-K}. Keeping the notation of the previous section, let $\fkn\subset\cO_K$ be an ideal divisible by $\fkf$. Let $K_{\fkn}$ be the ray class field of $K$ with conductor $\fkn$ and let $H_\fkn={\rm Gal}(K_{\fkn}/K)$ be the corresponding ray class group. We write $K(\fkn)$ for the largest $p$-subextension of $K$ contained in $K_{\fkn}$, so that its Galois group ${\rm Gal}(K(\fkn)/K)\cong H_{\fkn}^{(p)}$ is the largest $p$-power quotient of $H_{\fkn}$. For any ideal $\mathfrak{t}$ of $K$ coprime to $\fkn$, let $[\mathfrak{t}]$ denote its class in $H_{\fkn}$. Define the level $N=N_{K/\Q}(\fkn){\rm disc}(K/\Q)$. Let $\mathbb{T}'(N)$ be the subalgebra of ${\rm End}_{\Z}(H^1(Y_1(N)(\bC),\Z))$ generated by $\langle d\rangle'$ and $T_\ell'$ for all primes $\ell$. Under these assumptions, one can prove that:

\begin{prop}[Proposition 3.2.1 in \cite{LLZ-K}]\label{phi_n}
 There exists an algebra homomorphism $\phi_{\fkn}: \mathbb{T}'(N)\rightarrow \mathcal{O}[H_{\fkn}]$ 
defined by 
    \begin{align*}
        \phi_{\fkn}(T_\ell')=\sum_{\substack{\mathfrak{l}\subset\cO_K,\fkl\nmid\mathfrak{n},\\N_{K/\Q}(\fkl)=\ell}}[\fkl]\psi(\fkl),\qquad
        \phi_{\fkn}(\langle d\rangle ')=\chi_\psi(d)\epsilon_{K}(d)[(d)].
    \end{align*}
\end{prop}

For $\fkm=\fkn\fkl$, where $\fkl$ is a prime ideal and $(\fkm,p)=1$, set $M=N_{K/\Q}(\fkm)\text{disc}(K/\Q)$. This yields the following `norm map'
\begin{equation}
\mathcal{N}_{\fkn}^{\fkm}:\mathcal{O}[H_{\fkm}^{(p)}]\otimes_{\mathbb{T}'(M)\otimes\Z_p,\phi_{\fkm}}H^1_{\et}(Y_1(M)_{\overline{\Q}},\Z_p(1))\rightarrow \mathcal{O}[H_{\fkn}^{(p)}]\otimes_{\mathbb{T}'(N)\otimes\Z_p,\phi_{\fkn}}H^1_{\et}(Y_1(N)_{\overline{\Q}},\Z_p(1)).\nonumber
\end{equation}
Following the convention for degeneracy maps in \cite[Sec.,2.1.2]{CD}, this norm map is explicitly defined by distinguishing three cases:
\begin{enumerate}
    \item If $\fkl\mid\fkn$: 
    \begin{equation*}
        \mathcal{N}_{\fkn}^{\fkm}=1\otimes \text{pr}_{1*};
    \end{equation*}
    \item If $\fkl\nmid\fkn$ is split or ramified in $K$ with $N_{K/\Q}(\fkl)=\ell$: 
    \begin{equation*}
        \mathcal{N}_{\fkn}^{\fkm}=1\otimes \text{pr}_{1*}-\frac{\psi(\mathfrak{l)}[\fkl]}{\ell}\otimes \text{pr}_{\ell*};
    \end{equation*}
    \item If $\fkl\nmid\fkn$ is inert in $K$ (where $\fkl=(\ell)$): 
    \begin{equation*}
        \mathcal{N}_{\fkn}^{\fkm}=1\otimes \text{pr}_{1*}-\frac{\psi(\fkl)[\fkl]}{\ell^2}\otimes \text{pr}_{\ell\ell*}.
    \end{equation*}
\end{enumerate}
By composition, the definition of $\mathcal{N}_{\mathfrak{n}}^{\mathfrak{m}}$ naturally extends to any pair of ideals $\fkn\mid \fkm$.

Suppose that $p$ splits in $K$ and $(p,\fkf)=1$. For any ideal $\fkn\subset\cO_K$ divisible by $\mathfrak{f}$ such that $(\fkn,\ppbar)=1$, the maximal ideal of $\mathbb{T}'(N)$ defined as the kernel of the composition
\[
\mathbb{T}'(N)\overset{\phi_\fkn}\longrightarrow\cO[H_\fkn]\overset{{\rm aug}}\longrightarrow\cO\rightarrow\cO/\mathfrak{P},
\]
is $p$-ordinary, non-Eisenstein, and $p$-distinguished. We conclude this section by recalling a crucial result in \cite{LLZ-K}, which will later be used to establish the norm relation of our Euler system.

\begin{thm}[Corollary 5.2.6 in \cite{LLZ-K}]\label{norm1}
Let $\mathcal{A}$ be the set of ideals $\fkm\subset\cO_K$  with $(\fkm,\ppbar)=1$. Put $\mathcal{A}_{\fkf}=\{\fkf\fkm\colon\fkm\in\mathcal{A}\}$. For any  $\fkn\in \mathcal{A}_{\mathfrak{f}}$, there is a $G_{\Q}$-equivariant isomorphism of $\mathcal{O}[H_{\fkn}^{(p)}]$-modules 
\begin{equation*}
    \nu_{\fkn}:\mathcal{O}[H_{\fkn}^{(p)}]\otimes_{\mathbb{T}'(N)\otimes\Z_p,\phi_{\fkn}}H^1_{\et}(Y_1(N)_{\overline{\Q}},\Z_p(1))\overset{\cong}{\longrightarrow}{\rm Ind}^{\Q}_{K(\fkn)}\mathcal{O}(\psi_{\mathfrak{P}}^{-1}).
\end{equation*}
 Furthermore, for any $\fkn,\fkm\in \mathcal{A}_\mathfrak{f}$ with $\fkn\mid\fkm$, the following diagram commutes:
\begin{equation*}
    \xymatrix{
\mathcal{O}[H_{\fkm}^{(p)}]\otimes_{\mathbb{T}'(M)\otimes\Z_p,\phi_{\fkm}}H^1_{\et}(Y_1(M)_{\overline{\Q}},\Z_p(1)) \ar[d]_{\mathcal{N}_{\fkn}^{\fkm}} \ar[r]^(.7){\nu_{\fkm}}_(.7){\cong} & {\rm Ind}^{\Q}_{K(\fkm)}\mathcal{O}(\psi_{\mathfrak{P}}^{-1}) \ar[d]_{\mathrm{Norm}_{\fkn}^{\fkm}} \\
\mathcal{O}[H_{\fkn}^{(p)}]\otimes_{\mathbb{T}'(N)\otimes\Z_p,\phi_{\fkn}}H^1_{\et}(Y_1(N)_{\overline{\Q}},\Z_p(1)) \ar[r]^(.7){ \nu_{\fkn}}_(.7){\cong} & {\rm Ind}^{\Q}_{K(\fkn)}\mathcal{O}(\psi_{\mathfrak{P}}^{-1}), }
\end{equation*}
where $\mathrm{Norm}_{\fkn}^{\fkm}$ is the natural norm map of the induced representations.
\end{thm}

\section{The construction}\label{sec:main-thms}

Let $K_0$ be the compositum of two distinct imaginary quadratic fields, $K_1$ and $K_2$. Given two Hecke characters $\psi_1$ and $\psi_1'$ of $K_1$, alongside a Hecke character $\psi_2$ of $K_2$, we use the results of \cite{Do-biquadratic, CD, BSV, LLZ-K} to construct a family of cohomology classes for the character $\tilde{\psi}_1(\tilde{\psi}_1')^c\tilde{\psi}_2$ defined over ring class field extensions of $K_0$. We prove that these classes satisfy the norm relations of an anticyclotomic Euler system. Following the approach of \cite[\S{3}]{Do-biquadratic}, we first perform this construction and establish the tame norm relations when the associated theta series $(\theta_{\psi_1}, \theta_{\psi_1'}, \theta_{\psi_2})$ have weights $(2,2,2)$. By varying the diagonal cycle classes within Hida families, we then extend the construction to more general weights and establish the wild norm relations.

Throughout this section, we assume the following setup:
\begin{itemize}
    \item Let $K_1/\Q$ and $K_2/\Q$ be distinct imaginary quadratic fields with discriminants $D_1$ and $D_2$, respectively. For $i \in \{1, 2\}$, let $\epsilon_{K_i}$ denote the quadratic character associated with $K_i$.

\item Let $K_0 = K_1 K_2$ be the biquadratic compositum field, and let $K_3$ denote the unique real quadratic field contained within $K_0$.

\item For $i\in \{0,1,2\}$ and an integral ideal $\fkn_i \subset \cO_{K_i}$, we denote the ray class field of $K_i$ with conductor $\fkn_i$ by $^{i}K_{\fkn_i}$, and its ray class group by $^{i}H_{\fkn_i}$. Let $K_i(\fkn_i)$ be the largest $p$-subextension of $K_i$ contained in $^{i}K_{\fkn_i}$, so ${\rm Gal}(K_i(\fkn_i)/K_i)\cong {^{i}H_{\fkn_i}^{(p)}}$ is the maximal $p$-power quotient of $^{i}H_{\fkn_i}$.

\item Let $\psi_1$ and $\psi_1'$ be Hecke characters of $K_1$ of infinity types $(1-k_1,0)$ and $(1-k_1',0)$, with $k_1, k_1' \ge 1$ and moduli $\fkf_1$ and $\fkf_1'$, respectively. Let $\psi_2$ be a Hecke character of $K_2$ of infinity type $(1-k_2,0)$, with $k_2 \ge 1$ and modulus $\fkf_2$.

\item For $\psi \in \{\psi_1, \psi_1', \psi_2\}$ (defined over $K \in \{K_1, K_2\}$ with infinity type $(1-k,0)$ and modulus $\fkf$), we define $\tilde{\psi} = \psi\circ\mathrm{N}_{K_0/K}$ as the base-changed Hecke character of $K_0$ obtained via the composition $\mathbb{A}_{K_0}^{\times} \xrightarrow{\mathrm{N}_{K_0/K}} \mathbb{A}_{K}^{\times} \xrightarrow{\psi} \mathbb{C}^\times$. Let $\theta_{\psi} \in S_k(N_{\psi}, \chi_{\psi}\epsilon_K)$ be the associated theta series, see Section \ref{theta}.

\item We impose the following self-duality requirement:
\begin{equation}\label{eq:self-dual}\chi_{\psi_1'}\epsilon_{K_1} = 1 \quad \text{and} \quad \chi_{\psi_1}\epsilon_{K_1}\chi_{\psi_2}\epsilon_{K_2} = 1.\end{equation}
\end{itemize}

Let $L/K_0$ be a finite extension containing the Fourier coefficients of $\theta_{\psi_1'}$, $\theta_{\psi_1}$, and $\theta_{\psi_2}$. Fix a prime $p\geq 5$ that splits completely in $K_0$ such that $(p,N_{\psi_1'}N_{\psi_1}N_{\psi_2})=1$. Let $\mathfrak{P}\vert\mathfrak{p}$ be the prime of  $L$ determined by a fixed embedding $i_p:\overline{\Q}\hookrightarrow\overline{\Q}_p$. Let $L_\mathfrak{P}$ be the completion of $L$ at $\mathfrak{P}$, and let $\cO$ be its ring of integers. By enlarging $L_{\mathfrak{P}}$ if
necessary, we may assume that the character $\psi_1,\psi_1',\psi_2$ (or, by abuse of notation, their associated $p$-adic avatars) take values in $\cO$.

\subsection{Prime decomposition convention}\label{subsec:prime-decomp} 

Following \cite[\S{3.1}]{Do-biquadratic}, let $\ell$ be a rational prime that splits completely in $K_0$. Fix a prime $\cL_1 \mid \ell$ in $\cO_{K_0}$. Let $\tau_i$ generate $\mathrm{Gal}(K_0/K_i)$ for $i \in \{1,2,3\}$. Note that $\tau_3 = \tau_1\tau_2$ acts as complex conjugation on $K_0/K_3$. The Galois orbit of $\cL_1$ forms the four primes above $\ell$ in $K_0$:
$$\cL_1, \quad \cL_4 = \tau_1\cL_1, \quad \cL_3 = \tau_2\cL_1, \quad \cL_2 = \tau_3\cL_1=\tau_1\tau_2\cL_1.$$

This decomposition entirely determines the prime factorizations of $\ell$ in the subfields $K_1$, $K_2$, and the real quadratic subfield $K_3$:
\begin{itemize}
    \item In $K_1$: $(\ell) = \la_1\bar{\la}_1$, where $\la_1 = \cL_1\cL_4$ and $\bar{\la}_1 = \cL_2\cL_3$.

\item In $K_2$: $(\ell) = \la_2\bar{\la}_2$, where $\la_2 = \cL_1\cL_3$ and $\bar{\la}_2 = \cL_2\cL_4$.

\item In $K_3$: $(\ell) = \la_3\tilde{\la}_3$, where $\la_3 = \cL_3\cL_4$ (so $\la_3 \mid \la_1\la_2$) and $\tilde{\la}_3 = \cL_1\cL_2$.
\end{itemize}

Let $\cL$ be the set of primes $\la_3 \subset \cO_{K_3}$ such that $\ell = \mathrm{N}_{K_3/\Q}(\la_3)$ is prime to $p$ and splits completely in $K_0$. Let $\mathcal{N}$ denote the set of squarefree products of primes in $\cL$ whose absolute norm down to $\Q$ remains squarefree. For any $\mu_3 \in \mathcal{N}$ with norm $m = \mathrm{N}_{K_3/\Q}(\mu_3) = \prod_i \ell_i$, the decomposition described above extends multiplicatively. Specifically, we obtain analogous factorizations
\begin{itemize}
    \item In $K_0$: $(m) = \mathcal{M}_1\mathcal{M}_2\mathcal{M}_3\mathcal{M}_4$    
    \item In $K_1$: $(m) = \mu_1\bar{\mu}_1$, where $\mu_1 = \cM_1\cM_4$ and $\bar{\mu}_1 = \cM_2\cM_3$.

\item In $K_2$: $(\ell) = \mu_2\bar{\mu}_2$, where $\mu_2 = \cM_1\cM_3$ and $\bar{\mu}_2 = \cM_2\cM_4$.

\item In $K_3$: $(\ell) = \mu_3\tilde{\mu}_3$, where $\mu_3 = \cM_3\cM_4$ (so $\mu_3 \mid \mu_1\mu_2$) and $\tilde{\mu}_3 = \cM_1\cM_2$.
\end{itemize}
Here, for every $i$, $\ell_i=\prod_j \cL_{j,i}$, $\cM_j=\prod_{i}\cL_{j,i}$,  for $1\le j\le 4$, and $\mu_k=\prod_i \la_{k,i}$ for $1\le k\le 3$.

\subsection{Construction in weight (2,2,2) and tame norm relation}
\label{subsec:tame}
 Throughout this subsection, we assume the weights $k_1'=k_1=k_2=2$. Let $N=\text{lcm}(N_{\psi_1'},N_{\psi_1},N_{\psi_2})$. Following Section $3.1$ of \cite{CD}, which relies on diagonal classes in the triple product of modular curves \cite[\S{3}]{BSV}, we obtain cohomology classes for every positive integer $m$:
\begin{equation}\label{kappa3}
\mathcal{Z}_m^{(1)}:=\tilde{\kappa}_{m}^{(3)}\in H^1\bigl(\Q,H^1_{\et}(Y_1(N_{\psi_1'})_{\overline{\Q}},\Z_p(1))\otimes H^1_{\et}(Y_1(N_{\psi_1}m)_{\overline{\Q}},\Z_p(1))\otimes H^1_{\et}(Y_1(N_{\psi_2}m)_{\overline{\Q}},\Z_p(1))(-1)\bigr).
\end{equation}
 We then choose a test vector $\breve{f}\in S_2(N)[\theta_{\psi_1'}]$. As noted in op. cit., the maps used to construct $\mathcal{Z}_m^{(1)}$ are compatible with Hecke correspondences. This compatibility allows us to tensor the classes with $\cO$, yielding:
\begin{align*}
\mathcal{Z}_{\mu_3}^{(1)}\in H^1\bigl(\Q,{\rm Ind}^{\Q}_{K_1}\mathcal{O}({\psi_1'^{-1}})&\otimes H^1_{\et}(Y_1(N_{\psi_1}m)_{\overline{\Q}},\Z_p(1))\otimes_{\mathbb{T}'(N_{\psi_1}m)}\mathcal{O}[^{1}H_{\mathfrak{f}_1\mu_1}^{(p)}]\\
 &\quad\otimes H^1_{\et}(Y_1(N_{\psi_2}m)_{\overline{\Q}},\Z_p(1))\otimes_{\mathbb{T}'(N_{\psi_2}m)}\mathcal{O}[^{2}H_{\mathfrak{f}_2\mu_2}^{(p)}](-1)\bigr).
\end{align*}
Here, the test vector $\breve{f}$ determines the  projection $H^1_{\et}(Y_{\overline{\Q}},\Z_p(1))\rightarrow {\rm Ind}^{\Q}_{K_1}\mathcal{O}({\psi_1'^{-1}})$ in the first factor. The tensor products are taken from Proposition\,\ref{phi_n} 
\[
\phi_{\mathfrak{f}_1\mu_1}:\mathbb{T}'(N_{\psi_1}m)\rightarrow\mathcal{O}[^{1}H_{\mathfrak{f}_1\mu_1}^{(p)}],\quad 
\phi_{\mathfrak{f}_2\mu_2}:\mathbb{T}'(N_{\psi_2}m)\rightarrow\mathcal{O}[^{2}H_{\mathfrak{f}_2\mu_2}^{(p)}]
\]
associated with two distinct imaginary quadratic fields $K_1$ and $K_2$, respectively.  Applying the isomorphisms from Proposition \ref{norm1} for both fields:
\begin{align*}
\nu_{\fkf_1\mu_1}:H^1_{\et}(Y_1(N_{\psi_1}m)_{\overline{\Q}},\Z_p(1))\otimes_{\mathbb{T}'(N_{\psi_1}m)}\mathcal{O}[^{1}H_{\mathfrak{f}_1\mu_1}^{(p)}]&\xrightarrow{\sim}{\rm Ind}^{\Q}_{K_1(\fkf_1\mu_1)}\mathcal{O}(\psi_{1}^{-1}),\\
\nu_{\fkf_2\mu_2}:H^1_{\et}(Y_1(N_{\psi_2}m)_{\overline{\Q}},\Z_p(1))\otimes_{\mathbb{T}'(N_{\psi_2}m)}\mathcal{O}[^{2}H_{\mathfrak{f}_2\mu_2}^{(p)}]&\xrightarrow{\sim}{\rm Ind}^{\Q}_{K_2(\fkf_2\mu_2)}\mathcal{O}(\psi_{2}^{-1}),
\end{align*}
we then obtain a group cohomology class in
\[
H^1\bigl(\Q,{\rm Ind}^{\Q}_{K_1}\mathcal{O}({\psi_1'^{-1}})\otimes_{\mathcal{O}}\text{Ind}^{\Q}_{K_1(\fkf_1\mu_1)}\mathcal{O}(\psi_1^{-1})\otimes_{\mathcal{O}}\text{Ind}^{\Q}_{K_2(\fkf_2\mu_2)}\mathcal{O}(\psi_2^{-1})(-1)\bigr).
\]
Via the natural projections $
^{1}H_{\fkf_1\mu_1}\twoheadrightarrow {^{1}H_{\mu_1}}$ and $^{2}H_{\fkf_2\mu_2}\twoheadrightarrow {^{2}H_{\mu_2}}$, this maps to a class
\begin{equation}\label{eq:inf-res-decomp}
\mathcal{Z}_{\mu_3}^{(2)}\in H^1\bigl(\Q,{\rm Ind}^{\Q}_{K_1}\mathcal{O}_{\psi_1'^{-1}}\otimes_{\mathcal{O}}{\rm Ind}^{\Q}_{K_1}\mathcal{O}_{\psi_1^{-1}}[^{1}H_{\mu_1}^{(p)}]\otimes_{\mathcal{O}}{\rm Ind}^{\Q}_{K_2}\mathcal{O}_{\psi_2^{-1}}[^{2}H_{\mu_2}^{(p)}](-1)\bigr).
\end{equation}
Note that the cohomology group can be rewritten as
\[
H^1\bigl(\Q,{\rm Ind}^{\Q}_{K_1}\mathcal{O}_{\psi_1'^{-1}}\otimes_{\mathcal{O}}{\rm Ind}^{\Q}_{K_0}\mathcal{O}_{\tilde{\psi}_1^{-1}\tilde{\psi}_2^{-1}}[{^{1}H_{\mu_1}^{(p)}}\times {^{2}H_{\mu_2}^{(p)}}](-1)\bigr),\]
which, by Shapiro's lemma, defines an element:
\begin{equation}
    \label{eqn:pre-shapiro}\mathcal{Z}_{\mu_3}^{(3)}\in
H^1\bigl(K_0,{\rm Ind}^{\Q}_{K_1}\mathcal{O}_{\psi_1'^{-1}}\otimes_{\mathcal{O}}\mathcal{O}_{\tilde{\psi}_1^{-1}\tilde{\psi}_2^{-1}}[{^{1}H_{\mu_1}^{(p)}}\times {^{2}H_{\mu_2}^{(p)}}](-1)\bigr).
\end{equation}

Given an integral ideal $\fkn$ of $K_3$, let $H[\fkn]^{(p)}={\rm Gal}(K_0[\fkn]/K_0)$ be the maximal $p$-power quotient of the ring class group of $K_0$ of conductor $\fkn$, (so $K_0[\fkn]$ is the maximal $p$-subextension inside the ring class field of $K_0$ of conductor $\fkn$). Note that for the ring class objects of $K_0$, we drop the upper left corner $0$ notation. The following result enables us to project to ring class fields of $K_0$.

\begin{prop}[Theorem 3.2.1 in \cite{Do-biquadratic}]\label{ringclass} Suppose $p\nmid 6h_{K_0}$ and $\mu_3$ is a squarefree ideal of $\cO_{K_3}$ of norm $m$, where $m$ is divisible only by primes split in $K_0$. We have the following short exact sequence:
\[
1\rightarrow(\cO_{K_3}/\mu_3\cO_{K_3})^{\times, (p)}\rightarrow{^{1}H}_{\mu_1}^{(p)}\times {^{2}H}_{\mu_2}^{(p)}\xrightarrow{e\circ w} H[\mu_3]^{(p)}\rightarrow 1.
\]
Furthermore, if $(\ell)=\cL_1\cL_2\cL_3\cL_4$ is a prime splitting in $K_0$ and coprime to $m$, the projection $e\circ w$ maps 
\begin{align*}
      \quad  [\la_1]\times[\bar{\la}_2]&\mapsto \mathrm{Frob}_{\cL_4/\la_3} 
    \end{align*}
where $\mathrm{Frob}_{\cL_4/\la_3} $ is the Frobenius element of $\cL_4$ in $H[\mu_3]^{(p)}$. 
\end{prop}

Pushing $\mathcal{Z}_{\mu_3}^{(3)}$ from (\ref{eqn:pre-shapiro}) through the composition $e\circ w$ yields the class 
\[
\mathcal{Z}_{\mu_3}^{(4)}\in  H^1\bigl(K_0,{\rm Ind}^{\Q}_{K_1}\mathcal{O}_{\psi_1'^{-1}}\otimes_{\mathcal{O}}\mathcal{O}_{\tilde{\psi}_1^{-1}\tilde{\psi}_2^{-1}}[H[\mu_3]^{(p)}](-1)\bigr).
\]
By Shapiro's lemma, this defines
\[
\mathcal{Z}_{\mu_3}^{(5)}\in H^1\bigl(K_0[\mu_3],{\rm Ind}^{\Q}_{K_1}\mathcal{O}_{\psi_1'^{-1}}\otimes_{\cO}\cO(\tilde{\psi}_{1}^{-1}\tilde{\psi}_{2}^{-1})(-1)\bigr).
\]
These classes satisfy the following relations:

\begin{prop}[Theorem 3.2.3 in \cite{Do-biquadratic}]\label{wrongnorm}
The family $\{\mathcal{Z}_{\mu_3}^{(5)}\}$ satisfies the following norm relation:
\begin{align*}
{\rm Norm}_{K_0[\mu_3]}^{K_0[\mu_3\la_3]}(\mathcal{Z}_{\mu_3\la_3}^{(5)})=&(\ell-1)\bigg( (\psi_1'(\la_1)+\psi_1'(\bar{\la}_1))-\frac{\psi_1(\la_1)\psi_2(\bar{\la}_2)}{\ell}([\la_1]\times [\bar{\la}_2])\\&-\frac{\psi_1(\bar{\la}_1)\psi_2(\la_2)}{\ell}([\bar{\la}_1]\times [\la_2])
   +(1-\ell)\frac{\psi_1(\la_1)\psi_2(\la_2)}{\ell^2}([\la_1]\times[\la_2])\bigg)
    (\mathcal{Z}_{\mu_3}^{(5)}).
\end{align*}
\end{prop}

Following \cite[$\S{3.2.1}$]{CD} verbatim by quotienting out the diamond operators action (an idea originating from Darmon--Rotger \cite[$\S{1.4}$]{DR2}), we can remove the $(\ell-1)$ factor  to obtain modified classes
\[
\mathcal{Z}_{\mu_3}^{(6)}\in H^1\bigl(K_0[\mu_3],{\rm Ind}^{\Q}_{K_1}\mathcal{O}_{\psi_1'^{-1}}\otimes_{\cO}\cO(\tilde{\psi}_{1}^{-1}\tilde{\psi}_{2}^{-1})(-1)\bigr).
\]
This cohomology group actually decomposes as 
\begin{equation}\label{eqn:main-decomposition}
    H^1\bigl(K_0[\mu_3],(\cO({\tilde{\psi}_1'^{-1}})\oplus\cO( {\tilde{\psi}_1'^{-\cc}}))\otimes_{\cO}\cO(\tilde{\psi}_{1}^{-1}\tilde{\psi}_{2}^{-1})(-1)\bigr)\simeq H^1\bigl(K_0[\mu_3],\cO_{\psi})\oplus H^1\bigl(K_0[\mu_3],\cO_{\chi})
\end{equation}
where  $\psi_1'^\cc(\sigma)=\psi_1'(\cc\sigma\cc^{-1})$ denotes the composition of $\psi_1'$ with the action of the non-trivial automorphism of $K_1/\Q$. Here,
$$\psi^{-1}=\widetilde{\psi_1{\psi}_1'}\,\widetilde{\psi_2}\,\bN$$
is a self-dual (i.e. $\psi\psi^\cc=\bN$) Hecke character over $K_0$ with infinity type $(2,-1,0,1)$ and
$$\chi^{-1}=\widetilde{\psi_1{\psi}_1'^{\cc}}\,\widetilde{\psi_2}\,\bN$$
is a self-dual Hecke character over $K_0$ with infinity type $(1,0,1,0)$.

By projecting to the first direct summand of equation (\ref{eqn:main-decomposition}),
we obtain classes
\[\mathcal{Z}_{\mu_3,\psi}^{(6)}\in H^1\bigl(K_0[\mu_3],\cO_{\psi}),\]
which satisfy the following analogue of Proposition \ref{wrongnorm} (now stripped of the $(\ell-1)$ factor):
\begin{prop}\label{wrongnorm-psi}
The family $\{\mathcal{Z}_{\mu_3,\psi}^{(6)}\}$ satisfies the following relation:
\begin{align*}
{\rm Norm}_{K_0[\mu_3]}^{K_0[\mu_3\la_3]}(\mathcal{Z}_{\mu_3\la_3,\psi}^{(6)})=\bigg(& (\psi_1'(\la_1)+\psi_1'(\bar{\la}_1))-\frac{\psi_1(\la_1)\psi_2(\bar{\la}_2)}{\ell}([\la_1]\times [\bar{\la}_2])-\frac{\psi_1(\bar{\la}_1)\psi_2(\la_2)}{\ell}([\bar{\la}_1]\times [\la_2])\\
   &+(1-\ell)\frac{\psi_1(\la_1)\psi_2(\la_2)}{\ell^2}([\la_1]\times[\la_2])\bigg)
    (\mathcal{Z}_{\mu_3,\psi}^{(6)}).
\end{align*}
\end{prop}

\begin{rem}\label{assumption-key-psi} Upon restriction to inertia at $\cP_1$ and reduction modulo $p$: $\bar{\psi}|_{I_{\cP_1}}=\omega^2$ and $\bar{\chi}|_{I_{\cP_1}}=\omega$, where $\omega$ is the mod $p$ cyclotomic character having order $(p-1)$. Because their exponents are not congruent modulo $(p-1)$, the characters are distinct. Hence, there exists:
\begin{itemize}
    \item $\tau_{\psi}\in G_{K_0}$ such that $\psi(\tau_{\psi})\equiv1\pmod{\mathfrak{P}}$ while $\chi(\tau_{\psi})\not\equiv 1\pmod{\mathfrak{P}}$,
    \item $\tau_{\chi}\in G_{K_0}$ such that $\chi(\tau_{\chi})\equiv1\pmod{\mathfrak{P}}$ while $\psi(\tau_{\chi})\not\equiv 1\pmod{\mathfrak{P}}$.
\end{itemize}

\end{rem}

\begin{defn}\label{def:thin-primes}
  For $\gamma\in\{\chi,\psi\}$, let $\mathcal{T}'_{\gamma}$ be the set of primes $\ell \nmid pN$ split in $K_0$ such that
\begin{itemize}
    \item $\ell \equiv 1\pmod{p}$,
    \item $\mathrm{Frob}_\ell$ is equal to a conjugate of $\tau_{\gamma}$.
\end{itemize}
Denote by $\mathcal{L}'_{\gamma}$ the set of primes $\la_3\subset\cO_{K_3}$, for which $\ell=N_{K_3/\Q}(\la_3)\in \mathcal{T}'_{\gamma}$ . Let $\mathcal{N}'_{\gamma}$ be the set of squarefree products of primes inside $\mathcal{L}'_{\gamma}$ whose norm down to $\Q$ is still squarefree.
\end{defn}
Then the term in the right-hand side of Proposition~\ref{wrongnorm} can be massaged to agree with the local Euler factor at $\cL_4$ of the Galois representation $[\mathcal{O}_{\psi}]^\vee(1)=\mathcal{O}_{\psi^{-1}}(1)$, giving the correct norm relations:

\begin{thm}\label{maintheorem1}
Suppose $p\nmid 6h_{K_0}$. Let $\mu_3\in\mathcal{N}'_{\psi}$ be a squarefree ideal of $\cO_{K_3}$ of norm $m$. Then there exists a collection of cohomology classes
\[
\mathcal{Z}_{\mu_3,\psi}\in H^1\bigl(K_0[\mu_3],\mathcal{O}_{\psi}\bigr)
\]
such that for every split prime $\la_3\in\mathcal{L}'_{\psi}$ of $\cO_{K_3}$ of norm $\ell$ with $(\ell,m)=1$, we have the norm relation
\begin{equation*}
    \mathrm{Norm}_{K_0[\mu_3]}^{K_0[\mu_3\la_3]}(\mathcal{Z}_{\mu_3\la_3,\psi})=P_{\cL_4}(\mathrm{Frob}_{\cL_4/\la_3})(\mathcal{Z}_{\mu_3,\psi}),
\end{equation*}
where $P_{\cL_4}(X)=\det(1-X\cdot \mathrm{Frob}_{\cL_4/\la_3}\,|\,\mathcal{O}_{\psi^{-1}}(1))$.
\end{thm}

\begin{proof} 
Note that $[\la_1]\times[\bar{\la}_2]$ corresponds to ${\rm Frob}_{\cL_4/\la_3}\in H[\mu_3]^{(p)}$ under the map $e\circ w$ of Proposition \ref{ringclass}. First, we scale the class $\mathcal{Z}_{\mu_3\la_3,\psi}^{(6)}$ by multiplying it with
 $-\psi_1(\la_1)\psi_2(\bar{\la}_2)([\la_1]\times [\bar{\la}_2])$. Looking at modulo $(\ell-1)$, the factor now looks like
\begin{align*}
    -\psi_1(\la_1)\psi_2(\bar{\la}_2){\rm Frob}_{\cL_4/\la_3}\Bigl((\psi_1'(\la_1)+\psi_1'(\bar{\la}_1))-\frac{\psi_1(\la_1)\psi_2(\bar{\la}_2)}{\ell}{\rm Frob}_{\cL_4/\la_3}-\frac{\psi_1(\bar{\la}_1)\psi_2(\la_2)}{\ell}{\rm Frob}_{\cL_4/\la_3}^{-1}\Bigr)
\end{align*}
 Note that:
 \begin{enumerate}
     \item $\psi_1\psi_2((\ell))=\chi_{\psi_1}\chi_{\psi_2}(\ell)\ell^2=\epsilon_{K_1}(\ell)^{-1}\epsilon_{K_2}(\ell)^{-1}\ell^2=\ell^2\equiv 1\pmod{\ell-1}$ since $\ell$ splits in $K_0$.
     \item $\psi_1'(\la_1)=\tilde{\psi}_1'(\cL_4)$, $\psi_1(\la_1)=\tilde{\psi}_1(\cL_4)$, and $\psi_2(\bar{\la}_2)=\tilde{\psi}_2(\cL_4)$.
     \item $\psi_1'(\la_1)\psi_1'(\bar{\la}_1)=\chi_{\psi_1'}(\ell)\ell =\epsilon_{K_1}^{-1}(\ell)\ell\equiv 1\pmod{\ell-1}$.
 \end{enumerate}
 So modulo $(\ell-1)$, the factor is congruent to 
\begin{align*}
    -(\psi^{-1}(\cL_4)&+\chi^{-1}(\cL_4)){\rm Frob}_{\cL_4/\la_3}+\psi^{-1}(\cL_4)\chi^{-1}(\cL_4){\rm Frob}_{\cL_4/\la_3}^2+1 \\
    &= (1-\psi^{-1}(\cL_4){\rm Frob}_{\cL_4/\la_3})(1-\chi^{-1}(\cL_4){\rm Frob}_{\cL_4/\la_3}).
\end{align*}

Following \cite[Thm $5.3.3$]{LZ-Sym-square}, Remark \ref{assumption-key-psi} allows us to eliminate the second factor by multiplying with $(1-\chi^{-1}(\cL_4'){\rm Fr}_{\cL_4'})^{-1}$ for every prime $\ell'|m\ell$. Here, ${\rm Fr}_{\cL_4'}$ is the unique element of $\mathrm{Gal(K_0[\mu_3\la_3]/K_0)}$ that acts trivially on $\mathrm{Gal(K_0[\la_3']/K_0)}$ and maps to $\mathrm{Frob_{\cL_4'/\la_3'}}$ inside $\mathrm{Gal(K_0[\mu_3\la_3/\la_3']/K_0)}$. We explain why this inversion is well-defined. Indeed,  $R=\mathcal{O}[\mathrm{Gal}(K_0(\mu_3\la_3)/K_0)]$ is a local ring with a unique maximal ideal $\pi=\mathfrak{P}+I_G$ (where $I_G$ is the augmentation ideal). Because $\ell' \in \mathcal{T}'_{\psi}$, its Frobenius is conjugate to $\tau_{\psi}$, meaning $\chi(\mathcal{L}_4') \equiv \chi(\tau_{\psi}) \not\equiv 1 \pmod{\mathfrak{P}}$. Therefore, $(1-\chi^{-1}(\cL_4'){\rm Fr}_{\cL_4'})\equiv 1-\chi^{-1}(\cL_4')\not\equiv 0 \pmod{\pi}$.  

An important thing to note here is that doing all of this will not affect the base class besides flipping it sign. Together with Lemma 9.6.1 and 9.6.3 in Rubin \cite{Rubin-ES}, the result immediately follows.
\end{proof}


We can project to the second summand of equation (\ref{eqn:main-decomposition}), remove the $\psi$-part as in the proof of Theorem \ref{maintheorem1}
and obtain the analogous result for $\chi$:
\begin{thm}\label{maintheorem1-chi}
Suppose $p\nmid 6h_{K_0}$. Let $\mu_3\in\mathcal{N}'_{\chi}$ be a squarefree ideal of $\cO_{K_3}$ of norm $m$. Then there exists a collection of cohomology classes
\[
^{\cc}\mathcal{Z}_{\mu_3,\chi}\in H^1\bigl(K_0[\mu_3],\mathcal{O}_{\chi}\bigr)
\]
such that for every split prime $\la_3\in\mathcal{L}'_{\chi}$ of $\cO_{K_3}$ of norm $\ell$ with $(\ell,m)=1$, we have the norm relation
\begin{equation*}
    \mathrm{Norm}_{K_0[\mu_3]}^{K_0[\mu_3\la_3]}(^{\cc}\mathcal{Z}_{\mu_3\la_3,\chi})=P_{\cL_4}(\mathrm{Frob}_{\cL_4/\la_3})(^{\cc}\mathcal{Z}_{\mu_3,\chi}),
\end{equation*}
where $P_{\cL_4}(X)=\det(1-X\cdot \mathrm{Frob}_{\cL_4/\la_3}\,|\,\mathcal{O}_{\chi^{-1}}(1))$.
\end{thm}

\subsection{Construction for general weights and wild norm relations}\label{wildnormsetup}

We now extend the above construction to other weights $(k_1',k_1,k_2)\in\mathbb{Z}_{ \ge 1}^3$ following \cite[Sec.\,3.3]{Do-biquadratic} and \cite[Sec.\,3.4]{CD}. Ultimately, we will demonstrate that these constructed cohomology classes satisfy the wild norm relations for the anticyclotomic $\Z_p^2$-extension of $K_0$. 

We assume throughout that $p \nmid 6h_{K_0}$ and that $p$ splits completely in $K_0$. Specifically, we have
\begin{equation*}
    (p)=\cP_1\cP_2\cP_3\cP_4 \text{ in } K_0,\quad\textrm{$(p)=\pp_1\ppbar_1$  in $K_1$,}\quad \textrm{ $(p)=\pp_2\ppbar_2$ in $K_2$,} 
\end{equation*}
where 
\begin{equation*}   \cP_2=\tau_3\cP_1,\;\cP_3=\tau_2\cP_1,\;\cP_4=\tau_1\cP_1.
\end{equation*}
Here, $\cP_1$ the prime of $K_0$ above $p$ induced by our fixed embedding $i_p:\overline{\Q}\hookrightarrow\overline{\Q}_p$. For $i\in\{1,2\}$, $\cP_1$ lies above $\pp_i\subset K_i$. Note that the numbering convention parallels the one established in Section \ref{subsec:prime-decomp}. Let $\Gamma_{\pp_i}^{K_i}$ be the Galois group of the unique $\Z_p$-extension of $K_i$ unramified outside $\pp_i$. There exists a unique Hecke character $\psi_{0,i}$ of $K_i$ of infinity type $(-1,0)$ and conductor $\pp_i$, whose $p$-adic avatar factors through $\Gamma_{\pp_i}^{K_i}$. The character $\psi_i$ fixed at the beginning of this section can be decomposed as 
\[
\psi_i=\ch_i {\psi_{0,i}^{k_i-1}},
\]
where $\ch_i$ is a ray class character of $K$ of conductor dividing $\fkf_i\pp_i$. Because $\Gamma_{\pp_i}^{K_i}$ is a quotient of $^iH_{\fkf_i\pp_i^\infty}$, we can naturally view $\psi_{0,i}$ and $\ch_i$ as characters of $^iH_{\fkf_i\pp_i^\infty}$. Consider the formal $q$-expansion
\[
\boldsymbol{\theta}_{\ch_i}(q)=\sum_{(\mathfrak{a},\fkf_i\pp_i)=1}\ch_i\psi_{0,i}(\mathfrak{a})[\mathfrak{a}]q^{N_{K/\Q}(\mathfrak{a})}\in\Lambda_{\pp_i}[\![q]\!],
\]
where $\Lambda_{\pp_i}=\cO[\![\Gamma_{\pp_i}^{K_i}]\!]\simeq\cO\dBr{\Gamma}$ and $\Gamma=1+p\Z_p$. This is the Hida family passing through $\theta_{\psi_i}$ (the specialization of $\boldsymbol{\theta}_{\ch_i}$ at weight $k_i$ and trivial character recovers the ordinary $p$-stabilization of $\theta_{\psi_i}$). Here, we identify $\Gamma_{\pp_i}^{K_i}$ with $1+p\Z_p$ via the geometrically normalized local Artin map.

Let $\bff$ be the Hida family associated to $\theta_{\psi_1'}$. Let $\bfg=\boldsymbol{\theta}_{\ch_1}$, $\bfh=\boldsymbol{\theta}_{\xi_2}$
be the CM Hida families associated to $\psi_1$ and $\psi_2$, respectively. After choosing a level-$N$ test vector for $\theta_{\psi_1'}$ denoted as $\breve{f}$, following \cite[eqn\,(3-8)]{Do-biquadratic} yields
\begin{equation}\label{eq:big-5}
\boldsymbol{\mathcal{Z}}^{(3)}_{\mu_3}\in 
H^1\bigl(K_3,{\rm Ind}^{\Q}_{K_1}\mathcal{O}({\psi_1'^{-1}})\otimes_\cO{\rm Ind}_{K_0[\mu_3]}^{K_3}\Lambda_\cO(\tilde{\psi}_1^{-1}\tilde{\psi}_2^{-1}\tilde{\kappa}_{\rm ac,1}^{(k_1-2)/2}\tilde{\kappa}_{\rm ac,2}^{(k_2-2)/2}\boldsymbol{\kappa}_{\rm ac}^{-1})(2-(k_1'+k_1+k_2)/2)\bigr).
\end{equation}
 For $i\in\{1,2\}$, we identify  $\Ga_i^{-}=\text{Gal}(K_{i,\infty}^{-}/K_i)$ with the anti-diagonal in $(1+p\Z_p)\times (1+p\Z_p)\simeq \mathcal{O}_{K_i,\mathfrak{p}_i}^{(1)}\times \mathcal{O}_{K_i,\ppbar_i}^{(1)}$  via the geometric normalized Artin map, and define $\kappa_{ac,i}:\Ga_i^- \rightarrow \Z_p^{\times}$ where $((1+p)^{-1/2},(1+p)^{1/2})\mapsto (1+p)$. Note that here we identify the anticyclotomic $\Z_p$ extension $K_{i,\infty}^{-}$ of $K_i$ with $\Gamma_{\pp_i}^{K_i}$.
We then identify the anticyclotomic $\Z_p^2$ extension $\Ga^{-}=\text{Gal}(K_{0,\infty}^{-}/K_0)$ of $K_0$ with $\Gamma_{\pp_1}^{K_1}\times \Gamma_{\pp_2}^{K_2}$ via the diagram (3-9) of loc. cit:
\begin{equation}
    \label{eq:G_K0^-identification}
\xymatrix{
(\mathcal{O}_{K_0,\cP_1}^{(1)}\times \mathcal{O}_{K_0,\cP_2}^{(1)})\times (\mathcal{O}_{K_0,\cP_3}^{(1)}\times \mathcal{O}_{K_0,\cP_4}^{(1)})\ar[d]\ar@{->>}[rr]\ar[d]&&\frac{(1+p\Z_p)\times(1+p\Z_p)}{\mathrm{diag}}\times \frac{(1+p\Z_p)\times(1+p\Z_p)}{\mathrm{diag}}\ar[d] \ar[r]^-{\simeq}&\Z_p^2\ar[d]\\
(\mathcal{O}_{K_1,\mathfrak{p}_1}^{(1)}\times \mathcal{O}_{K_1,\ppbar_1}^{(1)})\times (\mathcal{O}_{K_2,\mathfrak{p}_2}^{(1)}\times \mathcal{O}_{K_2,\ppbar_2}^{(1)})\ar@{->>}[rr]&&\frac{(1+p\Z_p)\times(1+p\Z_p)}{\mathrm{diag}}\times \frac{(1+p\Z_p)\times(1+p\Z_p)}{\mathrm{diag}}\ar[r]^-{\simeq}&\Z_p\times \Z_p
}
\end{equation}
Let $\La^-=\Z_p\dBr{\Ga^-}$ and define further $\bm{\kappa}_{ac}:\Ga^- \rightarrow \La^{\times}$ where
\begin{align*}
    ((1+p)^{-1/2},(1+p)^{1/2},(1+p)^{-1/2},(1+p)^{1/2})\mapsto [(1+p),(1+p)].
\end{align*}
For an $\cO$-lattice $T$  inside a $G_{K_0}$-representation $V$, Shapiro's lemma gives the isomorphism
\[
H^1\bigl(K_0,T\hat\otimes_{\cO}\Lambda_\cO^-(\boldsymbol{\kappa}_{\rm ac}^{-1})\bigr)\simeq H^1_{\rm Iw}(K_0[p^\infty],T),
\]
where $H^1_{\rm Iw}(K_0[p^\infty],T):=\varprojlim_{r,s} H^1(K_0[\mathfrak{p}_3^r\tilde{\mathfrak{p}}_3^s],T)$ with limit under the corestriction maps. Under Shapiro's lemma, the image of $\boldsymbol{\mathcal{Z}}^{(3)}_{\mu_3}$ in (\ref{eq:big-5})  becomes an Iwasawa cohomology class
\begin{equation}\nonumber
\boldsymbol{\mathcal{Z}}_{\mu_3}\in H^1_{\rm Iw}\bigl(K_0[\mu_3 p^\infty],{\rm Ind}^{\Q}_{K_1}\mathcal{O}({\psi_1'^{-1}})\otimes\tilde{\psi}_1^{-1}\tilde{\psi}_2^{-1}\tilde{\kappa}_{\rm ac,1}^{(k_1-2)/2}\tilde{\kappa}_{\rm ac,2}^{(k_2-2)/2}(2-(k_1'+k_1+k_2)/2)\bigr)
\end{equation}
which naturally decomposes into the $\psi$ and $\chi$ components:
\begin{equation}\nonumber
     H^1_{\rm Iw}\bigl(K_0[\mu_3 p^\infty],(\cO({\tilde{\psi}_1'^{-1}})\oplus\cO( {\tilde{\psi}_1'^{-\cc}}))\otimes\tilde{\psi}_1^{-1}\tilde{\psi}_2^{-1}\tilde{\kappa}_{\rm ac,1}^{(k_1-2)/2}\tilde{\kappa}_{\rm ac,2}^{(k_2-2)/2}(2-(k_1'+k_1+k_2)/2)\bigr).
\end{equation}
We can rewrite this decomposition as
\begin{equation}\label{eqn:main-decomposition-2}
    H^1_{\rm Iw}(\cO_{\psi}\otimes\tilde{\kappa}_{\rm ac,1}^{(k_1-2)/2}\tilde{\kappa}_{\rm ac,2}^{(k_2-2)/2})\oplus H^1_{\rm Iw}(\cO_{\chi}\otimes\tilde{\kappa}_{\rm ac,1}^{(k_1-2)/2}\tilde{\kappa}_{\rm ac,2}^{(k_2-2)/2})
\end{equation}
where we have defined the following Hecke characters over $K_0$
\begin{align*}
    \psi^{-1}=\widetilde{\psi_1{\psi}_1'}\,\widetilde{\psi_2}\,\bN^{(k_1'+k_1+k_2-4)/2}\\
    \chi^{-1}=\widetilde{\psi_1{\psi}_1'^{\cc}}\,\widetilde{\psi_2}\,\bN^{(k_1'+k_1+k_2-4)/2}.
\end{align*}
Both characters are self-dual. Here $\psi$ has infinity type (corresponding to the order $(\cP_1,\cP_2,\cP_3,\cP_4)$ or $(1,\tau_3,\tau_2,\tau_1)$):
\[\left(\frac{k_1'+k_1+k_2}{2}-1,2-\frac{k_1'+k_1+k_2}{2},1-\frac{k_1'+k_1-k_2}{2},\frac{k_1'+k_1-k_2}{2}\right),\]
while $\chi$ has infinity type:
\[\left(\frac{-k_1'+k_1+k_2}{2},\frac{k_1'-k_1-k_2}{2}+1,\frac{k_1'-k_1+k_2}{2},1-\frac{k_1'-k_1+k_2}{2}\right).\] 
The following assumption should be compared with an earlier Remark \ref{assumption-key-psi}.

\begin{assumption}\label{assumption-key-psi-infinity} We assume that $p>k_1'$ (ensuring the difference of infinity type at $\cP_1$ is strictly less than $p-1$). Upon restriction to inertia at $\cP_1$ and reduction modulo $p$: $\bar{\psi}|_{I_{\cP_1}}=\omega^{\frac{k_1'+k_1+k_2}{2}-1}$ and $\bar{\chi}|_{I_{\cP_1}}=\omega^{\frac{-k_1'+k_1+k_2}{2}}$, where $\omega$ is the mod $p$ cyclotomic character having order $(p-1)$. Because their exponents are not congruent modulo $p-1$, the characters are distinct. This guarantees the existence of
\begin{itemize}
    \item $\tau_{\psi}\in G_{K_0[p^\infty]}$ such that $\psi(\tau_{\psi})\equiv1\pmod{\mathfrak{P}}$ while $\chi(\tau_{\psi})\not\equiv 1\pmod{\mathfrak{P}}$,
    \item $\tau_{\chi}\in G_{K_0[p^\infty]}$ such that $\chi(\tau_{\chi})\equiv1\pmod{\mathfrak{P}}$ while $\psi(\tau_{\chi})\not\equiv 1\pmod{\mathfrak{P}}$.
\end{itemize}
\end{assumption}
\begin{defn}\label{def:thin-primes-infinity}
For $\alpha\in\{\psi,\chi\}$, let $\mathcal{T}'_\alpha$ be the set of primes $\ell \nmid pN$ that splits in $K_0$ such that:
\begin{itemize}
    \item $\ell \equiv 1\pmod{p}$,
    \item $\mathrm{Frob}_\ell$ is conjugate to $\tau_\alpha$.
\end{itemize}
Denote by $\mathcal{L}'_\alpha$ the set of primes $\la_3\subset\cO_{K_3}$, where $\ell=N_{K_3/\Q}(\la_3)\in \mathcal{T}'_\alpha$ . Let $\mathcal{N}'_\alpha$ be the set of squarefree products of primes inside $\mathcal{L}'_\alpha$ whose norm down to $\Q$ remains squarefree.
\end{defn}

By projecting to the first direct summand of equation (\ref{eqn:main-decomposition-2}), following the proof of Theorem \ref{maintheorem1} and invoking \cite[Thm 6.4.1]{Rubin-ES}, we can obtain a collection of Iwasawa cohomology classes for anticyclotomic twists (to eliminate $\tilde{\kappa}_{\rm ac,1}^{(k_1-2)/2}\tilde{\kappa}_{\rm ac,2}^{(k_2-2)/2}$). This yields a collection of Iwasawa cohomology classes satisfying the wild norm relation stated below:

\begin{thm}\label{maintheorem2}
Suppose $p\nmid 6h_{K_0}$. Let $\mu_3\in \mathcal{N}'_{\psi}$ and denote $m=N_{K_3/\Q}(\mu_3)$. Then there exists a collection of Iwasawa cohomology classes
\[
\mathbf{z}_{\mu_3,\psi}\in H^1_{\rm Iw}\bigl(K_0[\mu_3 p^\infty],\cO_{\psi}\bigr)
\]
such that for every split prime $\la_3\in\cL'_{\psi}$ of $\cO_{K_3}$ of norm $\ell$, where $\ell$ splits in $K_0$, with $(\ell,mp)=1$ we have the norm relation
\begin{equation*}
    \mathrm{Norm}_{K_0[\mu_3]}^{K_0[\mu_3\la_3]}(\mathbf{z}_{\mu_3\la_3,\psi})=P_{\cL_4}(\mathrm{Frob}_{\cL_4})(\mathbf{z}_{\mu_3,\psi}).
\end{equation*}
Here, $P_{\cL_4}(X)=\det(1-X\cdot \mathrm{Frob}_{\cL_4}|\,(\cO_{\psi^{-1}}(1))$.
\end{thm}

By projecting to the second summand (the $\chi$-component) of equation (\ref{eqn:main-decomposition}),
we obtain an analogous result:

\begin{thm}\label{maintheorem2-chi}
Suppose $p\nmid 6h_{K_0}$. Let $\mu_3\in \mathcal{N}'_{\chi}$ and denote $m=N_{K_3/\Q}(\mu_3)$. Then there exists a collection of Iwasawa cohomology classes
\[
^{\cc}\mathbf{z}_{\mu_3,\chi}\in H^1_{\rm Iw}\bigl(K_0[\mu_3 p^\infty],\cO_{\chi}\bigr)
\]
such that for every split prime $\la_3\in\cL'_{\chi}$ of $\cO_{K_3}$ of norm $\ell$, where $\ell$ splits in $K_0$, with $(\ell,mp)=1$ we have the norm relation
\begin{equation*}
    \mathrm{Norm}_{K_0[\mu_3]}^{K_0[\mu_3\la_3]}(^{\cc}\mathbf{z}_{\mu_3\la_3,\chi})=P_{\cL_4}(\mathrm{Frob}_{\cL_4})(^{\cc}\mathbf{z}_{\mu_3,\chi}).
\end{equation*}
Here, $P_{\cL_4}(X)=\det(1-X\cdot \mathrm{Frob}_{\cL_4}|\,(\cO_{\chi^{-1}}(1))$.
\end{thm}

\section{Selmer groups }

In this section, we show that the classes constructed in Theorem \ref{maintheorem2} land in certain Selmer groups defined by Greenberg \cite{Greenberg55}, which we will recall below.

\subsection{Greenberg and Bloch--Kato Selmer groups}

\begin{defn} Let $\mathbb{V}$ be a $G_{K_0}$-representation, containing a stable lattice $\mathbb{T}$. 
For each prime $\cP\in\{\cP_1,\cP_2,\cP_3,\cP_4\}$ of $K_0$ above $p$, we fix a $G_{K_{0,\cP}}$-stable subspace $\mathscr{F}_\cP^+(\mathbb{V})\subset \mathbb{V}$ and define  $\mathscr{F}_\cP^-(\mathbb{V})=\mathbb{V}/\mathscr{F}_\cP^+(\mathbb{V})$.
Let $L$ be a finite extension of $K_0$. The \emph{Greenberg Selmer group} ${\rm Sel}_{\mathscr{F}}(L,\mathbb{V})$ attached to $\mathscr{F}=\{\mathscr{F}_{\cP}^+(\mathbb{V})\}_{\cP\vert p}$ is defined by 
\begin{equation}\label{eqn:Greenberg-Selmer-Defn}
    {\rm Sel}_{\mathscr{F}}(L,\mathbb{V}):={\rm ker}\biggl\{H^1(L,\mathbb{V})\rightarrow\prod_{w}\frac{H^1(L_w,\mathbb{V})}{H^1_{\mathscr{F}}(L_w,\mathbb{V})}
    \biggr\},
\end{equation}
where $w$ runs over the finite primes of $L$, and the local conditions are given by
\[
H^1_{\mathscr{F}}(L_w,\mathbb{V})=\begin{cases}
{\rm ker}\bigl\{H^1(L_w,\mathbb{V})\rightarrow H^1(L_{w}^{\rm ur},\mathbb{V})\bigr\} & \textrm {if $w\nmid p$},  \\[0.2em]
{\rm ker}\bigl\{H^1(L_w,\mathbb{V})\rightarrow H^1(L_{w},\mathscr{F}_{\cP}^-(\mathbb{V}))\bigr\} & \textrm{if $w\mid \cP\mid p$}.   
\end{cases}
\]
We fix a lattice $\mathbb{T}\subset \mathbb{V}$. Let $H^1_{\mathscr{F}}(L_w,\mathbb{T})$ be the inverse image of $H^1_{\mathscr{F}}(L_w,\mathbb{V})$ under the natural map $$H^1(L_w,\mathbb{T})\rightarrow H^1(L_w,\mathbb{V}).$$ 
This then defines ${\rm Sel}_{\mathscr{F}}(L,\mathbb{T})$ as in (\ref{eqn:Greenberg-Selmer-Defn}). For any $\Z_p^2$-extension $L_\infty=\bigcup_{r,s}L_{r,s}$ of $L$, we put
\[
{\rm Sel}_{\mathscr{F}}(L_\infty,\mathbb{T}):=\varprojlim_{r,s}{\rm Sel}_{\mathscr{F}}(L_{r,s},\mathbb{T}),
\]
where the inverse limit is taken with respect to the corestriction map. We also put ${\rm Sel}_{\mathscr{F}}(L_\infty,\mathbb{V}):=
{\rm Sel}_{\mathscr{F}}(L_\infty,\mathbb{T})\otimes_{\Z_p}\Q_p$. Note that this group is independent of the chosen lattice $\mathbb{T}$.
\end{defn}

Fix a choice of Galois stable subgroups $\mathscr{F}=\{\mathscr{F}_{\cP}^+(\mathbb{V})\}_{\cP\vert p}$ and let $$\mathbb{A}:={\rm Hom}_{\Z_p}(\mathbb{T},\mu_{p^\infty}).$$ 
Define the associated \emph{dual Selmer group} ${\rm Sel}_{\mathscr{F}^*}(L,\mathbb{A})$ by
\[
{\rm Sel}_{\mathscr{F}^*}(L,\mathbb{A}):={\rm ker}\biggl\{H^1(L,\mathbb{A})\rightarrow\prod_{w}\frac{H^1(L_w,\mathbb{A})}{H^1_{\mathscr{F}^*}(L_w,\mathbb{A})}\biggr\},
\]
where $H^1_{\mathscr{F}^*}(L_w,\mathbb{A})$ is the orthogonal complement of $H^1_{\mathscr{F}}(L_w,\mathbb{T})$ under local Tate pairing
\[
H^1(L_w,\mathbb{T})\times H^1(L_w,\mathbb{A})\rightarrow\Q_p/\Z_p.
\]

\begin{defn}
We also define the \emph{Bloch-Kato Selmer group} ${\rm Sel}_{\rm BK}(L,\mathbb{V})$ following \cite{BK}:
\[
{\rm Sel}_{\rm BK}(L,\mathbb{V}):={\rm ker}\biggl\{H^1(L,\mathbb{V})\rightarrow\prod_w\frac{H^1(L_w,\mathbb{V})}{H^1_f(L_w,\mathbb{V})}\biggr\},
\]
where the local conditions are given by
\[
H^1_{f}(L_w,\mathbb{V})=
{\rm ker}\bigl\{H^1(L_w,\mathbb{V})\rightarrow H^1(L_{w}^{\rm ur},\mathbb{V})\bigr\},\]
at primes $w\nmid p$, and the crystalline condition at primes $w\mid p$:
\[H^1_{f}(L_w,\mathbb{V})={\rm ker}\bigl\{H^1(L_w,\mathbb{V})\rightarrow H^1(L_{w},\mathbb{V}\otimes\mathbf{B}_{\rm cris})\bigr\}  
\]
with $\mathbf{B}_{\rm cris}$ being Fontaine's crystalline period ring. 
The local conditions $H^1_f(L_w,\mathbb{T})\subset H^1(L_w,\mathbb{T})$ are defined by propagation similarly. 
\end{defn}

\begin{defn}\label{def:rel_str_ord}    Denote by ${\rm Sel}_{\alpha,\beta,\gamma,\delta}(K_0,\mathbb{V})$ the subgroup of $H^1(K_0,\mathbb{V})$ where classes are unramified at all primes $v\nmid p$; and they satisfy the conditions $\alpha$, $\beta$, $\gamma$, $\delta$ at $\cP_1$, $\cP_2$, $\cP_3$, $\cP_4$ respectively, where $\alpha,\beta,\gamma,\delta \in\{\varnothing,0\}$. Here, at prime $\cP\in\{\cP_1,\cP_2,\cP_3,\cP_4\}$
\begin{itemize}
    \item the condition $\varnothing$ means $\mathscr{F}_\cP^+(\mathbb{V})=\mathbb{V}$ (the \emph{relaxed} condition in \cite{Do-biquadratic}),
    \item the condition $0$ means $\mathscr{F}_\cP^+(\mathbb{V})=0$ (the \emph{strict} condition in \cite{Do-biquadratic}).
\end{itemize}

\end{defn}

\subsection{Selmer groups for Hecke characters}
Following Section \ref{sec:main-thms}, we focus on the one-dimensional $G_{K_0}$ representation $V_{\chi}$ and $V_{\psi}$.
Here, they are self-dual Hecke characters of $K_0$, where $\psi$ has infinity type (corresponding to the order $(\cP_1,\cP_2,\cP_3,\cP_4)$ or $(1,\tau_3,\tau_2,\tau_1)$):
\[\left(\frac{k_1'+k_1+k_2}{2}-1,2-\frac{k_1'+k_1+k_2}{2},1-\frac{k_1'+k_1-k_2}{2},\frac{k_1'+k_1-k_2}{2}\right),\]
while $\chi$ has infinity type:
\[\left(\frac{-k_1'+k_1+k_2}{2},\frac{k_1'-k_1-k_2}{2}+1,\frac{k_1'-k_1+k_2}{2},1-\frac{k_1'-k_1+k_2}{2}\right).\] 

Now, we will compute the explicit local conditions for the Bloch-Kato Selmer groups. We adopt the convention that the $p$-adic cyclotomic character has Hodge--Tate weight $-1$. For example, if $\chi$ has infinity type $(a,1-a,b,1-b)$, then the $p$-adic avatar of $\chi$ has Hodge--Tate weight $-a,a-1,-b,b-1$ at ${\cP_1,\cP_2,\cP_3,\cP_4}$ respectively.

\begin{lem}\label{lem:BK-Gr}
For any finite extension $L$ of $K_0$ we have
\begin{equation}
{\rm Sel}_{\rm BK}(L,V_{\psi})=\begin{cases}
{\rm Sel}_{\varnothing,0,0,\varnothing}(L,V_{\psi}) &\textrm{if $k_1'\geq k_1+k_2$},\\[0.2em]
{\rm Sel}_{\varnothing,0,0,\varnothing}(L,V_{\psi}) &\textrm{if $k_1\ge k_1'+k_2$},\\[0.2em]
{\rm Sel}_{\varnothing,0,\varnothing,0}(L,V_{\psi}) &\textrm{if $k_2\ge k_1+k_1'$},\\
{\rm Sel}_{\varnothing,0,0,\varnothing}(L,V_{\psi}) &\textrm{if otherwise},
\end{cases}
\nonumber
\end{equation}
and 
\begin{equation}
{\rm Sel}_{\rm BK}(L,V_{\chi})=\begin{cases}
{\rm Sel}_{0,\varnothing,\varnothing,0}(L,V_{\chi}) &\textrm{if $k_1'\geq k_1+k_2$},\\[0.2em]
{\rm Sel}_{\varnothing,0,0,\varnothing}(L,V_{\chi}) &\textrm{if $k_1\ge k_1'+k_2$},\\[0.2em]
{\rm Sel}_{\varnothing,0,\varnothing,0}(L,V_{\chi}) &\textrm{if $k_2\ge k_1+k_1'$},\\
{\rm Sel}_{\varnothing,0,\varnothing,0}(L,V_{\chi}) &\textrm{if otherwise}.
\end{cases}
\nonumber
\end{equation}

\end{lem}

\begin{proof}
By the Panchiskin condition (see Loeffler \cite{Loeffler--panchishkin}), for every prime $w\vert \cP\vert p$ of $L/K_0/\Q$  we have 
\[
H^1_f(L_w,\mathbb{V})={\rm im}\bigl\{H^1(L_w,{\rm Fil}_{\cP}^1(\mathbb{V}))\rightarrow H^1(L_w,\mathbb{V})\bigr\},
\]
where there exists a $G_{K_{\cP}}$-stable subspace ${\rm Fil}_{\cP}^1(\mathbb{V})\subset \mathbb{V}$ so that the Hodge--Tate weights of ${\rm Fil}_{\cP}^1(\mathbb{V})$ (resp. $\mathbb{V}/{\rm Fil}_{\cP}^1(\mathbb{V})$) are all $<0$  (resp. $\geq 0$). Computing the Hodge--Tate weights, 
we obtain the necessary equalities.
\end{proof}

We finish this section by noting that it is quite trivial to compute the conditions for the dual Selmer groups in our setting (the dual for the condition $\varnothing$  is the condition $0$ and vice versa). For example, the dual Selmer group of ${\rm Sel}_{\varnothing,0,0,\varnothing}(L,T_{\chi})$ consists of classes that are unramified outside $p$ and have the condition $(0,\varnothing,\varnothing,0)$ at $(\cP_1,\cP_2,\cP_3,\cP_4)$. Compatibly with Definition \ref{def:rel_str_ord}, this can be denoted as ${\rm Sel}_{0,\varnothing,\varnothing,0}(L,A_{\chi})$.

\section{Triple product \texorpdfstring{$p$}{p}-adic \texorpdfstring{$L$}{L}-function and Selmer group}\label{subsec:triple}

Here, we will recall some conventions on Hida families, triple product $p$-adic $L$-function (the $\bfg$-{unbalanced}) and Selmer groups (both balanced and $\bfg$-{unbalanced}) following Hsieh \cite{hsieh-triple}.

\subsection{Hida families} 
\label{subsubsec:hida}
We follow the convention of \cite[\S{3.1}]{hsieh-triple}. Let $\cO$ be the ring of integers of a finite extension of $\Q_p$. Let $\cR$ be a normal domain, finite flat over the Iwasawa algebra
\[
\Lambda:=\cO\dBr{1+p\Z_p}.
\] 
 Let $N$ be a positive integer primes $p$ and $\chi:(\Z/Np\Z)^\times\rightarrow\cO^\times$ be a Dirichlet character. Denote by $S^o(N,\chi,\cR)\subset\cR\dBr{q}$ the space of ordinary $\cR$-adic cusp forms of tame level $N$ and branch character $\chi$.

Let $\mathfrak{X}_\cR^+\subset{\rm Spec}\,\cR(\overline{\Q}_p)$ be the set of {arithmetic points} of $\cR$, which consists of the ring homomorphisms $Q:\cR\rightarrow\overline{\Q}_p$ such that for some $k_Q\in\Z_{\geq 2}$ called the {weight of $Q$} and $\epsilon_Q(z)\in\mu_{p^\infty}$,
$$Q\vert_{1+p\Z_p}:z\mapsto z^{k_Q-1}\epsilon_Q(z).$$   We say that $\boldsymbol{f}=\sum_{n=1}^\infty a_n(\bfff)q^n\in S^o(N,\chi,\cR)$ is a {primitive Hida family} if the specialization $\boldsymbol{f}_Q$ for every $Q\in\mathfrak{X}_\cR^+$ gives the $q$-expansion of an ordinary $p$-stabilised newform of weight  $k_Q$ and tame conductor $N$. Let $\mathfrak{X}_{\cR}^{\rm cls}\subset{\rm Spec}\,\cR(\overline{\Q}_p)$ be the set of ring homomorphisms $Q$ as above with $k_Q\in\Z_{\geq 1}$ such that $\bfff_Q$ is the $q$-expansion of a classical modular form.

Given $\bfff$ a primitive Hida family of tame conductor $N$, one can associate a Galois representation 
\[
\rho_{\bfff}:G_\Q\rightarrow{\rm Aut}_{\cR}(V_\bfff)\simeq{\rm GL}_2(\cR),
\]
where the determinant of $\rho_\bfff$ is $\chi_\cR\cdot\varepsilon_{\rm cyc}$, see \cite[\S{3.2}]{hsieh-triple}. By the result of Wiles \cite[Thm.~2.2.2]{wiles88}, the restriction of $V_\bfff$ to $G_{\Q_p}$ is reducible and one has a short exact sequence
\[
0\rightarrow V_\bfff^+\rightarrow V_\bfff\rightarrow V_\bfff^-\rightarrow 0.
\]
Here the quotient $V_\bfff^-$ is free of rank one over $\cR$, with $G_{\Q_p}$ acting via the unramified character sending an arithmetic Frobenius ${\rm Frob}_p^{-1}$ to $a_p(\bfff)$. Let $\bT(N,\cR)$ be the Hecke algebra acting on $\bigoplus_\chi S^o(N,\chi,\cR)$, where $\chi$ runs over the characters of $(\Z/Np\Z)^\times$. There is a $\cR$-algebra homomorphism attached to $\bfff$
\[
\lambda_{\bfff}:\bT(N,\cR)\rightarrow\cR
\] 
that factors through a local component $\bT_{\fkm}$, where $\fkm$ is the maximal ideal containing $\ker \lambda_{\bfff}$. Following Hida \cite{hida-AJM}, we define the {congruence ideal} $C(\bfff)$ of $\bfff$ by
\[
C(\bfff):=\lambda_{\bfff}({\rm Ann}_{\bT_\fkm}({\rm ker}\,\lambda_\bfff))\subset\cR.
\] 
Under the assumption that the residual representation $\bar{\rho}_{\bfff}$ is absolutely irreducible and $p$-distinguished, it can be shown that $C(\bfff)$ is generated by a nonzero element $\eta_{\bfff}\in\cR$ by \cite{hida-AJM,Wiles}.

\subsection{CM Hida families revisited}
\label{subsec:CM}
We explicitly construct CM Hida families, following the exposition in \cite[\S{8.1}]{hsieh-triple}. Let $K$ be an imaginary quadratic field of discriminant $-D_K<0$, and suppose that $p=\pp\ppbar$ splits in $K$, with $\pp$ the prime of $K$ above $p$ induced by our fixed embedding $\imath_p:\overline{\Q}\rightarrow\overline{\Q}_p$.

Let $K_\infty$ be the $\Z_p^2$-extension of $K$. Let $K({\pp^\infty})$ be the maximal subfield of $K_\infty$ unramified outside $\pp$. Put 
\[
\Gamma_\infty:={\rm Gal}(K_\infty/K)\simeq\Z_p^2,\quad\quad
\Gamma_{\pp}:={\rm Gal}(K({\pp^\infty})/K)\simeq\Z_p.
\]

For every ideal $\mathfrak{c}\subset\cO_K$, recall that $K_\mathfrak{c}$ is the ray class field of $K$ of conductor $\mathfrak{c}$. Using our notation, $K({\pp^\infty})$ is the maximal $\Z_p$-extension of $K$ inside $K_{\pp^\infty}$. Denote by ${\rm Art}_\pp$ the restriction of the Artin map to $K_\pp^\times$, with geometric normalisation. Then ${\rm Art}_\pp$ induces an embedding $1+p\Z_p\rightarrow\Gamma_{\pp}$, where we identified $\Z_p^\times$ and $\cO_{K_\pp}^\times$ via $\iota_p$. Let $\gamma_\pp$ be the image of $1+p$ hence it will be a topological generator of $\Gamma_{\pp}$.  

For each variable $S$ let $\Psi_S:\Gamma_\infty\rightarrow\cO\dBr{S}^\times$ be the universal character 
given by
\begin{equation}\label{eqn:top_gen}
    \Psi_S(\sigma)=(1+S)^{l(\sigma)},
\end{equation}
where $l(\sigma)\in\Z_p$ is such that $\sigma\vert_{K({\pp^\infty})}=\gamma_\pp^{l(\sigma)}$. 
Now assume that $\mathfrak{c}$ is prime to $p$. Given a finite order character $\ch:G_K\rightarrow\cO^\times$ of conductor dividing $\mathfrak{c}$, let
\[
\boldsymbol{\theta}_\ch(S)(q)=\sum_{(\fa,\pp\mathfrak{c})=1}\ch(\sigma_{\fa})\Psi^{-1}_{\frac{1+S}{1+p}-1}(\sigma_\fa)q^{N_{K/\Q}(\fa)}\in\cO\dBr{S}\dBr{q},
\]
where $\sigma_\fa\in{\rm Gal}(K_{\mathfrak{c}\pp^\infty}/K)$ is the Artin symbol of $\fa$. 
Then $\boldsymbol{\theta}_\ch(S)$ is a Hida family  defined over $\cO\dBr{S}$ of tame level $N_{K/\Q}(\mathfrak{c})D_K$ and tame character $(\ch\circ\mathscr{V})\epsilon_{K}\omega^{-1}$, where $\mathscr{V}:G_\Q^{\rm ab}\rightarrow G_K^{\rm ab}$ is the transfer map and $\epsilon_{K}$ is the quadratic character corresponding to $K/\Q$.

\subsection{Triple products of Hida families}\label{subsubsec:triple-hida}

Let
\[
\bff\in S^o(N_f,\chi_f,\cR_f),\quad\bfg\in S^o(N_g,\chi_g,\cR_g),\quad\bfh\in S^o(N_h,\chi_h,\cR_h)
\]
be three primitive Hida families such that 
\begin{equation}\label{eq:a}
\textrm{$\chi_f\chi_g\chi_h=\omega^{2a}$ for some $a\in\Z$,}
\end{equation} 
where $\omega$ is the Teichm\"uller character. Let 
\[
\mathcal{R}=\cR_f\hat\otimes_{\cO}\cR_g\hat\otimes_{\cO}\cR_h
\] 
be a finite extension of the three-variable Iwasawa algebra $\Lambda\hat\otimes_{\cO}\Lambda\hat\otimes_{\cO}\Lambda$. 

Let $\mathfrak{X}_{\mathcal{R}}^+\subset{\rm Spec}\,\mathcal{R}(\overline{\Q}_p)$ be the weight space of $\mathcal{R}$ given by
\[
\mathfrak{X}_{\mathcal{R}}^+:=\left\{\underline{Q}=(Q_1,Q_2,Q_3)\in\mathfrak{X}_{\cR_f}^+\times\mathfrak{X}_{\cR_g}^{\rm cls}\times\mathfrak{X}_{\cR_h}^{\rm cls}\;:\;k_{Q_1}+k_{Q_2}+k_{Q_3}\equiv 0\;({\rm mod}\;2)\right\}.
\] 
One can then partition $\mathfrak{X}_\mathcal{R}^+=\mathfrak{X}_\mathcal{R}^{\rm bal}\sqcup\mathfrak{X}_\mathcal{R}^{\bff}\sqcup\mathfrak{X}_\mathcal{R}^{\bfg}\sqcup\mathfrak{X}_\mathcal{R}^{\bfh}$ as follows:
\begin{enumerate}
    \item the set of \emph{balanced} weights:
    \begin{align*}
\mathfrak{X}_{\mathcal{R}}^{\rm bal}&:=\left\{\underline{Q}\in\mathfrak{X}_{\mathcal{R}}^+\;:\;\textrm{$k_{Q_1}+k_{Q_2}+k_{Q_3}> 2k_{Q_i}$ for all $i\in\{1,2,3\}$}\right\},
\end{align*}
    \item the set of $\bff$-{unbalanced} weights:
    \begin{align*}
        \mathfrak{X}_{\mathcal{R}}^\bff&:=\left\{\underline{Q}\in\mathfrak{X}_{\mathcal{R}}^+\;:\;\textrm{$k_{Q_1}\geq k_{Q_2}+k_{Q_3}$}\right\},
    \end{align*}
    \item the set of $\bfg$-{unbalanced} weights:
    \begin{align*}
        \mathfrak{X}_{\mathcal{R}}^\bfg&:=\left\{\underline{Q}\in\mathfrak{X}_{\mathcal{R}}^+\;:\;\textrm{$k_{Q_2}\geq k_{Q_1}+k_{Q_3}$}\right\},
    \end{align*}
    \item the set of $\bfh$-{unbalanced} weights:
    \begin{align*}
    \mathfrak{X}_{\mathcal{R}}^\bfh&:=\left\{\underline{Q}\in\mathfrak{X}_{\mathcal{R}}^+\;:\;\textrm{$k_{Q_3}\geq k_{Q_1}+k_{Q_2}$}\right\}.
    \end{align*}
\end{enumerate}

Let $\mathbf{V}=V_\bff\hat\otimes_{\cO}V_{\bfg}\hat\otimes_{\cO}V_{\bfh}$ be the triple tensor product Galois representation attached to $(\bff,\bfg,\bfh)$. By (\ref{eq:a}), one can decompose the determinant of $\mathbf{V}$ as $\det\mathbf{V}=\mathcal{X}^2\varepsilon_{\rm cyc}$. Put  
\begin{equation}\label{eqn:V-dagger}
\Vdag:=\mathbf{V}\otimes\mathcal{X}^{-1}.
\end{equation}
This is a self-dual twist of $\mathbf{V}$. For any $\underline{Q}=(Q_1,Q_2,Q_3)\in\mathfrak{X}_{\mathcal{R}}^\bfg$,  denote by $\VQdag$  the corresponding specialization.

For each prime $\ell$, let $\varepsilon_\ell(\VQdag)$ be the epsilon factor attached to the local representation $\VQdag\vert_{G_{\Q_\ell}}$  
(see Tate \cite[p.\,21]{tate-background}). We assume that for some $\underline{Q}\in\mathfrak{X}_{\mathcal{R}}^\bfg$, we have
\begin{equation}\label{eq:+1}
\textrm{ $\varepsilon_\ell(\VQdag)=+1$ for all prime factors $\ell$ of $N_f N_g N_h$.}
\end{equation}
Note that condition (\ref{eq:+1}) is independent of $\underline{Q}$ (see \cite[\S{1.2}]{hsieh-triple}). Furthermore it implies that the sign of the functional equation for the triple product $L$-function (with center at $s=0$) 
\[
L(\VQdag,s)
\] 
is $+1$ (resp. $-1$) for all $\underline{Q}\in\mathfrak{X}_{\mathcal{R}}^\bff\cup\mathfrak{X}_{\mathcal{R}}^\bfg\cup\mathfrak{X}_{\mathcal{R}}^\bfh$ (resp. $\underline{Q}\in\mathfrak{X}_{\mathcal{R}}^{\rm bal}$).

\begin{thm}[Theorem~A in \cite{hsieh-triple}]\label{thm:hsieh-triple}
Let $\bff,\bfg,\bfh$ be three primitive Hida families 
satisfying conditions (\ref{eq:a}) and (\ref{eq:+1}). Assume also that $\gcd(N_f,N_g,N_h)$ is squarefree,
and the residual representation $\bar{\rho}_{\bfg}$ is absolutely irreducible and $p$-distinguished.
Fix a generator $\eta_{\bfg}$ of the congruence ideal of $\bfg$. Then there exists a unique element 
\[
\mathscr{L}_p^{\bfg}(\bff,\bfg,\bfh)\in\mathcal{R}
\]
such that for all $\underline{Q}=(Q_1,Q_2,Q_3)\in\mathfrak{X}_{\mathcal{R}}^{\bfg}$ of weight $(k_1,k_2,k_3)$ with $\epsilon_{Q_1}=1$ we have
\[
\bigl(\mathscr{L}_p^{\bfg}(\bff,\bfg,\bfh)(\underline{Q})\bigr)^2=\Gamma_{\VQdag}(0)\cdot\frac{L(\VQdag,0)}{(\sqrt{-1})^{2k_2}\cdot\Omega_{\bfg_{Q_1}}^2}\cdot\mathcal{E}_p(\mathscr{F}_p^{\bfg}(\VQdag))\cdot\prod_{q\in\Sigma_{\rm exc}}(1+q^{-1})^2,
\]
where:
\begin{itemize}
\item $\Gamma_{\VQdag}(0)=16(2\pi)^{-2k_2}\Gamma(w_{\underline{Q}})\Gamma(w_{\underline{Q}}+2-k_1-k_3)\Gamma(w_{\underline{Q}}+1-k_1)\Gamma(w_{\underline{Q}}+1-k_3)$, 
\[
\text{and  }\quad w_{\underline{Q}}=(k_1+k_2+k_3-2)/2;
\] 
\item $\Omega_{\bfg_{Q_1}}$ is the Hida canonical period
\[
\Omega_{\bfg_{Q_1}}:=(-2\sqrt{-1})^{k_2+1}\cdot\frac{\Vert\bfg_{Q_1}^\circ\Vert_{\Gamma_0(N_f)}^2}{\eta_{\bfg_{Q_1}}}\cdot\Bigl(1-\frac{\chi_{g}'(p)p^{k_2-1}}{\alpha_{Q_1}^2}\Bigr)\Bigl(1-\frac{\chi_{g}'(p)p^{k_2-2}}{\alpha_{Q_1}^2}\Bigr),
\]
with $\bfg_{Q_1}^\circ\in S_{k_2}(\Gamma_0(N_{g}))$ the newform of conductor $N_g$ associated with $\bfg_{Q_1}$, $\chi_g'$ the prime-to-$p$ part of $\chi_g$, and $\alpha_{Q_1}$ the specialization of $a_p(\bfg)\in\cR_g^\times$ at $Q_1$;
\item $\mathcal{E}_p(\mathscr{F}_p^{\bfg}(\VQdag))$ is the modified $p$-Euler factor
and  $\Sigma_{\rm exc}$ is an explicitly defined subset of the prime factors of $N_f N_g N_h$, \cite[p.~416]{hsieh-triple}.

\end{itemize}

\end{thm}

\begin{rem}
    One can obtain similar $\bff,\bfh$-unbalanced triple product $p$-adic $L$-functions
    \[\mathscr{L}_p^{\bff}(\bff,\bfg,\bfh),\quad\mathscr{L}_p^{\bfh}(\bff,\bfg,\bfh)\in\mathcal{R}\]
    under an appropriate set of assumptions and weights.
\end{rem}

\subsection{Triple product Selmer groups}
\label{subsec:tripleSelmer}

Recall from equation (\ref{eqn:V-dagger}) that $\Vdag=\mathbf{V}\otimes\mathcal{X}^{-1}$ is the rank $8$ self-dual twist of the Galois representation associated to a triple of primitive Hida families $(\bff,\bfg,\bfh)$ 
given (\ref{eq:a}).

\begin{defn}\label{def:local-p}
Let
\[
\mathscr{F}^{\rm bal}_p(\Vdag)
:=\bigl(V_{\bff}\otimes V_{\bfg}^+\otimes V_{\bfh}^++V_{\bff}^+\otimes V_{\bfg}\otimes V_{\bfh}^++V_{\bff}^+\otimes V_{\bfg}^+\otimes V_{\bfh}\bigr)\otimes\mathcal{X}^{-1},
\]
be a rank $4$ submodule of $\Vdag$, and define the {balanced local condition} $\rH^1_{\rm bal}(\Q_p,\Vdag)$ by 
\[
\rH^1_{\rm bal}(\Q_p,\Vdag):={\rm im}\bigl(\rH^1(\Q_p,\mathscr{F}_p^{\rm bal}(\Vdag))\rightarrow\rH^1(\Q_p,\Vdag)\bigr).
\]
Similarly, we define other rank $4$ submodules of $\Vdag$
\begin{align*}
    \mathscr{F}_p^{\unb}(\Vdag)=\bigl(V_{\bff}^+\otimes V_{\bfg}\otimes V_{\bfh}\bigr)\otimes\mathcal{X}^{-1},\\
     \mathscr{F}_p^{\bfg}(\Vdag)=\bigl(V_{\bff}\otimes V_{\bfg}^+\otimes V_{\bfh}\bigr)\otimes\mathcal{X}^{-1},\\
    \mathscr{F}_p^{\bfh}(\Vdag)=\bigl(V_{\bff}\otimes V_{\bfg}\otimes V_{\bfh}^+\bigr)\otimes\mathcal{X}^{-1}.
\end{align*}
 Define the {$\bff$-unbalanced local condition} $\rH^1_{\unb}(\Q_p,\Vdag)$ by
\[
\rH^1_{\bff}(\Q_p,\Vdag):={\rm im}\bigl(\rH^1(\Q_p,\mathscr{F}_p^{\bff}(\Vdag))\rightarrow\rH^1(\Q_p,\Vdag)\bigr),
\]
the {$\bfg$-unbalanced local condition} $\rH^1_{\bfg}(\Q_p,\Vdag)$ by
\[
\rH^1_{\bfg}(\Q_p,\Vdag):={\rm im}\bigl(\rH^1(\Q_p,\mathscr{F}_p^{\bfg}(\Vdag))\rightarrow\rH^1(\Q_p,\Vdag)\bigr),
\]
and the {$\bfh$-unbalanced local condition} $\rH^1_{\bfh}(\Q_p,\Vdag)$ by
\[
\rH^1_{\bfh}(\Q_p,\Vdag):={\rm im}\bigl(\rH^1(\Q_p,\mathscr{F}_p^{\bfh}(\Vdag))\rightarrow\rH^1(\Q_p,\Vdag)\bigr).
\]
\end{defn}

Since the maps appearing in these definitions are injective, for $\star\in\{{\rm bal},\bff,\bfg,\bfh\}$, we can identify $\rH^1_{\star}(\Q_p,\Vdag)$ with $\rH^1(\Q_p,\mathscr{F}_p^\star(\Vdag))$.

\begin{defn}\label{def:Sel-bu}
Let $\star\in\{{\rm bal},\bff,\bfg,\bfh\}$. Define the Selmer group ${\rm Sel}^\star(\Q,\Vdag)$ by
\[
{\rm Sel}^\star(\Q,\Vdag):=\ker\biggl\{\rH^1(\Q,\Vdag)\rightarrow\frac{\rH^1(\Q_p,\Vdag)}{\rH^1_\star(\Q_p,\Vdag)}\times\prod_{v\neq p}\rH^1(\Q_v^{\rm nr},\Vdag)\biggr\}.
\]
We call ${\rm Sel}^{\rm bal}(\Q,\Vdag)$ the {balanced} Selmer group. For $\mathbf{t}\in\{\bff,\bfg,\bfh\}$, we call ${\rm Sel}^{\mathbf{t}}(\Q,\Vdag)$ the {$\mathbf{t}$-unbalanced} Selmer group.
\end{defn}
\begin{defn}
Let $\Adag={\rm Hom}_{\Z_p}(\Vdag,\mu_{p^\infty})$ and let $\star\in\{{\rm bal},\bff,\bfg,\bfh\}$. Define $\rH^1_{\star}(\Q_p,\Adag)\subset\rH^1(\Q_p,\Adag)$ to be the orthogonal complement of $\rH^1_\star(\Q_p,\Vdag)$ under the local Tate pairing
\[
\rH^1(\Q_p,\Vdag)\times\rH^1(\Q_p,\Adag)\rightarrow\Q_p/\Z_p.
\]
Similarly as above, we then define the balanced, the $\bff$-unbalanced, the $\bfg$-unbalanced and the $\bfh$-unbalanced  Selmer groups with coefficients in $\Adag$ by
\[
{\rm Sel}^\star(\Q,\Adag):=\ker\biggl\{\rH^1(\Q,\Adag)\rightarrow\frac{\rH^1(\Q_p,\Adag)}{\rH_\star^1(\Q_p,\Adag)}\times\prod_{v\neq p}\rH^1(\Q_v^{\rm nr},\Adag)\biggr\}.
\]
Let $$X^\star(\Q,\Adag)={\rm Hom}_{\Z_p}({\rm Sel}^\star(\Q,\Adag),\Q_p/\Z_p)$$ be the Pontryagin dual of ${\rm Sel}^\star(\Q,\Adag)$.
 
\end{defn}

\section{Arithmetic applications}
Finally, we obtain some arithmetic results from our constructed Euler system through identifying our classes as an anticyclotomic Euler system in the sense of Jetchev--Nekov\'a{\v r}--Skinner \cite{JNS} and the explicit reciprocity law.

\subsection{g--unbalanced reciprocity law}
\label{subsec:diag}

Let $(\bff,\bfg,\bfh)$ be a triple of primitive Hida families as in $\S\ref{subsubsec:hida}$ satisfying (\ref{eq:a}). Let $N={\rm lcm}(N_f,N_g,N_h)$. The big diagonal class constructed in \cite[\S{8.1}]{BSV}
\begin{equation}\label{eq:3-diag}
\kappa(\bff,\bfg,\bfh)\in\rH^1(\Q,\Vdag(N)),
\end{equation}
where $\Vdag(N)$ (this is $V(\bff,\bfg,\bfh)$ using the notation of \cite{BSV}) is a free $\mathcal{R}$-module isomorphic to finitely many copies of $\Vdag$, can be identified with class $\boldsymbol{\kappa}_m^{(2)}$ in equation $(3.18)$ of \cite{CD} when $m=1$. The definition of the Selmer groups in $\S\ref{subsec:tripleSelmer}$ extends to $\Vdag(N)$, and by \cite[Cor.\,8.2]{BSV} we have $\kappa(\bff,\bfg,\bfh)\in{\rm Sel}^{\rm bal}(\Q,\Vdag(N))$. Now we choose level-$N$ test vectors $(\breve{\bff},\breve{\bfg},\breve{\bfh})$, provided by \cite[Thm.~A]{hsieh-triple}, to project the classes from $\Vdag(N)$ to $\Vdag$ (see pages $50,51$ of loc. cit.).

We define more one-dimensional $G_{\Q_p}$-invariant subspaces of $\Vdag$:  
\begin{equation}\label{eq:3-summands}
\begin{aligned}
\mathscr{F}_p^{3}(\Vdag)&=V_\bff^+\hat\otimes_{\cO} V_{\bfg}^+\hat\otimes_{\cO} V_{\bfh}^+\otimes\mathcal{X}^{-1}, \\
\mathbf{V}_{\bff}^{\bfg\bfh}&=V_\bff^-\hat\otimes_{\cO} V_{\bfg}^+\hat\otimes_{\cO} V_{\bfh}^+\otimes\mathcal{X}^{-1},\\
\mathbf{V}_{\bfg}^{\bff\bfh}&=V_\bff^+\hat\otimes_{\cO} V_{\bfg}^-\hat\otimes_{\cO} V_{\bfh}^+\otimes\mathcal{X}^{-1},\\
\mathbf{V}_{\bfh}^{\bff\bfg}&=V_\bff^+\hat\otimes_{\cO} V_{\bfg}^+\hat\otimes_{\cO} V_{\bfh}^-\otimes\mathcal{X}^{-1}.
\end{aligned}
\end{equation}
Note that 
\begin{equation}\label{eq:gr2}
\mathscr{F}^{\rm bal}_p(\Vdag)/\mathscr{F}_p^3(\Vdag)\cong
\mathbf{V}_{\bff}^{\bfg\bfh}\oplus\mathbf{V}_{\bfg}^{\bff\bfh}\oplus\mathbf{V}_{\bfh}^{\bff\bfg},
\end{equation}

Assume that the congruence ideal $C(\bfg)\subset\cR_g$ is principal, generated by the nonzero $\eta_\bfg\in\cR_g$ (recall that this will be satisfied when the residual representation $\bar{\rho}_{\bfg}$ is absolutely irreducible and $p$-distinguished). 
One can deduce from results in \cite{KLZ} the construction of an injective three-variable $p$-adic regulator map with pseudo-null cokernel:
\begin{equation}\label{eq:Log}
{\rm Log}^{\eta_\bfg}
:\rH^1(\Q_p,\mathbf{V}_{\bfg}^{\bff\bfh})\rightarrow\mathcal{R},
\end{equation}
see the explicit map in \cite[\S{5.3.1}]{CD} and the explanation in \cite[\S{7.3}]{BSV}.

Let ${\rm res}_p(\kappa(\bff,\bfg,\bfh))_{\bfg}$ be the image of $\kappa(\bff,\bfg,\bfh)$ under the natural composition of maps:
\begin{equation}\label{eq:map-ERL}
{\rm Sel}^{\rm bal}(\Q,\Vdag)\xrightarrow{{\rm res}_p}\rH^1(\Q_p,\mathscr{F}_p^{\rm bal}(\Vdag))\rightarrow\rH^1(\Q_p,\mathscr{F}_p^{\rm bal}(\Vdag)/\mathscr{F}_p^3(\Vdag))\rightarrow\rH^1(\Q_p,\mathbf{V}_{\bfg}^{\bff\bfh}),
\end{equation}
where we first restrict at $p$ and then project onto the first direct summand in (\ref{eq:gr2}). The following result is an explicit reciprocity law that relates diagonal cycles with the triple product $p$-adic $L$-functions.

\begin{thm}
\label{thm:ERL}
Let $(\bff,\bfg,\bfh)$ be a triple of primitive Hida families as in Theorem~\ref{thm:hsieh-triple}. Assume that $\bar{\rho}_g$ is absolutely irreducible and $p$-distinguished and fix a generator $\eta_{\bfg}$ of the congruence ideal of $\bfg$. Then for every triple $(\breve{\bff},\breve{\bfg},\breve{\bfh})$ of level $N$ test vectors  for $(\bff,\bfg,\bfh)$ we have
\begin{equation}\label{eq:ERL-g}
\eta_{\bfg}\cdot{\mathrm{Log}}^{\bfg}_{(\breve{\bff}^\star,\breve{\bfg}^\star,\breve{\bfh}^\star)}({\mathrm{res}}_p(\kappa(\bff,\bfg,\bfh))_\bfg)=\mathscr{L}_p^{\bfg}(\bff,\bfg,\bfh),
\end{equation}
\end{thm}

\begin{proof}
   This is the main theorem A in \cite{BSV}. This version is extracted from Remark 5.7 in \cite{CD}.
\end{proof}

\subsection{The preparation} We will specialize the first CM Hida family $\bff$ at $Q_1$ and leave the other two CM Hida families intact, following the convention in Section \ref{subsec:CM}. Let $f\in S_{k_1'}(pN_{\psi_1'})$ be the ordinary $p$-stabilization of $\bff:=\boldsymbol{\theta}_{\xi_1'}(S_1')$ 
at $S_1'=(1+p)^{k_1'}-1$ with $U_p$ eigenvalue $\psi_1'(\ppbar_1)$. For $i\in\{1,2\}$ let $\fkf_i\subset\cO_{K_i}$ be an ideal coprime to $pN_{\psi_i}$, $\ch_i$ be ray class characters of $K_i$ of conductors dividing $\mathfrak{f}_i$. Let $\chi_{\ch_i}$ be the central character of $\ch_i$. We assume that  
\begin{equation}\label{eq:sd-chars}
\chi_{\ch_1}\epsilon_{K_1}\chi_{\ch_2}\epsilon_{K_2}=1,
\end{equation}
and let
\begin{equation}\label{eq:gg*}
\bfg:=\boldsymbol{\theta}_{\ch_1}(S_1)\in\cO\dBr{S_1}\dBr{q}, 
\quad
\bfh:=\boldsymbol{\theta}_{\ch_2}(S_2)\in\cO\dBr{S_2}\dBr{q}
\end{equation}
be the CM Hida families 
attached to $\ch_1$ and $\ch_2$, respectively.
\begin{assumption}
    We assume that $\psi_1(\pp_1)\neq \psi_1(\ppbar_1)\pmod{p}$ (this ensures that the CM Hida family $\bfg:=\boldsymbol{\theta}_{\xi_1}(S_1)$ passing through $\theta_{\psi_1}$ is residually absolutely irreducible and $p$-distinguished). 
\end{assumption}

The triple $(\bff,\bfg,\bfh)$ satisfies conditions (\ref{eq:a}) and the associated $\bfg$-unbalanced triple product $p$-adic $L$-function $\mathscr{L}_p^{\bfg,\eta_{\bfg}}(\bff,\bfg,\bfh)$ 
is an element in $\mathcal{R}=\cR\hat\otimes_{\cO}\cO\dBr{S_1}\hat\otimes_{\cO}\cO\dBr{S_2}\simeq\cR\dBr{S_1,S_2}$. Let 
\begin{equation}\label{eq:triple-2var}
\mathscr{L}_p^{\bfg,\eta_\bfg}(f,\bfg,\bfh)\in\cO\dBr{S_1,S_2}
\end{equation}
be its image  
under the natural map $\cR\dBr{S_1,S_2}\rightarrow\cO\dBr{S_1,S_2}$ defined by $Q_1$.

Write $\Vsdag$ for the specialization of $\Vdag$ at $Q_1$. Let $V_f^\vee$ be the Galois representation associated to $f$, which in this case is $\mathrm{Ind}_{K_1}^\Q (\psi_1'^{-1})$. Recall that $\det(V_f^\vee)=\varepsilon_{\rm cyc}^{k_1'-1}$ in our conventions. Setting $T_i=\frac{(1+S_i)}{1+p}-1$ ($i\in\{1,2\}$), we have 
$\det(V_{\bfg_{T_1}}\otimes V_{\bfh_{T_2}})=\Psi_{T_1}\Psi_{T_2}\circ\mathscr{V}$, and so
\begin{equation}\label{eq:dec-V-first}
\begin{aligned}
\Vsdag&\simeq \mathrm{Ind}_{K_1}^\Q \cO(\psi_1'^{-1})\otimes({\rm Ind}_K^\Q\ch_1^{-1}\Psi_{T_1})\otimes({\rm Ind}_K^\Q\ch_2^{-1}\Psi_{T_2})\otimes\varepsilon_{\rm cyc}^{1-k_1'/2}(\Psi_{T_1}^{-1/2}\Psi_{T_2}^{-1/2}\circ\mathscr{V})\\
&\simeq \mathrm{Ind}_{K_1}^\Q \cO(\psi_1'^{-1})(1-k_1'/2)\otimes{\rm Ind}_{K_0}^\Q\tilde{\ch}_1^{-1}\tilde{\Psi}_{T_1}\tilde{\ch}_2^{-1}\tilde{\Psi}_{T_2},
\end{aligned}
\end{equation}
In particular, we get
\begin{align}
    \label{eq:shapiro}
\rH^1(\Q,\Vsdag)&\simeq\rH^1(K_0,\mathrm{Ind}_{K_1}^\Q \cO(\psi_1'^{-1})(1-k_1'/2)\otimes\tilde{\ch}_1^{-1}\tilde{\Psi}_{T_1}\tilde{\ch}_2^{-1}\tilde{\Psi}_{T_2})\nonumber \\
&\simeq \rH^1(K_0,(\cO(\tilde{\psi}_1'^{-1})\oplus \cO(\tilde{\psi}_1'^{\cc}))(1-k_1'/2)\otimes\tilde{\ch}_1^{-1}\tilde{\Psi}_{T_1}\tilde{\ch}_2^{-1}\tilde{\Psi}_{T_2}) \nonumber\\
&\simeq \rH^1(K_0,\cO(\tilde{\psi}_1'^{-1})(1-k_1'/2)\otimes\tilde{\ch}_1^{-1}\tilde{\Psi}_{T_1}\tilde{\ch}_2^{-1}\tilde{\Psi}_{T_2})\oplus \rH^1(K_0,\cO(\tilde{\psi}_1'^{-\cc})(1-k_1'/2)\otimes\tilde{\ch}_1^{-1}\tilde{\Psi}_{T_1}\tilde{\ch}_2^{-1}\tilde{\Psi}_{T_2})
\end{align}
by Shapiro's lemma.

\begin{rem}
    Results in this section can be obtained similarly in the $\bfh$-unbalanced case under the assumption that $\psi_2(\pp_2)\neq \psi_2(\ppbar_2)\pmod{p}$.
\end{rem}

\subsection{Local conditions of the Euler system}\label{subsec:local-condition} 


Recall from Theorem~\ref{maintheorem2} and \ref{maintheorem2-chi} that we have constructed classes
\[
\mathbf{z}_{\mu_3,\psi}\in H^1_{\rm Iw}\bigl(K_0[\mu_3 p^\infty],\cO_{\psi}\bigr), \qquad ^{\cc}\mathbf{z}_{\mu_3,\chi}\in H^1_{\rm Iw}\bigl(K_0[\mu_3 p^\infty],\cO_{\chi}\bigr),
\]
where 
\[\psi^{-1}=\widetilde{\psi_1{\psi}_1'}\,\widetilde{\psi_2}\,\bN^{(k_1'+k_1+k_2-4)/2},\qquad
    \chi^{-1}=\widetilde{\psi_1{\psi}_1'^{\cc}}\,\widetilde{\psi_2}\,\bN^{(k_1'+k_1+k_2-4)/2}.\]
\begin{prop}
\label{propselmer2}
Suppose $p\nmid 6h_{K_0}$. Let $\mu_3\in\mathcal{N}'$ (see Definition \ref{def:thin-primes-infinity}) run over squarefree product of prime ideals of $\la_3\in\mathcal{L}'$  with $m=N_{K_3/\Q}(\mu_3)$ coprime to $p$. The classes of Theorems\,\ref{maintheorem2} and \ref{maintheorem2-chi} satisfy
\[
\mathbf{z}_{\mu_3,\psi}\in{\rm Sel}_{\varnothing,0,0,\varnothing}(K_0[\mu_3p^\infty],\cO_{\psi}),\qquad ^{\cc}\mathbf{z}_{\mu_3,\chi}\in{\rm Sel}_{\varnothing,0,\varnothing,0}(K_0[\mu_3p^\infty],\cO_{\chi}).
\]
\end{prop}

\begin{proof}
By \cite[Cor.\,8.2]{BSV} and \cite[Sec.\,4.2]{CD}, the classes $\mathbf{z}_{\mu_3,\psi},^{\cc}\mathbf{z}_{\mu_3,\chi}$ lands in the balanced Selmer group ${\rm Sel}^{\rm bal}(\Q,\Vdag)$, where the balanced local condition at $p$ upon specialised to $f$ is given by
\begin{equation}\label{eq:bal-def}
\begin{aligned}
\mathscr{F}^{\rm bal}_p(\Vsdag)
\simeq\bigl(\mathrm{Ind}_{K_1}^\Q \cO(\psi_1'^{-1})(1-k_1'/2)\otimes\tilde{\ch}_1^{-1}\tilde{\Psi}_{T_1}\tilde{\ch}_2^{-1}\tilde{\Psi}_{T_2}\bigr)&\oplus\bigl(\cO(\psi_1'^{-1})(1-k_1'/2)\otimes\tilde{\ch}_1^{-1}\tilde{\Psi}_{T_1}\tilde{\ch}_2^{-\cc}\tilde{\Psi}_{T_2}^{\cc}\bigr)\\
&\oplus\bigl(\cO(\psi_1'^{-1})(1-k_1'/2)\otimes\tilde{\ch}_1^{-\cc}\tilde{\Psi}_{T_1}^{\cc}\tilde{\ch}_2^{-1}\tilde{\Psi}_{T_2}\bigr).
\end{aligned}
\end{equation}

Put 
\begin{align*}
    \widetilde{\mathbf{V}}_{\psi,Q_1}^\dagger=\cO(\psi_1'^{-1})(1-k_1'/2)\otimes\tilde{\ch}_1^{-1}\tilde{\Psi}_{T_1}\tilde{\ch}_2^{-1}\tilde{\Psi}_{T_2}\\
\widetilde{\mathbf{V}}_{\chi,Q_1}^\dagger=\cO(\psi_1'^{-\cc})(1-k_1'/2)\otimes\tilde{\ch}_1^{-1}\tilde{\Psi}_{T_1}\tilde{\ch}_2^{-1}\tilde{\Psi}_{T_2}
\end{align*}
then Shapiro Lemma tells us that 
\[H^1(\Q,\Vsdag)\simeq H^1(K_0,\widetilde{\mathbf{V}}_{\psi,Q_1}^\dagger)\oplus H^1(K_0,\widetilde{\mathbf{V}}_{\chi,Q_1}^\dagger)\]
see (\ref{eq:shapiro}). The local condition  $\mathscr{F}_p^{\rm bal}(\Vsdag)$ cutting out the specialised balanced Selmer group at $p$ corresponds to
\begin{equation}\label{eqn:local-condition-for-reciprocity-law}
\begin{split}\
\mathscr{F}_{\cP_1}^{\rm bal}(\widetilde{\mathbf{V}}^\dagger_{Q_1}\vert_{G_{K_0}})&=\cO(\psi_1'^{-1})(1-k_1'/2)\otimes\tilde{\ch}_1^{-1}\tilde{\Psi}_{T_1}\tilde{\ch}_2^{-1}\tilde{\Psi}_{T_2}\oplus \cO(\psi_1'^{-\cc})(1-k_1'/2)\otimes\tilde{\ch}_1^{-1}\tilde{\Psi}_{T_1}\tilde{\ch}_2^{-1}\tilde{\Psi}_{T_2},\\
\mathscr{F}_{\cP_2}^{\rm bal}(\widetilde{\mathbf{V}}^\dagger_{Q_1}\vert_{G_{K_0}})&=\{0\}\oplus \{0\},\\
\mathscr{F}_{\cP_3}^{\rm bal}(\widetilde{\mathbf{V}}^\dagger_{Q_1}\vert_{G_{K_0}})&=\{0\}\oplus \cO(\psi_1'^{-1})(1-k_1'/2)\otimes\tilde{\ch}_1^{-\cc}\tilde{\Psi}_{T_1}^{\cc}\tilde{\ch}_2^{-1}\tilde{\Psi}_{T_2},\\
\mathscr{F}_{\cP_4}^{\rm bal}(\widetilde{\mathbf{V}}^\dagger_{Q_1}\vert_{G_{K_0}})&=\cO(\psi_1'^{-1})(1-k_1'/2)\otimes\tilde{\ch}_1^{-1}\tilde{\Psi}_{T_1}\tilde{\ch}_2^{-\cc}\tilde{\Psi}_{T_2}^{\cc}\oplus \{0\}.
\end{split}
\end{equation}
Hence the class $\mathbf{z}_{\mu_3,\psi}$ satisfies the $(\varnothing,0,0,\varnothing)$ condition and the class $\mathbf{z}_{\mu_3,\chi}$ satisfies the $(\varnothing,0,\varnothing,0)$ condition at the primes above $p$. 
\end{proof}

\begin{prop}\label{prop:factor-S}
Via the isomorphism (\ref{eq:shapiro}), 
\begin{enumerate}
    \item the balanced Selmer group ${\rm Sel}^{\rm bal}(\Q,\Vsdag)$ can be rewritten as
\begin{align*}
{\rm Sel}^{\rm bal}(\Q,\Vsdag)&\simeq{\rm Sel}_{\varnothing,0,0,\varnothing}(K_0,\widetilde{\mathbf{V}}_{\psi,Q_1}^\dagger)\oplus {\rm Sel}_{\varnothing,0,\varnothing,0}(K_0,\widetilde{\mathbf{V}}_{\chi,Q_1}^\dagger),
\end{align*}
\item the $\bff$-unbalanced Selmer group ${\rm Sel}^{\bff}(\Q,\Vsdag)$ can be rewritten as
\begin{align*}
{\rm Sel}^{\bff}(\Q,\Vsdag)&\simeq
{\rm Sel}_{\varnothing,0,0,\varnothing}(K_0,\widetilde{\mathbf{V}}_{\psi,Q_1}^\dagger)\oplus {\rm Sel}_{0,\varnothing,\varnothing,0}(K_0,\widetilde{\mathbf{V}}_{\chi,Q_1}^\dagger).\end{align*}
\item the $\bfg$-unbalanced Selmer group ${\rm Sel}^{\bfg}(\Q,\Vsdag)$ can be rewritten as
\begin{align*}
{\rm Sel}^{\bfg}(\Q,\Vsdag)&\simeq
{\rm Sel}_{\varnothing,0,0,\varnothing}(K_0,\widetilde{\mathbf{V}}_{\psi,Q_1}^\dagger)\oplus {\rm Sel}_{\varnothing,0,0,\varnothing}(K_0,\widetilde{\mathbf{V}}_{\chi,Q_1}^\dagger).\end{align*}
\item the $\bfh$-unbalanced Selmer group ${\rm Sel}^{\bfh}(\Q,\Vsdag)$ can be rewritten as
\begin{align*}
{\rm Sel}^{\bfh}(\Q,\Vsdag)&\simeq
{\rm Sel}_{\varnothing,0,\varnothing,0}(K_0,\widetilde{\mathbf{V}}_{\psi,Q_1}^\dagger)\oplus {\rm Sel}_{\varnothing,0,\varnothing,0}(K_0,\widetilde{\mathbf{V}}_{\chi,Q_1}^\dagger).\end{align*}
\end{enumerate}
\end{prop}

\begin{proof} For the balanced case, see Proposition \ref{propselmer2}. For the $\unb$-unbalanced case, note that 
\begin{equation}\label{eq:unbal-def}
\begin{aligned}
\mathscr{F}^{\bff}_p(\Vsdag\vert_{G_{K_0}})
\simeq\bigl(\cO(\psi_1'^{-1})(1-k_1'/2)\otimes\tilde{\ch}_1^{-1}\tilde{\Psi}_{T_1}\tilde{\ch}_2^{-1}\tilde{\Psi}_{T_2}\bigr)&\oplus\bigl(\cO(\psi_1'^{-1})(1-k_1'/2)\otimes\tilde{\ch}_1^{-1}\tilde{\Psi}_{T_1}\tilde{\ch}_2^{-\cc}\tilde{\Psi}_{T_2}^{\cc}\bigr)\\
\oplus\bigl(\cO(\psi_1'^{-1})(1-k_1'/2)\otimes\tilde{\ch}_1^{-\cc}\tilde{\Psi}_{T_1}^{\cc}\tilde{\ch}_2^{-1}\tilde{\Psi}_{T_2}\bigr)&\oplus\bigl(\cO(\psi_1'^{-1})(1-k_1'/2)\otimes\tilde{\ch}_1^{-\cc}\tilde{\Psi}_{T_1}^{\cc}\tilde{\ch}_2^{-\cc}\tilde{\Psi}_{T_2}^{\cc}\bigr).
\end{aligned}
\end{equation}
and the result follows. Similar computations can be performed for the $\bfg$-unbalanced and the $\bfh$-unbalanced cases.
\end{proof}

As a consequence, we also obtain the following isomorphisms for the Selmer groups with coefficients in 
\begin{align*}\mathbf{A}_{Q_1}^\dagger&={\rm Hom}_{\Z_p}(\Vsdag,\mu_{p^\infty}),\\
    \widetilde{\mathbf{A}}_{\psi,Q_1}^\dagger&=\cO(\psi_1')(k_1'/2)\otimes\tilde{\ch}_1\tilde{\Psi}_{T_1}^{-1}\tilde{\ch}_2\tilde{\Psi}_{T_2}^{-1},\\
\widetilde{\mathbf{A}}_{\chi,Q_1}^\dagger&=\cO(\psi_1'^{\cc})(k_1'/2)\otimes\tilde{\ch}_1\tilde{\Psi}_{T_1}^{-1}\tilde{\ch}_2\tilde{\Psi}_{T_2}^{-1}
\end{align*}by local Tate duality. 

\begin{cor}\label{cor:factor-S} We can identify
the balanced Selmer group ${\rm Sel}^{\rm bal}(\Q,\mathbf{A}_{Q_1}^\dagger)$ as
\begin{align*}
{\rm Sel}^{\rm bal}(\Q,\mathbf{A}_{Q_1}^\dagger)&\simeq{\rm Sel}_{0,\varnothing,\varnothing,0}(K_0,\widetilde{\mathbf{A}}_{\psi,Q_1}^\dagger)\oplus {\rm Sel}_{0,\varnothing,0,\varnothing}(K_0,\widetilde{\mathbf{A}}_{\chi,Q_1}^\dagger),
\end{align*}
the $\unb$-unbalanced Selmer group ${\rm Sel}^{\unb}(\Q,\mathbf{A}_{Q_1}^\dagger)$ as
\begin{align*}
{\rm Sel}^{\unb}(\Q,\mathbf{A}_{Q_1}^\dagger)&\simeq{\rm Sel}_{0,\varnothing,\varnothing,0}(K_0,\widetilde{\mathbf{A}}_{\psi,Q_1}^\dagger)\oplus {\rm Sel}_{\varnothing,0,0,\varnothing}(K_0,\widetilde{\mathbf{A}}_{\chi,Q_1}^\dagger),
\end{align*}
the $\bfg$-unbalanced Selmer group ${\rm Sel}^{\bfg}(\Q,\Vsdag)$ can be rewritten as
\begin{align*}
{\rm Sel}^{\bfg}(\Q,\mathbf{A}_{Q_1}^\dagger)&\simeq
{\rm Sel}_{0,\varnothing,\varnothing,0}(K_0,\widetilde{\mathbf{A}}_{\psi,Q_1}^\dagger)\oplus {\rm Sel}_{0,\varnothing,\varnothing,0}(K_0,\widetilde{\mathbf{A}}_{\chi,Q_1}^\dagger),\end{align*}
and the $\bfh$-unbalanced Selmer group ${\rm Sel}^{\bfh}(\Q,\mathbf{A}_{Q_1}^\dagger)$ as
\begin{align*}
{\rm Sel}^{\bfh}(\Q,\mathbf{A}_{Q_1}^\dagger)&\simeq
{\rm Sel}_{0,\varnothing,0,\varnothing}(K_0,\widetilde{\mathbf{A}}_{\psi,Q_1}^\dagger)\oplus {\rm Sel}_{0,\varnothing,0,\varnothing}(K_0,\widetilde{\mathbf{A}}_{\chi,Q_1}^\dagger).\end{align*}
\end{cor}

\subsection{Applying the general machinery} \label{sec:JNS-Euler-system}
We show crucial arithmetic applications by invoking the general Euler system machinery of Jetchev--Nekov\'a{\v r}--Skinner  \cite{JNS} (see some details for the imaginary quadratic case in \cite[\S{4.3}]{Do-PhD} and \cite[\S{8}]{ACR}). These results will be utilized to deduce the Bloch--Kato conjecture and the anticyclotomic Iwasawa main conjecture  via an explicit reciprocity law that relates our Euler system classes and special values of complex and $p$-adic $L$-functions.

For every ideal $\mu_3\in \mathcal{N}$, denote by
\[
z_{\mu_3,\psi}\in{\rm Sel}_{\varnothing,0,0,\varnothing}(K_0[\mu_3],\cO_{\psi})
\]
the image of $\mathbf{z}_{\mu_3,\psi}$ from Theorem\,\ref{maintheorem2} and Proposition \ref{propselmer2}  under the projection
\begin{equation}\label{eq:Sel_projection_from_K_infty}
  {\rm Sel}_{\varnothing,0,0,\varnothing}(K_0[\mu_3p^\infty],\cO_{\psi})\rightarrow {\rm Sel}_{\varnothing,0,0,\varnothing}(K_0[\mu_3],\cO_{\psi}).
\end{equation}
And denote the base class
\[
z_{\psi}:={\rm Norm}^{K_0[1]}_{K_0}(z_{1,\psi})\in{\rm Sel}_{\varnothing,0,0,\varnothing}(K_0,\cO_{\psi}).
\]
Note that since we assume $p \nmid h_{K_0}$, $K_0[1]$ is actually the same with $K_0$ (recall that $K_0[\fkn]$ is the maximal $p$-extension inside the ring class field of $K_0$ of conductor $\fkn$). Therefore, $z_{\psi}=z_{1,\psi}$.

Similarly, we define the image of $^{\cc}\mathbf{z}_{\mu_3,\chi}$ from Theorem \ref{maintheorem2-chi} under the projection map

\[^{\cc}z_{\mu_3,\chi}\in{\rm Sel}_{\varnothing,0,\varnothing,0}(K_0[\mu_3],\cO_{\chi})\]
together with its base class
\[
^{\cc}z_{\chi}\in{\rm Sel}_{\varnothing,0,\varnothing,0}(K_0,\cO_{\chi}).
\]

\begin{thm}\label{thm:rank-1-general}
Assume $p\nmid 6h_{K_0}$ and $p>k_1'$. One has the following implications:
\begin{align*}
    z_{\psi} \text{ is non-torsion } \quad\Rightarrow\quad{\rm Sel}_{\varnothing,0,0,\varnothing}(K_0,V_{\psi}) \text{ is one-dimensional},\\
    ^{\cc}z_{\chi} \text{ is non-torsion } \quad\Rightarrow\quad{\rm Sel}_{\varnothing,0,\varnothing,0}(K_0,V_{\chi}) \text{ is one-dimensional.}
\end{align*}

\end{thm}

\begin{proof}
We prove the first statement. Combining Theorem~\ref{maintheorem2} and Proposition~\ref{propselmer2}, the system of classes 
\begin{equation}\label{eq:ES-rel-str}
\bigl\{z_{\mu_3,\psi}\in{\rm Sel}_{\varnothing,0,0,\varnothing}(K_0[\mu_3],\cO_{\psi})\;\colon\;\mu_3\in\mathcal{N}\bigr\}
\end{equation}
forms an anticyclotomic Euler system in the sense of  Jetchev--Nekov\'a{\v r}--Skinner \cite{JNS}.

The following properties (i)--(iii) follow easily:
\begin{itemize}
    \item[(i)] $V_{\psi}$ is absolutely irreducible (as it is one-dimensional);
    \item[(ii)] There is an element $\sigma\in G_{K_0}$ fixing $K_0^{\infty}:=K_0[1]K_0(\mu_{p^\infty},(\cO_{K_0}^\times)^{1/p^\infty})$ such that $V_{\psi}/(\sigma-1)V_{\psi}$ is one-dimensional (take $\sigma=1$);
    \item[(iii)] There is an element $\gamma\in G_{K_0}$ fixing $K_0^{\infty}$ such that $V_{\psi}^{\gamma=1}=0$. One can take $\gamma\not\in \ker \psi$, which is possible since $\psi$ does not factor through ${\rm Gal}(K_0^{\infty}/K_0)$. Indeed, if $\psi$ factors through $K_0^{\infty}$ then it will factor through the maximal abelian sub-extension of $K_0^{\infty}/K_0$, which is $K^{\rm cyc}=K_0[1]K_0(\mu_{p^\infty})$ as the Kummer extension part contributes absolutely nothing to the maximal abelian sub-extension for $p\ge 5$. Hence, $\psi$ will factor through $K^{\rm cyc}$, which forces $\psi$ to be equal to a power of cyclotomic character times a finite order character. Therefore, the infinity type of $\psi$ will be $(a,a,a,a)$, a contradiction.
\end{itemize}
Hence, the fact that $z_{\psi}$ is non-torsion implies the one-dimensionality of ${\rm Sel}_{\varnothing,0,0,\varnothing}(K_0,V_{\psi})$ by the general machinery of \cite{JNS}. The proof for the second part is identical.
\end{proof}

Recall that $K_{0,\infty}^-$ is the anticyclotomic $\Z_p^2$ extension over $K_0$ and $\Lambda_{K_0}^-=\Z_p\dBr{{\rm Gal}(K_{0,\infty}^-/K_0)}$. 
Let $\mathbf{z}_{1,\psi}$ be the $\Lambda_{K_0}^-$-adic class of Theorem\,\ref{maintheorem2} of conductor $\mu_3=(1)$, and put the Iwasawa-theoretic base class
\[
\mathbf{z}_{\psi}:={\rm Norm}^{K_0[1]}_{K_0}(\mathbf{z}_{1,\psi}).
\]
Again $\mathbf{z}_{1,\psi}=\mathbf{z}_{\psi}$ from the assumption on the class number of $K_0$. Note that by Proposition \ref{propselmer2}, one has
\[
\mathbf{z}_{\psi}\in{\rm Sel}_{\varnothing,0,0,\varnothing}(K_0[p^\infty],\cO_{\psi}).
\]
Similarly, we have another base class
\[
^{\cc}\mathbf{z}_{\chi}\in{\rm Sel}_{\varnothing,0,\varnothing,0}(K_0[p^\infty],\cO_{\chi}).
\]

Denote by 
\begin{align*}
    X_{0,\varnothing,\varnothing,0}(K_{0,\infty}^-,\cO_{\psi^{-1}}(1))&={\rm Hom}_{\Z_p}\bigl(\varinjlim{\rm Sel}_{\varnothing,0,0,\varnothing}(K_0[\mu_3p^\infty],\cO_{\psi^{-1}}(1)),\Q_p/\Z_p\bigr) \\
    X_{0,\varnothing,0,\varnothing}(K_{0,\infty}^-,\cO_{\chi^{-1}}(1))&={\rm Hom}_{\Z_p}\bigl(\varinjlim{\rm Sel}_{\varnothing,0,\varnothing,0}(K_0[\mu_3p^\infty],\cO_{\chi^{-1}}(1)),\Q_p/\Z_p\bigr).
\end{align*}

One then has a divisibility towards an anticyclotomic Iwasawa main conjecture `without $p$-adic $L$-functions' as follows:
\begin{thm}\label{thm:IMC-general}
Assume that $p\nmid 6h_{K_0}$ and $p>k_1'$. 
\begin{enumerate}
    \item If $\mathbf{z}_{\psi}$ is non-torsion, then: 
\begin{enumerate}
    \item $X_{0,\varnothing,\varnothing,0}(K_{0,\infty}^-,\cO_{\psi^{-1}}(1))$ and ${\rm Sel}_{\varnothing,0,0,\varnothing}(K_0[\mu_3p^\infty],\cO_{\psi})$ both have $\Lambda_{K_0}^-$-rank one.
\item And we have the divisibility
\[
{\rm char}_{\Lambda_{K_0}^-}(X_{0,\varnothing,\varnothing,0}(K_{0,\infty}^-,\cO_{\psi^{-1}}(1))_{\rm tors})\supset{\rm char}_{\Lambda_{K_0}^-}\biggl(\frac{{\rm Sel}_{\varnothing,0,0,\varnothing}(K_0[p^\infty],\cO_{\psi})}{\Lambda_{K_0}^-\cdot\mathbf{z}_{\psi}}\biggr)^2
\]
in $\Lambda_{K_0}^-$. 
\end{enumerate}
     \item If $^{\cc}\mathbf{z}_{\chi}$ is non-torsion, then: 
\begin{enumerate}
    \item $X_{0,\varnothing,0,\varnothing}(K_{0,\infty}^-,\cO_{\chi^{-1}}(1))$ and ${\rm Sel}_{\varnothing,0,\varnothing,0}(K_0[\mu_3p^\infty],\cO_{\chi})$ both have $\Lambda_{K_0}^-$-rank one.
\item And we have the divisibility
\[
{\rm char}_{\Lambda_{K_0}^-}(X_{0,\varnothing,0,\varnothing}(K_{0,\infty}^-,\cO_{\chi^{-1}}(1))_{\rm tors})\supset{\rm char}_{\Lambda_{K_0}^-}\biggl(\frac{{\rm Sel}_{\varnothing,0,\varnothing,0}(K_0[p^\infty],\cO_{\chi})}{\Lambda_{K_0}^-\cdot \,^{\cc}\mathbf{z}_{\chi}}\biggr)^2
\]
in $\Lambda_{K_0}^-$. 
\end{enumerate}
\end{enumerate}

Here, the subscript ${\rm tors}$ denotes the $\Lambda_{K_0}^-$-torsion submodule.
\end{thm}

\begin{proof}
We prove the first part, the proof of the second one is similar. Combining Theorem\,\ref{maintheorem2} and Proposition\,\ref{propselmer2}, the system of classes
\begin{equation}\label{eq:ES-Lambda}
\bigl\{\mathbf{z}_{\mu_3,\psi}\in{\rm Sel}_{\varnothing,0,0,\varnothing}(K_0[\mu_3p^\infty],\cO_{\psi})\,\colon\;\mu_3\in\mathcal{N}\bigr\}
\end{equation}
forms a $\Lambda_{K_0}^-$-adic anticyclotomic Euler system in the sense of Jetchev--Nekov\'a{\v r}--Skinner.

The following properties hold
\begin{itemize}
    \item[(i)] $\bar{\cO}_{\psi}:=\cO_{\psi}/\mathfrak{P}\cO_{\psi}$ is absolutely irreducible;
    \item[(ii)] There is an element $\sigma\in G_K$ fixing $K_0^{\infty}$ such that $\cO_{\psi}/(\sigma-1)\cO_{\psi}$ is free of rank $1$ over $\cO$;
    \item[(iii)] There is an element $\gamma\in G_{K_0}$ fixing $K_0^{\infty}$ and acting as multiplication by a scalar $a_\gamma\neq 1$ on $\bar{\cO}_{\psi}$ (one takes $\gamma\not\in\ker \bar{\psi}$, argues similar to Theorem \ref{thm:rank-1-general} modulo $(p-1)$ by replacing the cyclotomic character in the argument by $\omega$).
\end{itemize}
and so the non-torsionness of $\mathbf{z}_{\psi}$ implies the conclusions by the general machinery of \cite{JNS}.
\end{proof}

\subsection{On the Bloch--Kato conjecture in rank zero}
\label{subsec:BK-def}

Our first application is the Bloch--Kato conjecture in analytic rank zero for the conjugate self-dual $G_{K_0}$-representation $V_{\chi}=L_\mathfrak{P}(\chi)$.

\begin{assumption}\label{eq:chi-decomp}
    We assume that the conjugate self-dual Hecke character $\chi$ over $K_0$ can be decomposed as: 
    \begin{equation*}
\chi^{-1}=\widetilde{\psi_1{\psi}_1'^{\cc}}\,\widetilde{\psi_2}\,\bN^{(k_1'+k_1+k_2-4)/2},
    \end{equation*}
    where
    \begin{enumerate}
     \item $\psi_1,\psi_1'$ are Hecke characters of $K_1$ of infinity type $(1-k_1,0)$, $(1-k_1',0)$, with $k_1,k_1'\geq 1$, and modulus $\fkf_1$, $\fkf_1'$ respectively.
    \item  $\psi_2$ is a Hecke character of $K_2$ of infinity type $(1-k_2,0)$, with $k_2\geq 1$ and modulus  $\fkf_2$.
    \item $\tilde{\psi}_i$ is the Hecke character of $K_0$, obtained by composing $\mathbb{A}_{K_0}^{\times}\xrightarrow{\mathbb{N}_{K_0/K_i}}\mathbb{A}_{K_i}^{\times}\xrightarrow{\psi_i}\mathbb{C}$ for each $i\in\{1,2\}$.
    \item The following condition forces the \emph{self-duality} property of $\chi$:
    \begin{equation*}
{{\chi_{\psi_1'}\epsilon_{K_1}=1},\qquad\chi_{\psi_1}\epsilon_{K_1}\chi_{\psi_2}\epsilon_{K_2}=1}.
    \end{equation*}
    \item $\psi_1(\pp_1)\neq \psi_1(\ppbar_1)\pmod{p}$.
    \end{enumerate}
In this scenario, the infinity type of $\chi$ (corresponding to the order $(\cP_1,\cP_2,\cP_3,\cP_4)$ or $(1,\tau_3,\tau_2,\tau_1)$) is
\[\left(\frac{-k_1'+k_1+k_2}{2},\frac{k_1'-k_1-k_2}{2}+1,\frac{k_1'-k_1+k_2}{2},1-\frac{k_1'-k_1+k_2}{2}\right).\]
\end{assumption}

\begin{thm}\label{thm:BK-def}
Let $\chi$ be an anticyclotomic Hecke character of $K_0$ as in Assumption (\ref{eq:chi-decomp}). Assume further that:
\begin{enumerate}
    \item $k_1\ge k_1'+k_2$,
    \item $\prod_{p|N}\epsilon_p=1$ where $N=\mathrm{lcm}(N_{\psi_1'},N_{\psi_1},N_{\psi_2})$ and $\epsilon_p$ is the local root number of the triple product, (This will be automatic if $\mathrm{gcd}(N_{\psi_1'},N_{\psi_1},N_{\psi_2})=1$.)
    \item $p>k_1'$,
    \item $p\nmid 6h_{K_0}$, the class number of $K_0$,
\end{enumerate}
then we have the following implication
\[
L(\chi,0)\neq 0\quad\Longrightarrow\quad{\rm Sel}_{\rm BK}(K_0,V_{\chi})=0.
\] 
In other words, the Bloch--Kato conjecture holds in analytic rank zero for $V_{\chi}$.
\end{thm}

\begin{proof}
We consider the CM Hida families
\[
\bfg=\boldsymbol{\theta}_{\ch_1}(S_1),\quad\bfh=\boldsymbol{\theta}_{\ch_2}(S_2),
\]
that pass through $\theta_{\psi_1}$ and $\theta_{\psi_2}$ respectively. Note that the triple $(\bff,\bfg,\bfh)$ satisfies (\ref{eq:+1}). Then the isomorphism (\ref{eq:dec-V-first}) of the associated $\Vsdag$ together with the further specialization of $Q_1$ to ${\mathcal{Q}_1}$, corresponding to $\theta_{\psi_1}$ and $\theta_{\psi_2}$ (here $T_i=(1+p)^{k_i-1}-1$ for $i=1,2$) show a parallel decomposition for $L$-values
\[
L(\mathbb{V}_{\mathcal{Q}_1}^\dag,0)=L(\chi,0)\cdot L(\psi,0),
\]
where 
\begin{align*}
\psi&=(\widetilde{\psi_1{\psi}_1'}\,\widetilde{\psi_2}\,\bN^{(k_1'+k_1+k_2-4)/2})^{-1}\\
    \chi&=(\widetilde{\psi_1{\psi}_1'^{\cc}}\,\widetilde{\psi_2}\,\bN^{(k_1'+k_1+k_2-4)/2})^{-1}.
\end{align*}
Under the condition $k_1\ge k_1'+k_2$, we are in the range of interpolation for $\mathscr{L}_p^{\bfg}$. 

Note further that the first two conditions force the sign of the triple product to be $+1$, and this combines with the sign of the functional equation of $\chi$ being $+1$ (since $L(\chi,0)\neq 0$) will ensure that the sign of the functional equation of $\psi$ being $+1$ as well.

For any ring class character $\phi$ of $K_1$ with conductor $\ell^m$, one can change $(\psi_1,\psi_1')\mapsto(\psi_1\phi,\psi_1'\phi)$. This will fix $\chi$ and turn $\psi$ into $\psi\tilde{\phi}^{-2}$. Note that $\phi|_{\mathbb{A}_{\Q}^\times}=1$ so condition $(4)$ of Assumption \ref{eq:chi-decomp} is intact. Furthermore, as $\mathrm{Gal}(K_0/\Q)$ is abelian, $\tilde{\phi}$ will be anticyclotomic over $K_0$.  We now choose $\phi$ such that $L(\psi\tilde{\phi}^{-2},0)\neq 0$, which is possible by Theorem 1.4 of \cite{BHTY}, by assuming that $\ell\nmid 6pD_{K_3}$ and splits in $K_0$. Combining with Theorem\,\ref{thm:ERL}, we deduce that
\[
L(\chi,0)\neq 0\quad\Longrightarrow\quad L(\mathbb{V}_{\mathcal{Q}_1}^\dag,0)\neq 0.
\]

We observe that the specialization at $Q_1$ of the composition of maps in (\ref{eq:map-ERL}), which is a map from $ {\rm Sel}^{\rm bal}(\Q,\Vdag)$ to $\rH^1(\Q_p,\mathbf{V}_{\bfg}^{\bff\bfh})$, corresponds to the composition of (see also equation (\ref{eqn:local-condition-for-reciprocity-law}))
\begin{align*}
    {\rm Sel}^{\rm bal}(\Q,\Vsdag)&\rightarrow {\rm Sel}_{\varnothing,0,\varnothing,0}(K_0,\widetilde{\mathbf{V}}_{\chi,Q_1}^\dagger) \xrightarrow{{\rm res}_{\mathcal{P}_3}} \rH^1(K_{0,\mathcal{P}_3},\cO(\psi_1'^{-\cc})(1-k_1'/2)\otimes\tilde{\ch}_1^{-1}\tilde{\Psi}_{T_1}\tilde{\ch}_2^{-1}\tilde{\Psi}_{T_2}).
\end{align*}
Here, the first arrow is the projection to the second summand of part (1) of Proposition \ref{prop:factor-S}. Further specialization at $\cQ_1$ implies that
\[L(\mathbb{V}_{\mathcal{Q}_1}^\dag,0)\neq 0\qquad\Longrightarrow \qquad {\rm res}_{\mathcal{P}_3}(^{\cc}z_{\chi})\neq 0\]
by the reciprocity law in Theorem \ref{thm:ERL}. 

From our construction and Section \ref{sec:JNS-Euler-system}, the class $^{\cc}z_{\chi}\in{\rm Sel}_{\varnothing,0,\varnothing,0}(K_0,\cO_{\chi})$ is the base class of the anticyclotomic Euler system 
\[
\bigl\{^{\cc}z_{\mu_3,\chi}\in{\rm Sel}_{\varnothing,0,\varnothing,0}(K_0[\mu_3],\cO_{\chi})\;\colon\;\mu_3\in\mathcal{N}\bigr\}.
\]
 Recall again that $K_0[1]=K_0$ from the assumption $p \nmid h_{K_0}$. By Theorem\,\ref{thm:rank-1-general}, we conclude that the Selmer group ${\rm Sel}_{\varnothing,0,\varnothing,0}(K_0,V_{\chi})$ is one-dimensional, 
spanned by $^{\cc}z_{\chi}$.

 Using the fact that $\chi$ is conjugate self-dual, we see that ${\rm res}_{\cP_3}({\rm Sel}_{\varnothing,0,\varnothing,0}(K_0,V_{\chi}))\neq 0$ implies that $\mathrm{res}_{\cP_4}( \mathrm{Sel}_{0,\varnothing,0,\varnothing}(K_0,V_\chi^*))\neq 0$. Let $x\in \mathrm{Sel}_{0,\varnothing,0,\varnothing}(K_0,V_\chi^*)$ such that $\mathrm{res}_{\cP_4}(x)\neq 0$. Take any $\kappa\in \mathrm{Sel}_{\rm BK}(K_0,V_\chi) $, which is equal to $ {\rm Sel}_{\varnothing,0,0,\varnothing}(K_0,V_{\chi}) $ by Lemma \ref{lem:BK-Gr} due to our assumption that $k_1\ge k_1'+k_2$. By Mazur--Rubin \cite[Thm\,2.3.4]{MR-KS}, we have 
        \[\sum_{i=1}^4\langle \mathrm{res}_{\cP_i} (\kappa),\mathrm{res}_{\cP_i}(x)\rangle=0 \qquad\Rightarrow \qquad  \langle \mathrm{res}_{\cP_4} (\kappa),\mathrm{res}_{\cP_4}(x)\rangle=0 \qquad\Rightarrow \qquad \mathrm{res}_{\cP_4} (\kappa)=0,\]
        i.e. $\mathrm{Sel}_{\rm BK}(K_0,V_\chi)=\mathrm{Sel}_{\varnothing,0,0,0}(K_0,V_\chi)$. Now using the exact sequence
        \[0\longrightarrow \mathrm{Sel}_{\varnothing,0,0,0}(K_0,V_\chi)\longrightarrow {\rm Sel}_{\varnothing,0,\varnothing,0}(K_0,V_{\chi})\xrightarrow{{\rm res}_{\mathcal{P}_3}} \rH^1(K_{0,\mathcal{P}_3},V_\chi),\]
         we conclude that the Bloch--Kato Selmer group vanishes.
\end{proof}

\begin{rem}
    Note that since $\chi\chi^c=\bN$, we have $L(\chi^{-1}\bN,0)=L(\chi^\cc,0)$. And $L(\chi^\cc,0)=0$ if and only if $L(\chi,0)=0$, so our assumption fits into the framework of the Bloch--Kato conjecture.

\end{rem}

\subsubsection{Examples}\label{Example}
We give the following examples as an application of our main theorem: 
\begin{enumerate}
    \item $\chi$ is a Hecke character of infinity type $(1,0,0,1)$ and conductor $(42+42i)$ defined over $K_0=\Q(\zeta_{12})$ and decomposes as in Assumption (\ref{eq:chi-decomp}) where:
    \begin{enumerate}
        \item 
   $\psi_1'$ is a Hecke character of infinity type $(-1,0)$ and conductor $(2\sqrt{-3})$ defined over $K_1=\Q(\sqrt{-3})$ and gives rise to the CM modular form $f\in S_2(\Ga_0(36))$ (this is the modular form 36.2.a.a in the LMFDB list). 
   \item $\psi_1$ is a Hecke character of infinity type $(-2,0)$ and conductor $(6-3\sqrt{-3})$ defined  over $K_1=\Q(\sqrt{-3})$ and gives rise to the CM modular form $g\in S_3(\Ga_1(189),\chi_g)$ (this is the modular form 189.3.d.b in the LMFDB list). Here, $\chi_g$ is a character over $(\Z/189\Z)^\times$, induced from the Kronecker symbol $\left(\frac{-7}{\bullet}\right)$.
   \item $\psi_2$ is a Hecke character of finite order and conductor $(14+14i)$ defined  over $K_2=\Q(i)$ and gives rise to the CM modular form $h\in S_1(\Ga_1(1568),\chi_h)$ (this is the modular form 1568.1.c.a in the LMFDB list). Here, $\chi_h$ is a character over $(\Z/1568\Z)^\times$, induced also from the Kronecker symbol $\left(\frac{-7}{\bullet}\right)$, so this triple satisfies condition (4) of Assumption (\ref{eq:chi-decomp}). 
   \item The global root number for the triple product in this example is $+1$ since we are in the unbalanced range $(2,3,1)$ and $\gcd(36,189,1568)=1$.
    \end{enumerate}
    \item $\chi$ is a Hecke character of infinity type $(1,0,0,1)$ and conductor $(6+6i)$ defined over $K_0=\Q(\zeta_{12})$ and decomposes as in Assumption (\ref{eq:chi-decomp}) where:
    \begin{enumerate}
        \item 
   $\psi_1'$ is a Hecke character of infinity type $(-1,0)$ and conductor $(3)$ defined over $K_1=\Q(\sqrt{-3})$ and gives rise to the CM modular form $f\in S_2(\Ga_0(27))$ (this is the modular form 27.2.a.a in the LMFDB list). 
   \item $\psi_1$ is a Hecke character of infinity type $(-2,0)$ and conductor $(3)$ defined  over $K_1=\Q(\sqrt{-3})$ and gives rise to the CM modular form $g\in S_3(\Ga_1(27),\chi_g)$ (this is the modular form 27.3.b.a in the LMFDB list). Here, $\chi_g$ is a character over $(\Z/27\Z)^\times$, induced from the Kronecker symbol $\left(\frac{-3}{\bullet}\right)$.
   \item $\psi_2$ is a Hecke character of finite order and conductor $(6+6i)$ defined over $K_2=\Q(i)$ and gives rise to the CM modular form $h\in S_1(\Ga_1(288),\chi_h)$ (this is the modular form 288.1.e.a in the LMFDB list). Here, $\chi_h$ is a character over $(\Z/288\Z)^\times$, induced also from the Kronecker symbol $\left(\frac{-3}{\bullet}\right)$, so this triple satisfies condition (4) of Assumption (\ref{eq:chi-decomp})). 
   \item The computation of the local root number at $3$ of this example is a bit trickier since $3^2|\gcd(27,27,288)$, nevertheless the global root number for the triple product is $+1$.
    \end{enumerate}
\end{enumerate}

\subsection{Factorization of the triple product \texorpdfstring{$p$}{p}-adic \texorpdfstring{$L$}{L}-function}
\subsubsection{Anticyclotomic Katz \texorpdfstring{$p$}{p}-adic \texorpdfstring{$L$}{L}-function for imaginary quadratic fields}

We recall the interpolation property of the Katz $p$-adic $L$-functions by Katz and Hida--Tilouine \cite{Katz49,HT-ENS}, following the exposition in de Shalit \cite{deshalit}. For any ideal $\mathfrak{c}\subset\cO_K$ coprime to $p$, we let $Z(\mathfrak{c})$ denote the ray class group of $K$ of conductor $\mathfrak{c}p^\infty$ (so $Z(\mathfrak{c})\simeq{\mathrm{Gal}}(K_{\mathfrak{c}p^\infty}/K)$). Here, $K$ is an imaginary quadratic field.

\begin{thm}\label{thm:katz}
There exists an element $\Lcal_{\pp,\mathfrak{c}}^{\mathrm{K}}
\in\cO\dBr{Z(\mathfrak{c})}\hat\otimes_{\Z_p}\Z_p^{\ur}$ 
such that for every character $\xi$ of $Z(\mathfrak{c})$ that is crystalline at both $\pp$ and $\ppbar$, corresponding to a Hecke character of infinity type $(k,j)$ with $k>0$ and $j\leq 0$, then we have
\[
\Lcal_{\pp,\mathfrak{c}}^{\mathrm{K}}(\xi)=\frac{\Omega_p^{k-j}}{\Omega_\infty^{k-j}}\cdot\Gamma(k)\cdot\biggl(\frac{\sqrt{D_K}}{2\pi}\biggr)^j\cdot \mathcal{E}_p(\xi)\cdot L_{\mathfrak{c}}(\xi,0),
\]
where $L_{\mathfrak{c}}(\xi,s)$ denotes the Hecke $L$-function of $\xi$ with the Euler factors at the primes dividing $p$ and $\mathfrak{c}$ removed and $\mathcal{E}_p(\xi)=(1-\xi^{-1}(\pp)p^{-1})(1-\xi(\bar{\pp}))$.  
Moreover, we have the functional equation 
\[
\Lcal_{\pp,\mathfrak{c}}^{\mathrm{K}}(\xi)=\Lcal_{\pp,\bar{\mathfrak{c}}}^{\mathrm{K}}(\xi^{-\cc}\mathbf{N}^{-1}),
\] 
where the equality is up to a $p$-adic unit.
\end{thm}

\begin{proof}
Our $\Lcal_{\pp,\mathfrak{c}}^{\mathrm{K}}$ corresponds to the measure denoted $\mu(\mathfrak{c}\bar{\pp}^\infty)$ in \cite[Thm.~II.4.14]{deshalit}, 
and the functional equation is given in \cite[Thm.~II.6.4]{deshalit} (which allows to extend the interpolation property from $k>-j\geq 0$ to the entire range in the statement).
\end{proof}


Let $\Gamma_\mathfrak{c}$ be the maximal torsion-free subgroup of $Z(\mathfrak{c})$, and fix a (non-canonical) splitting
\[
Z(\mathfrak{c})\simeq\Delta_{\mathfrak{c}}\times\Gamma_{\mathfrak{c}}
\]
with $\Delta_{\mathfrak{c}}$ a finite group and $\Gamma_{\mathfrak{c}}\simeq\Z_p^2$. For $\mathfrak{c}'\supset\mathfrak{c}$ the natural projection $Z(\mathfrak{c})\twoheadrightarrow Z(\mathfrak{c}')$ takes $\Delta_{\mathfrak{c}}$ to $\Delta_{\mathfrak{c}'}$, inducing an isomorphism $\Gamma_{\mathfrak{c}}\xrightarrow{\sim}\Gamma_{\mathfrak{c}'}$. Thus in the following we shall identify $\Gamma_{\mathfrak{c}}$ with $\Gamma_\infty:=\Gamma_{(1)}$, the Galois group of the $\Z_p^2$-extension of $K$ as introduced in Section $\ref{subsec:CM}$. 

Suppose $\eta$ is a Hecke character of $K$ of conductor dividing $\mathfrak{c}p^\infty$. Viewing $\eta$ as a character on $Z(\mathfrak{c})\simeq\Delta_{\mathfrak{c}}\times\Gamma_K$, we put $\bar{\eta}:=\eta\vert_{\Delta_\mathfrak{c}}$, and denote by $\mathcal{L}_{\pp,\bar{\eta}}^{{\mathrm{K}},-}$ the image of $\mathcal{L}_{\pp,\mathfrak{c}}^{\mathrm{K}}$ under the composite map 
\begin{equation}\label{eqn:proj_hida_fam}  \cO\dBr{Z(\mathfrak{c})}\hat\otimes_{\Z_p}\Z_p^{\ur}\rightarrow\cO\dBr{\Gamma_K}\hat\otimes_{\Z_p}\Z_p^{\ur}\rightarrow\cO\dBr{\Gamma_K^{1-\cc}}\hat\otimes_{\Z_p}\Z_p^{\ur}\rightarrow\cO\dBr{\Gamma_\pp}\hat\otimes_{\Z_p}\Z_p^{\ur},
\end{equation}
where the first arrow is the natural projection defined by $\bar{\eta}$, the second arrow is defined by $\gamma\mapsto\gamma^{1-\cc}$ for $\gamma\in\Gamma_\infty$, and the last arrow is the isomorphism $\Gamma_K^{1-\cc}\simeq \Gamma_\pp$ via the map $\gamma^{1-\cc}\mapsto{(\gamma\vert_{K({\pp^\infty})})}^{1/2}$. We put also $\bar{\eta}^-:=\bar{\eta}^{\cc-1}$.

\begin{lem}\label{lem:HT}
Let $\ch$ be a ray class character of $K$ such that $\bar{\ch}^-$ has conductor $\mathfrak{c}$ prime-to-$p$. Assume that:
\begin{enumerate}
\item[(i)] $\mathfrak{c}$ is only divisible by primes that are split in $K$; 
\item[(ii)] $\Delta_{\mathfrak{c}}$ has order prime-to-$p$; 
\item[(iii)] $\bar{\xi}^-\vert_{G_{K_v}}\neq 1$ for all primes $v\mid p$ in $K$;
\item[(iv)] $\bar{\ch}^-$ has order at least $3$. 
\end{enumerate}
Then the congruence ideal of the CM Hida family $\boldsymbol{\theta}_{\ch}(S)$ is principal, generated by $\mathcal{L}_{\pp,\bar{\ch}^-}^{{\mathrm{K}},-}$.
\end{lem}

\begin{proof}
See Alonso--Castella--Rivero \cite[Prop.~4.6]{ACR2}. This is a consequence of the proof of the anticyclotomic Iwasawa main conjecture for Hecke characters by Hida--Tilouine \cite{HT-ENS,HT-117} and Hida \cite{hida-coates} (note that $p\nmid h_K$). 
\end{proof}

\begin{rem}
    Following Section \ref{subsec:CM}, we may view $\mathcal{L}_{\pp,\bar{\ch}^-}^{{\mathrm{K}},-}$ as an element of $\cO\dBr{S}\hat\otimes_{\Z_p}\Z_p^{\ur}$ after choosing a topological generator $\gamma_\pp$ of $\Gamma_\pp$.
\end{rem}

\subsubsection{Anticyclotomic Katz $p$-adic $L$-function for imaginary biquadratic fields}
For any ideal $\mathfrak{c}_0\subset\cO_{K_0}$ coprime to $p$, we let $Z(\mathfrak{c}_0)$ denote the ray class group of $K_0$ of conductor $\mathfrak{c}_0p^\infty$ (so $Z(\mathfrak{c}_0)\simeq{\mathrm{Gal}}(K_{0,\mathfrak{c}_0p^\infty}/K_0)$). Here, $K_0$ is our imaginary biquadratic field. 
\begin{thm}\label{thm:Katz-K_0}
Given a CM type $\Sigma$ of $K_0$, denote by $\Sigma_p$ the set of places of $K_0$ above p induced by $i_p \circ\sigma$ for $\sigma\in \Sigma$. There exists an element $\Lcal_{\Sigma_p,\mathfrak{c}_0}^{\mathrm{K}_0}
\in\cO\dBr{Z(\mathfrak{c}_0)}\hat\otimes_{\Z_p}\Z_p^{\ur}$ 
such that for every character $\lambda$ of $Z(\mathfrak{c}_0)$ that is crystalline at $\{\cP_1,\cP_2,\cP_3,\cP_4\}$, corresponding to a Hecke character of infinity type $k\Sigma+\kappa(1-\cc)$ with $k>0$ and $\kappa\in\Z_{\ge 0}[\Sigma]$, then we have
 \[\Lcal_{\Sigma_p,\mathfrak{c}_0}^{\mathrm{K}_0}(\lambda) = \frac{\Omega_p^{k\Sigma+2\kappa}}{\Omega_\infty^{k\Sigma+2\kappa}}\cdot \pi^{\kappa}\cdot L_{p\mathfrak{c}_0}(\lambda,0) \cdot \mathcal{E}_p(\lambda) \mathcal{E}_{\mathfrak{c}_0^+}(\lambda)  \frac{\Gamma_\Sigma(k\Sigma + \kappa)}{\sqrt{|D_{K_3}|}(\mathrm{Im}\,\vartheta)^\kappa } [\mathcal{O}_{K_0}^\times : \mathcal{O}_{K_3}^\times].\]
\end{thm}
\begin{proof}
We extract this version from Hsieh \cite[p.\,2]{hsieh-crelle-mu}.    
\end{proof}

\begin{rem}Note that we can view $\lambda$ as a character on $Z(\mathfrak{c_0})\simeq\Delta_{\mathfrak{c_0}}\times\Gamma_{K_0}$. Similar to equation (\ref{eqn:proj_hida_fam}), we have a corresponding composition for the biquadratic field $K_0$ via identifying the anticyclotomic $\Z_p^2$ extension $\Ga^{-}=\text{Gal}(K_{0,\infty}^{-}/K_0)$ of $K_0$ with $\Gamma_{\pp_1}^{K_1}\times \Gamma_{\pp_2}^{K_2}$ (by the diagram \ref{eq:G_K0^-identification}). The image of $\Lcal_{\Sigma_p,\mathfrak{c}_0}^{\mathrm{K}_0}$, denoted as $\Lcal_{\Sigma_p,\bar{\lambda}^-}^{\mathrm{K}_0,-}$, can be viewed as an element of $\cO\dBr{S_1,S_2}\hat\otimes_{\Z_p}\Z_p^{\ur}$, after choosing a topological generator $\gamma_{\pp_i}$ of $\Gamma_{\pp_i}$ for $i\in\{1,2\}$.
\end{rem}

\subsubsection{The factorization result}

We now fix our choice $\eta_{\bfg}$ of generator of the congruence ideal of $\bfg=\boldsymbol{\theta}_{\xi_1}(S_1)$. For $\xi_1$ satisfying the conditions of Lemma~\ref{lem:HT}, choose
\[
\eta_{\bfg}=\mathcal{L}_{\pp_1,\bar{\xi}_1^-}^{{\mathrm{K}_1},-},
\] inside Theorem \ref{thm:hsieh-triple}, and let $\mathscr{L}_p^{\bfg}(f,\bfg,\bfh)$ denote its image under the specialization by $Q_1$.

Note that $\ch_1$ can be replaced by a twist $\ch_1\cdot\phi\circ\mathbf{N}$ for a Dirichlet character $\phi$ without changing $\bar\ch_1^-$, and thus in the following we may assume that $\ch_1$ satisfies the following minimality hypotheses:
\begin{equation}\label{eq:minimal}
\textrm{the conductor of $\ch_1$ is minimal among Dirichlet twists.} 
\end{equation}

\begin{prop}\label{prop:factor-L-indef}
Assume that $\ch_1$ satisfies the conditions in Lemma~\ref{lem:HT}.  Denote 
\begin{align*}
W_1^{\psi} =  (1+p)^{k_1' - 1}\left(\frac{1+S_1}{1+p}\right)^{1/2}-1,\quad W_2^{\psi}= (1+p)^{1-k_1'}\left(\frac{1+S_2}{1+p}\right)^{1/2}-1,
\end{align*}
and 
\begin{align*}
W_1^{\chi} =  (1+p)^{1-k_1' }\left(\frac{1+S_1}{1+p}\right)^{1/2}-1,\quad W_2^{\chi}= (1+p)^{k_1'-1}\left(\frac{1+S_2}{1+p}\right)^{1/2}-1.
\end{align*}
Then
\begin{align*}
\bigl(\mathscr{L}_p^{\bfg}(f,\bfg,\bfh)(S_1,S_2)\bigr)^2
&=\mathbf{w}\cdot\mathcal{L}_{\cP_1,\cP_4,\bar{\psi}_0^-}^{{\mathrm{K}_0},-}(W_1^{\psi},W_2^{\psi})\cdot\mathcal{L}_{\cP_1,\cP_4,\bar{\chi}_0^-}^{{\mathrm{K}_0},-}(W_1^{\chi},W_2^{\chi}),
\end{align*}
where $\widetilde{\psi_1'\ch_1\ch_2}=\psi_0^{-1}$, $\widetilde{\psi_1'^{\cc}\ch_1\ch_2}=\chi_0^{-1}$ and $\mathbf{w}$ is a unit in $\cO\dBr{S_1,S_2}\otimes_{\Z_p}\bQ_p$.
\end{prop}

\begin{proof}
Let $k_1, k_2$ be integers with $k_1\equiv k_2\pmod{2}$ and $k_1\geq k_2+k_1'$. Recall that $T_i=\frac{(1+S_i)}{1+p}-1$ for $i=1,2$. Set $T_i=(1+p)^{k_i-1}-1$ for $i=1,2$,
and denote by $\VQdag$ the specialization of $\Vdag$ at $\underline{Q}=(Q_1,T_1,T_2)$. By restricting equation (\ref{eq:dec-V-first}) to $G_{K_0}$
\begin{equation}\label{eq:dec-V}
\begin{aligned}    \Vsdag|_{K_0}\simeq\cO(\tilde{\psi}_1'^{-1})(1-k_1'/2)\otimes\tilde{\ch}_1^{-1}\tilde{\Psi}_{T_1}\tilde{\ch}_2^{-1}\tilde{\Psi}_{T_2}\oplus \cO(\tilde{\psi}_1'^{-\cc})(1-k_1'/2)\otimes\tilde{\ch}_1^{-1}\tilde{\Psi}_{T_1}\tilde{\ch}_2^{-1}\tilde{\Psi}_{T_2}
\end{aligned}
\end{equation}
 we find that the completed $L$-value appearing in the interpolation formula of Theorem~\ref{thm:hsieh-triple} is given by
\begin{equation}\label{eq:explicit}
\begin{aligned}
\Gamma_{\VQdag}(0)\cdot L(\VQdag,0)&=\frac{\Gamma\bigl(\frac{k_1'+k_1+k_2}{2}-1\bigr)\Gamma\bigl(\frac{k_1-k_2-k_1'}{2}+1\bigr)\Gamma\bigl(\frac{k_1+k_2-k_1'}{2}\bigr)\Gamma\bigl(\frac{k_1-k_2+k_1'}{2}\bigr)}{2^4\cdot(2\pi)^{2k_1}}\times L(\psi,0)\cdot L(\chi,0),
\end{aligned}
\end{equation}
where
\begin{equation*}
   \psi= \tilde{\psi'}_1^{-1}\bN^{1-k_1'/2}\tilde{\ch}_1^{-1}\tilde{\ch}_2^{-1} \tilde{\Psi}_{T_1}\tilde{\Psi}_{T_2},\qquad \chi= \tilde{\psi'}^{-\cc}_1\bN^{1-k_1'/2}\tilde{\ch}_1^{-1}\tilde{\ch}_2^{-1} \tilde{\Psi}_{T_1}\tilde{\Psi}_{T_2}
\end{equation*}
Note that the modified Euler factor decomposes at $p$ as
\begin{equation}\label{eq:Ep}
\begin{aligned}
\mathcal{E}_p(\mathscr{F}_p^{\bfg}(\VQdag))
&=\mathcal{E}_p(\psi)\times \mathcal{E}_p(\chi)
\end{aligned}
\end{equation}
and the $\Gamma$-factor as well.
Moreover, letting $\chi_g'$ be the prime-to-$p$ part of the nebentypus character of $\bfg_{T_1}$, we have
\begin{align*}
\biggl(1-\frac{\chi_g'(p)p^{k_1-1}}{\ch_1\Psi^{-1}_{T_1}(\bar{\pp})^2}\biggr)\biggl(1-\frac{\chi_g'(p)p^{k_1-2}}{\ch_1\Psi^{-1}_{T_1}(\bar{\pp})^2}\biggr)
&=\bigl(1-\ch_1^{\cc-1}\Psi_{T_1}^{1-\cc}(\ppbar)\bigr)\bigl(1-\ch_1^{\cc-1}\Psi^{1-\cc}_{T_1}(\ppbar)p^{-1}\bigr),
\end{align*}
and therefore the canonical period $\Omega_{\bfg_{T_1}}$ in Theorem~\ref{thm:hsieh-triple} (associated with the generator $\eta_{\bfg}=\mathcal{L}_{\pp_1,\bar{\xi}_1^-}^{{\mathrm{K}_1},-}$ of $C(\bfg)$ from Lemma~\ref{lem:HT}) is given by
\begin{equation}\label{eq:period}
\Omega_{\bfg_{T_1}}=(-2\sqrt{-1})^{k_1+1}\cdot\frac{\Vert\bfg_{T_1}^\circ\Vert_{\Gamma_0(C)}^2}{\eta_{\bfg_{T_1}}^{\star}}\cdot\bigl(1-\ch_1^{\cc-1}\Psi_{T_1}^{1-\cc}(\ppbar_1)\bigr)\bigl(1-\ch_1^{\cc-1}\Psi^{1-\cc}_{T_1}(\ppbar_1)p^{-1}\bigr),
\end{equation}
where $C=N_{K/\Q}(\fkf_1)D_K$, and we note that we may ignore the term $\prod_{q\in\Sigma_{\mathrm{exc}}}(1+q^{-1})^2$ from Theorem~\ref{thm:hsieh-triple}, since (up to a $p$-adic unit) it only contributes a fixed power of $p$ (in particular, independent of $k_1,k_2$).
%

On the other hand, since $g_{T_1}$ has weight $k_1$, from Hida's formula for the adjoint $L$-value \cite[Thm.~7.1]{HT-ENS} (using that $\xi_1$ satisfies the minimality condition (\ref{eq:minimal})) and Dirichlet's class number formula we obtain 
\[
\Vert \bfg_{T_1}^\circ\Vert^2_{\Gamma_0(C)}=\Gamma(k_1)\cdot\frac{D_{K_1}^2}{2^{2k_1}\pi^{k_1+1}}\cdot\frac{2\pi h_{K_1}}{\vert\cO_{K_1}^\times\vert\sqrt{D_{K_1}}}\cdot L(\ch_1^{\cc-1}\Psi_{T_1}^{1-\cc},1).
\]

Note that $L(\ch_1^{\cc-1}\Psi_{T_1}^{1-\cc},1)=L(\ch_1^{\cc-1}\Psi_{T_1}^{1-\cc}\mathbf{N}^{-1},0)$, and $\ch_1^{\cc-1}\Psi_{T_1}^{1-\cc}\mathbf{N}^{-1}$ has infinity type $(k_1,2-k_1)$. Hence for $k_1\geq 2$ this character lies in the range of interpolation of $\cL_{\pp_1,\mathfrak{f}_1}^{\mathrm{K}_1}$, where $\mathfrak{f}_1$ denotes the conductor of $\xi_1^{1-\cc}$, and from the right above formula for $L(\ch_1^{\cc-1}\Psi_{T_1}^{1-\cc},1)$ and Theorem~\ref{thm:katz} we obtain
\begin{equation}\label{eq:hida-formula}
\begin{aligned}
\cL_{\pp_1,\mathfrak{f}_1}^{\mathrm{K}_1}(\ch_1^{\cc-1}\Psi_{T_1}^{1-\cc}\mathbf{N}^{-1})&
=\biggl(\frac{\Omega_p}{\Omega_\infty}\biggr)^{2k_1-2}\cdot\frac{\pi^{2k_1-2}\cdot 2^{3k_1-3}}{\sqrt{D_{K_1}}^{k_1+1}}\\
&\quad\times\bigl(1-\ch_1^{\cc-1}\Psi^{1-\cc}_{T_1}(\ppbar_1)\bigr)\bigl(1-\ch_1^{\cc-1}\Psi^{1-\cc}_{T_1}(\ppbar_1)p^{-1}\bigr)\cdot\Vert\bfg^\circ_{T_1}\Vert_{\Gamma_0(C)}^2\cdot\frac{\vert\cO_{K_1}^\times\vert}{h_{K_1}}.
\end{aligned}
\end{equation}
Moreover, by the functional equation of Theorem~\ref{thm:katz} 
and the definition of $\eta_{\bfg}$ we have the relation 
\begin{equation}\label{eq:Katz-and-congruence}
    \frac{h_{K_1}}{\vert\cO_{K_1}^\times\vert}\cdot\Lcal_{\pp_1,\mathfrak{f}_1}^{\mathrm{K}_1}(\ch_1^{\cc-1}\Psi_{T_1}^{1-\cc}\mathbf{N}^{-1})\sim_p\eta_{\bfg_{T_1}}^{\star},
\end{equation}
where $\sim_p$ denotes equality up to a $p$-adic unit. Therefore, equations \eqref{eq:hida-formula} and \eqref{eq:Katz-and-congruence} imply that
\begin{equation}\label{eq:subs-for-hida-formula}
\frac{\Vert\bfg_{T_1}^\circ\Vert_{\Gamma_0(C)}^2}{\eta_{\bfg_{T_1}}^{\star}}\cdot\bigl(1-\ch_1^{\cc-1}\Psi_{T_1}^{1-\cc}(\ppbar_1)\bigr)\bigl(1-\ch_1^{\cc-1}\Psi^{1-\cc}_{T_1}(\ppbar_1)p^{-1}\bigr)\sim_p  \biggl(\frac{\Omega_\infty}{\Omega_p}\biggr)^{2k_1-2}\cdot\frac{\sqrt{D_{K_1}}^{k_1+1}}{(2\pi)^{2k_1-2}}   
\end{equation}

Hence, from (\ref{eq:period}) and (\ref{eq:subs-for-hida-formula}) we arrive at
\begin{equation}\label{eq:Omega-bis}
\frac{1}{\Omega_{\bfg_{T_1}}^2}\sim_p\biggl(\frac{\Omega_p}{\Omega_\infty}\biggr)^{4k_1-4}\times\pi^{4k_1-4}.
\end{equation}

Finally, note that we will use the $p$-adic $L$-function over $K_0$ with CM type $\Sigma$ corresponding with $\Sigma_p=(\cP_1,\cP_4)$ via $i_p$, which implies that both $\psi$ and $\chi$ are in the range of interpolation for $\mathcal{L}_{\cP_1,\cP_4,\bar{\psi}_0^-}^{{\mathrm{K}_0},-}$ and $\mathcal{L}_{\cP_1,\cP_4,\bar{\chi}_0^-}^{{\mathrm{K}_0},-}$, respectively. Extracting from Theorem \ref{thm:Katz-K_0}, the transcendental part for $\mathcal{L}_{\cP_1,\cP_4,\bar{\psi}_0^-}^{{\mathrm{K}_0},-}$ is
\begin{equation}\label{eqn:transcendental-part-psi}
    \pi^{k_1'+k_1-3}\left(\frac{\Omega_p}{\Omega_\infty}\right)^{2k_1'+2k_1-4}
\end{equation}
and the transcendental part for $\mathcal{L}_{\cP_1,\cP_4,\bar{\chi}_0^-}^{{\mathrm{K}_0},-}$ is
\begin{equation}
    \label{eqn:transcendental-part-chi}\pi^{k_1-k_1'-1}\left(\frac{\Omega_p}{\Omega_\infty}\right)^{2k_1-2k_1'}.
\end{equation}
Here, we use the fact that the $p$-adic and infinity periods of the $p$-adic $L$-function over $K_0$ agree with the ones of the $p$-adic $L$-function over $K_1$ by defining the CM abelian surface $A$ over the biquadratic field as $E\times E$ (since $A$ has induced CM type), where $E$ is the CM elliptic curve over the imaginary quadratic field, which respects the decomposition $\cP_1\cP_4=\mathfrak{p}_1$. In particular
\[\Omega_{\infty,\pp_1}= \Omega_{\infty,\cP_1},\qquad \Omega_{p,\pp_1}= \Omega_{p,\cP_1}, \qquad \Omega_{\infty,\pp_1}= \Omega_{\infty,\cP_4},\qquad \Omega_{p,\pp_1}= \Omega_{p,\cP_4},\]
see also Hida \cite{Hida_2007}.

Thus substituting (\ref{eq:explicit}), (\ref{eq:Ep}), (\ref{eq:Omega-bis}), (\ref{eqn:transcendental-part-psi}), and (\ref{eqn:transcendental-part-chi}) into the interpolation formula for $\mathscr{L}_p^\bfg(f,\bfg,\bfh)$ in Theorem~\ref{thm:hsieh-triple} and we finally arrive at
\begin{align*}
\bigl(\mathscr{L}_p^{\bfg}(f,\bfg,\bfh)(S_1,S_2)\bigr)^2
=\mathcal{L}_{\cP_1,\cP_4,\bar{\psi}_0^-}^{{\mathrm{K}_0},-}(W_1^{\psi},W_2^{\psi})\cdot\mathcal{L}_{\cP_1,\cP_4,\bar{\chi}_0^-}^{{\mathrm{K}_0},-}(W_1^{\chi},W_2^{\chi}),
\end{align*}
up to a unit in $\cO\dBr{S_1,S_2}\otimes_{\Z_p}\bQ_p$, which yields the result.
\end{proof}

\subsection{On the Iwasawa main conjecture}
\label{subsec:IMC-def}

Our second application is the Iwasawa main conjecture for self-dual Hecke character over $K_0$ by combining with the main result of \cite{Hsieh-JAMS}. The strategy originates from proofs of Theorems $6.7$ and $7.12$ in \cite{CD}. Here, we use the formulation of Greenberg \cite{Greenberg55} in the context of Galois representations.

\begin{assumption}\label{Assumption:Hsieh}
    We assume the following two conditions, stemming from the main theorem $2$ of \cite{Hsieh-JAMS}:
    \begin{enumerate}
        \item $p\nmid D_{K_3}\#(\Delta_{\mathfrak{c}_0}),$
        \item $\chi$ is unramified at $\cP_2, \cP_3$, and $\chi\omega_{K_0}^{-a}$ is unramified at $\cP_1, \cP_4$ for some $a\not\equiv 2\pmod{p-1}$.
    \end{enumerate}
\end{assumption}

\begin{thm}\label{thm:IMC-def} Assume that the sign of $\chi$ is $+1$. Under the same assumption as in Theorem~\ref{thm:BK-def} together with Assumption \ref{Assumption:Hsieh}, we have the following equality inside $\Lambda_{K_0}\otimes_{\Z_p}\Q_p$
\[
{\rm char}_{\Lambda_{K_0}}(X_{0,\varnothing,\varnothing,0}(K_{0,\infty},\cO_{\chi^{-1}}(1)))= (\mathcal{L}_{\cP_1,\cP_4,\chi}^{{\mathrm{K}_0}}).
\]

\end{thm}

\begin{proof}
Following the proof of Theorem \ref{thm:BK-def}, we have 
\[\mathcal{L}_{\pp_1,\bar{\xi}_1^-}^{{\mathrm{K}_1},-}\cdot{\mathrm{Log}}^{\bfg}_{(\breve{\bff}^\star,\breve{\bfg}^\star,\breve{\bfh}^\star)}({\rm res}_{\mathcal{P}_3}(^{\cc}\mathbf{z}_{\chi}))=\mathscr{L}_p^{\bfg}(\bff,\bfg,\bfh),\]
by the reciprocity law \ref{thm:ERL}. By the decomposition \ref{prop:factor-L-indef}, the RHS is related to 
 \[\mathcal{L}_{\cP_1,\cP_4,\bar{\psi}_0^-}^{{\mathrm{K}_0},-}(W_1^{\psi},W_2^{\psi})\cdot\mathcal{L}_{\cP_1,\cP_4,\bar{\chi}_0^-}^{{\mathrm{K}_0},-}(W_1^{\chi},W_2^{\chi}).\]
  We can arrange the first factor to be a unit inside $\cO\dBr{S_1,S_2}\hat\otimes_{\Z_p}\Z_p^{\ur}$ by Theorem $1.4$ of \cite{BHTY} (see also \cite{BHKO}) by changing $(\psi_1,\psi_1')\mapsto(\psi_1\phi,\psi_1'\phi)$ as inside the proof of Theorem \ref{thm:BK-def}, see also extensive works of Hida \cite{Hida-annals-mu,Hida_2007,hida-dwork,hida-fix}, and \cite{hsieh-crelle-mu}. Theorem $1.1$ of \cite{BHTY} also implies that the second factor is non-zero in the Iwasawa algebra by the interpolation formula, which then shows that the base class $^{\cc}\mathbf{z}_{\chi}$ of our Euler system is non-torsion. By Theorem \ref{thm:IMC-general}, we have
\[
{\rm char}_{\Lambda_{K_0}^-}(X_{0,\varnothing,0,\varnothing}(K_{0,\infty}^-,\cO_{\chi^{-1}}(1))_{\rm tors})\supset{\rm char}_{\Lambda_{K_0}^-}\biggl(\frac{{\rm Sel}_{\varnothing,0,\varnothing,0}(K_0[p^\infty],\cO_{\chi})}{\Lambda_{K_0}^-\cdot \,^{\cc}\mathbf{z}_{\chi}}\biggr)^2
\]
in $\Lambda_{K_0}^-$. Via the explicit reciprocity law and global duality, we obtain that $X_{0,\varnothing,\varnothing,0}(K_{0,\infty}^-,\cO_{\chi^{-1}}(1))$ is torsion together with
\[
{\rm char}_{\Lambda_{K_0}^-}(X_{0,\varnothing,\varnothing,0}(K_{0,\infty}^-,\cO_{\chi^{-1}}(1)))\supset (\mathcal{L}_{\cP_1,\cP_4,\mathfrak{c}_0}^{{\mathrm{K}_0},-})
\]
in $\Lambda_{K_0}^-$, which is one divisibility of the anticyclotomic Iwasawa Main Conjecture for the character $\chi$.

Following the proof of \cite[Thm.\, 8.17]{Hsieh-JAMS}, replacing the Hida's solution to the anticyclotomic
main conjecture in \cite{hida-coates} by our divisibility, we obtain the equality over the full $\Lambda_{K_0}$.\qedhere

\end{proof}

\subsection{On the Bloch--Kato conjecture in rank at least one}

Our last application is extracted from the proof of Theorem~\ref{thm:IMC-def}. It provides a result towards the Bloch--Kato conjecture in rank at least $1$.

\begin{thm}\label{thm:BK-def-1}
Let $\chi$ be an anticyclotomic Hecke character of $K_0$ as in Assumption (\ref{eq:chi-decomp}). Assume further that:
\begin{enumerate}
    \item $\prod_{p|N}\epsilon_p=1$,
    \item $p>k_1'$,
    \item $p\nmid 6h_{K_0}$, the class number of $K_0$,
   \item the sign of the functional equation of $\chi$ being $+1$.
\end{enumerate}
If the triple $(k_1',k_1,k_2)$ forms $3$ sides of a triangle (which induces $L(V_\chi,0)=0$) then 
\[
{\rm dim}_{L_\mathfrak{P}}\,{\rm Sel}_{\rm BK}(K_0,V_{\chi})\geq 1.
\]

\end{thm}

\begin{proof}
 The class $^{\cc}\mathbf{z}_{\chi}\in{\rm Sel}_{\varnothing,0,\varnothing,0}(K_0[p^\infty],\cO_{\chi})$ is non-torsion via the proof of Theorem~\ref{thm:IMC-def}. Moreover, $^{\cc}\mathbf{z}_{\chi}$ is the base class of a $\Lambda_{K_0}^-$-adic anticyclotomic Euler system. Then Theorem~\ref{thm:IMC-general} implies that ${\rm Sel}_{\varnothing,0,\varnothing,0}(K_0[p^\infty],\cO_{\chi})$ has $\Lambda_{K_0}^-$-rank $1$. The natural map (compare with the projection (\ref{eq:Sel_projection_from_K_infty}))
\begin{equation}\label{eq:proj}
{\rm Sel}_{\varnothing,0,\varnothing,0}(K_0[p^\infty],\cO_{\chi})/(\gamma_{1,-}-1,\gamma_{2,-}-1)\rightarrow{\rm Sel}_{\varnothing,0,\varnothing,0}(K_0,\cO_{\chi})
\end{equation}
is injective (see also \cite[Prop.\,5.3.14]{MR-KS} and Greenberg \cite[p.\,72]{greenberg-cetraro}). 

Therefore, the Selmer group ${\rm Sel}_{\varnothing,0,\varnothing,0}(K_0,\cO_{\chi})$ has a positive $\cO$-rank. The theorem then follows by noting that ${\rm Sel}_{\varnothing,0,\varnothing,0}(K_0,\cO_{\chi})={\rm Sel}_{\rm BK}(K_0,\cO_{\chi})$ inside the balanced range by Lemma \ref{lem:BK-Gr}.
\end{proof}

\begin{rmk}\label{rmk:BK_rank_1}
Note that by letting $^{\cc}z_{\chi}\in{\rm Sel}_{\varnothing,0,\varnothing,0}(K_0,\cO_{\chi})$ be the image of $^{\cc}\mathbf{z}_{\chi}$ under the map (\ref{eq:proj}), such a class $^{\cc}z_{\chi}\in{\rm Sel}_{\rm BK}(K_0,V_{\chi})$ satisfies:
\[
^{\cc}z_{\chi} \text{ is non-torsion}\quad\Longrightarrow\quad{\rm dim}_{L_\mathfrak{P}}\,{\rm Sel}_{\rm BK}(K_0,V_{\chi})=1
\]
by Theorem~\ref{thm:rank-1-general}, under the same assumption in Theorem \ref{thm:BK-def-1}.
\end{rmk}

\begin{lem}For $k_1\ge k_1'+k_2$, if $\mathrm{res}_{\cP_4}(\mathrm{Sel}_{\mathrm{BK}}(K_0,V_\chi))\neq 0$ then
      \[\dim_{L_\mathfrak{P}} (\mathrm{Sel}_{\mathrm{BK}}(K_0,V_\chi))=\dim_{L_\mathfrak{P}}(\mathrm{Sel}_{\varnothing,0,\varnothing,0}(K_0,V_\chi))+1.\]

\end{lem}
\begin{proof}Recall that $\mathrm{Sel}_{\mathrm{BK}}(K_0,V_\chi)=\mathrm{Sel}_{\varnothing,0,0,\varnothing}(K_0,V_\chi)$ in this $\bfg$-unbalanced range by Lemma \ref{lem:BK-Gr}. Using the fact that $\chi$ is conjugate self-dual, we see that $\mathrm{Sel}_{\mathrm{BK}}(K_0,V_\chi^*)=\mathrm{Sel}_{0,\varnothing,\varnothing,0}(K_0,V_\chi^*)$. Furthermore,  $\mathrm{res}_{\cP_3}(\mathrm{Sel}_{\mathrm{BK}}(K_0,V_\chi^*))\neq 0$ by the assumption that  $\mathrm{res}_{\cP_4}(\mathrm{Sel}_{\mathrm{BK}}(K_0,V_\chi))\neq  0$. Let $x\in \mathrm{Sel}_{\mathrm{BK}}(K_0,V_\chi^*)$ such that $\mathrm{res}_{\cP_3}(x)\neq 0$. Take any $\kappa\in \mathrm{Sel}_{\varnothing,0,\varnothing,0}(K_0,V_\chi) $. By \cite[Thm\,2.3.4]{MR-KS}, we have 
        \[\sum_{i=1}^4\langle \mathrm{res}_{\cP_i} (\kappa),\mathrm{res}_{\cP_i}(x)\rangle=0 \qquad\Rightarrow \qquad  \langle \mathrm{res}_{\cP_3} (\kappa),\mathrm{res}_{\cP_3}(x)\rangle=0 \qquad\Rightarrow \qquad \mathrm{res}_{\cP_3} (\kappa)=0,\]
        i.e. $\mathrm{Sel}_{\varnothing,0,\varnothing,0}(K_0,V_\chi)=\mathrm{Sel}_{\varnothing,0,0,0}(K_0,V_\chi)$. Now using the exact sequence
        \[0\longrightarrow \mathrm{Sel}_{\varnothing,0,0,0}(K_0,V_\chi)\longrightarrow \mathrm{Sel}_{\mathrm{BK}}(K_0,V_\chi)\xrightarrow{{\rm res}_{\mathcal{P}_4}} \rH^1(K_{0,\mathcal{P}_4},V_\chi),\]
        and the assumption that the last restriction map is non-trivial, we conclude the result.\qedhere
   
\end{proof}
Combining with Theorem \ref{thm:rank-1-general}, we obtain the following rank two result:
\begin{cor} \label{cor:rank-2-preparation}
For $k_1\ge k_1'+k_2$, if $\mathrm{res}_{\cP_4}(\mathrm{Sel}_{\mathrm{BK}}(K_0,V_\chi))\neq 0$ and $^{\cc}z_{\chi}$ is non-torsion then
      \[\dim_{L_\mathfrak{P}} (\mathrm{Sel}_{\mathrm{BK}}(K_0,V_\chi))=2.\]

\end{cor}

\begin{thm}\label{thm:rank-2}

For $k_2=1$, $k_1'$ even and $k_1=k_1'+1$, we assume that
\begin{enumerate}
    \item $\chi^{-1}$, having infinity type $(1,0,0,1)$ is attached to a CM abelian surface,
    \item $L(\chi,0)=0$ with sign of the functional equation being $+1$ (hence vanishing to an even order) 
\end{enumerate}
then we have:
      \[ ^{\cc}z_{\chi}\text{ is non-torsion}\qquad\Rightarrow\qquad\dim_{L_\mathfrak{P}} (\mathrm{Sel}_{\mathrm{BK}}(K_0,V_\chi))=2.\]
\end{thm}

\begin{proof}
By Theorem \ref{thm:rank-1-general}, which is the expected outcome after feeding our Euler system to the machinery of \cite{JNS}, we know that ${\rm Sel}_{\varnothing,0,\varnothing,0}(K_0,V_{\chi})$ is one-dimensional since $^{\cc}z_{\chi}$ is non-torsion. From the vanishing of $L(\chi,0)$, we conclude that 
${\rm res}_{\mathcal{P}_3}(^{\cc}z_{\chi})= 0$ by the reciprocity law (see the proof of Theorem \ref{thm:BK-def}). Therefore, we have 
\[{\rm Sel}_{\varnothing,0,\varnothing,0}(K_0,V_{\chi})={\rm Sel}_{\varnothing,0,0,0}(K_0,V_{\chi}).\]
By Lemma \ref{lem:BK-Gr}, $\mathrm{Sel}_{\mathrm{BK}}(K_0,V_\chi)=\mathrm{Sel}_{\varnothing,0,0,\varnothing}(K_0,V_\chi)$ as $k_1=k_1'+k_2$. Hence, we know  the dimension of $\mathrm{Sel}_{\mathrm{BK}}(K_0,V_\chi)$ is at least $1$. Now, we have two cases, either $\mathrm{res}_{\cP_4}(\mathrm{Sel}_{\mathrm{BK}}(K_0,V_\chi))$ is trivial or not. In the first case, this leads to $\mathrm{Sel}_{\mathrm{BK}}(K_0,V_\chi)={\rm Sel}_{\varnothing,0,0,0}(K_0,V_{\chi})$, which is one dimensional. But that number has to be even due to the recent result of Bell \cite{rank-CM-abelian-variety}, which is a contradiction. Therefore, $\mathrm{res}_{\cP_4}(\mathrm{Sel}_{\mathrm{BK}}(K_0,V_\chi))$ is non-trivial and we conclude the result using Corollary \ref{cor:rank-2-preparation}.
\end{proof}

   \newpage

\bibliographystyle{amsalpha}
\bibliography{Schoen-cm}

\end{document}